\documentclass[10pt,twocolumn,twoside]{IEEEtran}
\usepackage{etex}
\usepackage[letterpaper,top=0.75in,left=0.75in,right=0.75in,bottom=0.75in]{geometry}
\reserveinserts{28}

\usepackage{amsmath, amsfonts, amssymb, mathtools, url, graphicx, balance, comment, caption}

\IEEEoverridecommandlockouts

 \newcommand{\subparagraph}{}
\usepackage[compact]{titlesec}

\newif\ifOneColumn
\OneColumnfalse

\newif\ifArxivVersion
\ArxivVersiontrue

\newif\ifNotArxivVersion

\ifArxivVersion
\def \myEquationStyle {\displaystyle}
\renewcommand{\baselinestretch}{1}
\newcommand{\myFrac}[2]{\dfrac{#1}{#2}}

\NotArxivVersionfalse
\else
\def \myEquationStyle {\textstyle}
\renewcommand{\baselinestretch}{0.95}
\newcommand{\myFrac}[2]{\tfrac{#1}{#2}}

\def \sspace{-2pt}
\def \ssspace{-2pt}
 \titlespacing{\section}{2pt}{-\sspace}{-\ssspace}
 \titlespacing{\subsection}{2pt}{-\sspace}{-\ssspace}
 \titlespacing{\subsubsection}{2pt}{-\sspace}{-\ssspace}
\usepackage[nodisplayskipstretch]{setspace}

\NotArxivVersiontrue
\fi

\newif\ifLemLpfNpfRelationship
\LemLpfNpfRelationshipfalse

\newif\ifpropGreedyOptimality
\propGreedyOptimalityfalse

\newif\ifappendixProofs
\appendixProofsfalse

\newif\ifgeneralApproachText
\generalApproachTexttrue

\newif\ifparameterTable
\parameterTablefalse

\newif\ifdetailedProofTheoremTwo
\detailedProofTheoremTwofalse

\newif\ifpropDownstream
\propDownstreamtrue

\newif\ifpropOptimalAttackVectors
\propOptimalAttackVectorsfalse

\newif\iflemOptimalAttackSetIndependentOfGamma
\lemOptimalAttackSetIndependentOfGammatrue

\newif\ifoldProofPropTypeThreeSecure 
\oldProofPropTypeThreeSecurefalse

\newif\ifsmallNotationTable
\smallNotationTabletrue

\usepackage[sort,compress]{cite}
\usepackage[normalem]{ulem}
\usepackage{cases}
\usepackage[usenames, dvipsnames]{color}
\usepackage[colorlinks=true,
            raiselinks=true,
            linkcolor=MidnightBlue,
            citecolor=Mahogany,
            urlcolor=ForestGreen,
            pdfauthor=Devendra Shelar and Saurabh Amin,
            pdftitle={Security Assessment of Electricity Distribution Networks under DER Node Compromises},
            pdfkeywords={DN Security},
            pdfsubject={Initial Manuscript},
            plainpages=false
            ]{hyperref}
\usepackage{mathrsfs,upgreek,wrapfig}

\usepackage{amsthm}
\usepackage[nameinlink]{cleveref}
\crefname{section}{\mytextsection}{\mytextsection\mytextsection}
\Crefname{section}{Sec.}{§§}
\crefname{subsection}{\mytextsection}{\mytextsection\mytextsection}

\newtheoremstyle{theoremdd}
{5pt}
{5pt}
{\itshape}
{0pt}
{\bfseries}
{.}
{ }
{\thmname{#1}\thmnumber{ #2}\thmnote{ (#3)}}

\theoremstyle{theoremdd}

\newtheorem{theorem}{Theorem}
\newtheorem{lemma}{Lemma}
\newtheorem{claim}{Claim}

\newtheorem{proposition}{Proposition}

\newtheorem*{formulation*}{}

\newtheorem{remark}{Remark}
\newtheorem{example}{Example}

\crefname{figure}{Fig.}{Figures}
\crefname{theorem}{Theorem}{Theorems}
\crefname{proposition}{Proposition}{Propositions}
\crefname{lemma}{Lemma}{Lemmas}
\crefname{algorithm}{Algorithm}{Algorithms}

\makeatletter
\def\thm@space@setup{%
  \thm@preskip=0.2cm plus 0.2cm minus 0.2cm
  \thm@postskip=\thm@preskip 
}
\renewenvironment{proof}[1][\proofname]{\par
  \pushQED{\qed}%
  \normalfont
  \topsep0pt \partopsep5pt 
  \trivlist
  \item[\hskip\labelsep
        \itshape
    #1\@addpunct{.}]\ignorespaces
}{%
  \popQED\endtrivlist\@endpefalse
  \addvspace{6pt plus 6pt} 
}
\makeatother

\makeatletter

\newenvironment{assumption}[1][]{\setcounter{subassmcounter}{-1}\hspace{-0.2cm}\vspace{-0.15cm}\refstepcounter{assmcounter} 
\par\medskip
   \hspace{-0.4cm}\textbf{(A\theassmcounter #1)} \rmfamily}{
   }

\newenvironment{subassumption}[1][]{\refstepcounter{subassmcounter}
\def \parentCounter  {\theassmcounter}
\par\textbf{$(\textbf{A\parentCounter})_\mathbf{\thesubassmcounter}$} \rmfamily}{}

\newcommand{\assmref}[1]{\textbf{\textcolor{MidnightBlue}{(A\ref{#1})}}}
\newcommand{\assm}[1]{\textbf{\textcolor{MidnightBlue}{(A#1)}}}

\newcommand{\subassmref}[1]{\textbf{\textcolor{MidnightBlue}{$(\textbf{A0})_{\mathbf{\ref{#1}}}$}}}

\newtheoremstyle{named}{}{}{\itshape}{}{\bfseries}{.}{.5em}{\thmnote{#3's }#1}
\theoremstyle{named}

\newtheoremstyle{mynamed}{}{}{\itshape}{}{\bfseries}{.}{.5em}{#1 \thmnote{A }}
\theoremstyle{named}

\newcommand{\abs}[1]{\left\lvert{#1}\right\rvert}

\newcommand{\conj}[1]{\bar{#1}}

\allowdisplaybreaks
\usepackage[table,x11names]{xcolor}

\usepackage{pgfplots,subfig,caption,fancybox,tikz,multirow,hhline,mathtools,chngcntr,pgfplots,booktabs,algpseudocode,mathabx,floatrow,algorithm,algorithmicx,mathrsfs,etex,etoolbox,xfrac}
\usetikzlibrary{patterns,arrows,shapes,snakes,automata,backgrounds,petri}
\pgfplotsset{compat=newest}
\counterwithout{equation}{section} 
\counterwithout{figure}{section} 

\algrenewcommand{\algorithmiccomment}[1]{\hskip3em$\slash\slash$ #1}

\newcommand{\lzeronorm}[1]{\left\lVert#1\right\rVert_0}
\newcommand{\lonenorm}[1]{\left\lVert#1\right\rVert_1}
\newcommand{\linfinityNorm}[1]{\left\lVert#1\right\rVert_\infty}

\DeclareMathOperator{\arccot}{arccot}
\def \emult {\small\odot}
\def \unity {\mathbf{1}_\NN}

\def \N {\mathcal{N}}
\def \E {\mathcal{E}}
\def \G {\mathcal{G}}
\def \P {\mathcal{P}}
\def \O {\mathcal{O}}
\def \Pmax {H}

\def \arcm {\mathrm{M}}

\def \R {\mathbb{R}}
\def \Z {\mathbb{Z}}
\def \C {\mathbb{C}}
\def \S {\mathcal{S}}
\def \D {\mathcal{D}}
\def \Da {\mathcal{D}_{\arcm}}
\def \Dd {\mathcal{U}_B}

\def \AD {\Dd}

\def \Re {\mathbf{Re}}
\def \Im {\mathbf{Im}}
\def \x {\mathrm{x}}
\def \j {\mathbf{j}}

\def \npf {\text{NPF}}
\def \lpf {\text{LPF}}
\def \cpf {\text{CPF}}

\def \NN {N}
\def \Nleaf {\N_{L}}

\def \Ns {\N_s}
\def \Nv {\N_v}

\def \pc { pc}
\def \qc { qc}

\def \pcdem { \mathrm{pc}^\mathrm{nom}}
\def \qcdem { \mathrm{qc}^\mathrm{nom}}
\def \scdem { \mathrm{sc}^\mathrm{nom}}

\def \W {W}
\def \c {C}

\def \pg { pg}
\def \qg { qg}

\def \sgset  {\mathrm{sp}}
\def \sgseta {\sgset^\mathrm{a}}
\def \sgsetd {\sgset^\mathrm{d}}

\def \sgmax {\overline{ \sgset}}

\def \sc { sc}
\def \sg { sg}
\def \sn { s}

\def \Vmax {\overline{V}}
\def \Vmin {\underline{V}}

\def \nuMax {\overline{\nu}}
\def \nuMin {\underline{\nu}}

\def \muMax {\overline{\mu}}
\def \muMin {\underline{\mu}}

\def \y { y}

\def \gammaMin {\underline{\gamma}}

\def \k {K}
\def \kmin {\underline{K}}
\def \kmax {\overline{K}}

\def \slr {\epsilon_0}

\def \slrrr {\epsilon}

\def \slovr {{\footnotesize\mathrm{VR}}}
\def \svoll {{\footnotesize\mathrm{LC}}}
\def \sll {{\footnotesize\mathrm{LL}}}

\def \Loss {\mathrm{L}}

\def \Losslpf {\widehat{\mathrm{L}}}
\def \Lossupf {\widecheck{\mathrm{L}}}

\def \llovr {\Loss_\slovr}
\def \lvoll {\Loss_\svoll}
\def \lloss {\Loss_\sll}
\def \lsupply {\Loss_\mathrm{S}}
\def \ldamage {\Loss_\mathrm{D}}
\def \lcurt {\Loss_\mathrm{AC}}

\def \rr {R}
\def \xx {X}
\def \zz {Z}
\def \nl {\Pmax}

\def \child {\mathcal{N}^c}
\def \level {h}

\def \l {\mathcal{L}}

\def \a {\delta}
\def \aa {\a}
\def \aastar {\a^{\star}}

\def \aastarlpf {\widehat{\a}^{\star}}

\def \aanew {\widetilde{\a}}

\def \u {u}
\def \ulpf {\widehat{\u}}

\def \ad {\u}
\def \adstar {\u^{\star}}
\def \adstarlpf {\widehat{\u}^{\star}}

\def \addash {\widetilde{\u}}

\def \adnew {\widetilde{\u}}

\def \ghm{\G^I}
\def \ght{\G^H}

\def \Fnpf {\mathcal{F}}

\def \X {\mathcal{X}}
\def \Xnpf {\mathcal{X}}
\def \Xlpf {\widehat{\mathcal{X}}}
\def \Xupf {\widecheck{\mathcal{X}}}
\def \Xcpf {\mathcal{X}_\cpf}

\def \Lnpf {\mathcal{L}}
\def \Llpf {\widehat{\mathcal{L}}}
\def \Lupf {\widecheck{\mathcal{L}}}

\def \Lunpf {\mathcal{L}^u}
\def \Lulpf {\widehat{\mathcal{L}}^u}
\def \Luupf {\widecheck{\mathcal{L}}^u}

\def \dad {\mathrm{DAD}}
\def \dadnpf {[\dad]}
\def \dadlpf {[\widehat{\dad}]}
\def \dadupf {[\widecheck{\dad}]}

\def \adnpf {[\mathrm{AD}]}
\def \adlpf {[\widehat{\mathrm{AD}}]}
\def \adupf {[\widecheck{\mathrm{AD}}]}

\def \mp {\mathrm{a}}
\def \sp {\mathrm{d}}
\def \admpnpf {[\mathrm{AD}]^\mp}
\def \admplpf {[\widehat{\mathrm{AD}}]^\mp}
\def \admpupf {[\widecheck{\mathrm{AD}}]^\mp}

\def \adspnpf {[\text{AD}]^\sp}
\def \adsplpf {[\widehat{\mathrm{AD}}]^\sp}
\def \adspupf {[\widecheck{\mathrm{AD}}]^\sp}
\def \adspcpf {[\widetilde{\mathrm{AD}}]^\sp}

\def \psilpf {\widehat{\psi}}
\def \psiupf {\widecheck{\psi}}
\def \psistarnpf {\psi^\star}
\def \psistarlpf {\widehat{\psi}^\star}
\def \psistarupf {\widecheck{\psi}^\star}
\def \Psistarlpf {\widehat{\Psi}^\star_\arcm}
\def \Psistarupf {\widecheck{\Psi}^\star_\arcm}
\def \psistarcpf {\widetilde{\psi}^\star}
\def \psicpf {\widetilde{\psi}}

\def \philpf {\widehat{\phi}}
\def \phiupf {\widecheck{\phi}}
\def \phistarnpf {\phi^\star}
\def \phistarlpf {\widehat{\phi}^\star}
\def \phistarupf {\widecheck{\phi}^\star}
\def \phistarcpf {\widetilde{\phi}^\star}
\def \phicpf {\widetilde{\phi}}

\def \nunpf {\nu}
\def \nulpf {\widehat{\nu}}
\def \nuupf {\widecheck{\nu}}
\def \nucpf {\widetilde{\nu}}
\def \nustarcpf {\widetilde{\nu}^\star}

\def \Snpf {S}
\def \Slpf {\widehat{S}}
\def \Supf {\widecheck{S}}
\def \Scpf {\widetilde{S}}
\def \Sstarcpf {\widetilde{S}^\star}

\def \Pnpf {P}
\def \Plpf {\widehat{P}}

\def \Qnpf {Q}
\def \Qlpf {\widehat{Q}}

\def \xnpf {\x}
\def \xlpf {\widehat{\x}}
\def \xupf {\widecheck{\x}}
\def \xcpf {\widetilde{\x}}

\def \ellnpf {\ell}
\def \elllpf {\widehat{\ell}}
\def \ellupf {\widecheck{\ell}}
\def \ellcpf {\widetilde{\ell}}
\def \ellstarcpf {\widetilde{\ell}^\star}
\def \ellstarnpf {{\ell}^\star}

\def \sgsetnpf {\sgset}
\def \sgsetlpf {\widehat{\sgset}}
\def \sgsetupf {\widecheck{\sgset}}

\def \sgsetdstarlpf {\sgsetlpf^{\mathrm{d}\star}}
\def \sgsetdstarupf {\sgsetupf^{\mathrm{d}\star}}

\def \pgnpf {\pg}
\def \pglpf {\widehat{\pg}}

\def \sgsetanpf {\sgset^\mathrm{a}}
\def \sgsetalpf {\widehat{\sgset}^\mathrm{a}}
\def \sgsetaupf {\widecheck{\sgset}^\mathrm{a}}
\def \sgsetanew {\widetilde{\sgset}^\mathrm{a}}
\def \sgsetastarnpf{\sgset^\mathrm{a\star}}
\def \sgsetacpf{\widetilde{\sgset}^\mathrm{a}}

\def \sgsetdnpf {\sgset^\mathrm{d}}
\def \sgsetdlpf {\widehat{\sgset}^\mathrm{d}}
\def \sgsetdupf {\widecheck{\sgset}^\mathrm{d}}

\def \gammastarnpf {\gamma^\star}
\def \gammastarlpf {\widehat{\gamma}^\star}

\def \scnpf {\sc}
\def \sclpf {\widehat{\sc}}
\def \scupf {\widecheck{\sc}}

\def \sgnpf {\sg}
\def \sglpf {\widehat{\sg}}
\def \sgupf {\widecheck{\sg}}

\def \gammanpf {\gamma}
\def \gammalpf {\widehat{\gamma}}
\def \gammaupf {\widecheck{\gamma}}

\def \f {f}
\def \fstar {f^\star}
\def \fnew {{f'}}
\def \fcpf {\widetilde{f}}
\def \xstarnpf {\xnpf^{\star}}

\def \ssum {\textstyle\sum}
\def \pprod {\textstyle\prod}
\def \mmax {\textstyle\max}
\def \mmin {\textstyle\min}

\def \aargmax {\textstyle\argmax}
\def \aargmin {\textstyle\argmin}

\def \UPF {\epsilon\text{-LPF}}

\def \Ni {\N^i}
\def \m {m}
\def \mi {\m^i}
\def \Mi {M^i}
\def \gi {g^i}

\def \J {J}

\def \aalpf {\widehat{\aa}}
\def \aaupf {\widecheck{\aa}}

\def \aailpf {\aalpf^i}
\def \aaklpf {\aalpf^k}
\def \aaiupf {\aaupf^i}

\def \aaistarlpf {\aalpf^{i\star}}

\def \Dim {\D_{\arcm}^i}
\def \Dimlpf {\widehat{\D}_{\arcm}^i} 
\def \Dimupf {\widecheck{\D}_{\arcm}^i}

\def \Dastarlpf{\widehat{\D}_{\arcm}^{\star}}
\def \Dastarupf{\widecheck{\D}_{\arcm}^{\star}}

\def \rxratio {\sfrac{\normalsize\textbf{r}}{\normalsize\textbf{x}}}
\def \wcratio {\sfrac{\normalsize\mathbf{\W}}{\normalsize\mathbf{\c}}}

\def \mytextsection {\textsection\hspace{-0.0cm}}

\def \securePreceq {\preccurlyeq}

\def \sgsetdstar {\sgset^{\mathrm{d}\star}}
\def \phistar {\phi^\star}

\def \aastarnew {\widetilde{\aa}^\star}
\def \phistarnew {\widetilde{\phi}^\star}
\def \Lunpfnew {\Lnpf^{\adnew}}
\def \Lulpfnew {\Llpf^{\adnew}}

\def \adstaronelpf {\ad^{\star1}}
\def \adstartwolpf {\ad^{\star2}}

\newcommand{\myComment}[1]{\Comment{\begin{footnotesize}
#1
\end{footnotesize}}}

\newcommand{\argmin}{\operatornamewithlimits{argmin}}
\newcommand{\argmax}{\operatornamewithlimits{argmax}}

\usepackage{chngcntr}
\counterwithout{subsubsection}{subsection}

\titleformat{\subsubsection}[runin]
  {\normalfont\normalsize\bfseries\itshape}{\thesubsubsection)}{2pt}{}[:]
\makeatletter
\@addtoreset{subsubsection}{subsection}
\makeatother

\newcommand{\tcb}[1]{\color{black} #1 \color{black}}

\usepackage{tcolorbox}

\usepackage{mdframed}

\title{Security Assessment of Electricity Distribution Networks under DER Node Compromises}
\author{Devendra Shelar and Saurabh Amin
\thanks{Manuscript submitted on August 3, 2016}
\thanks{Mailing address: Department
of Civil and Environmental Engineering, Massachusetts Institute of Technology, 77 Massachusetts Avenue 1-241, Cambridge, MA 02139 USA (e-mail: \texttt{shelard,amins}@mit.edu, phone: 857-253-8964).}
\thanks{This work was supported by EPRI grant for ``Modeling the Impact of ICT Failures on the Resilience of Electric Distribution Systems'' (contract ID: 10000621), and NSF project ``FORCES'' (award $\#$: CNS-1239054).}
}

\begin{document}
\maketitle

\begin{abstract}
This article focuses on the security assessment of electricity Distribution Networks (DNs) with vulnerable Distributed Energy Resource (DER) nodes. The adversary model is simultaneous compromise of DER nodes by strategic manipulation of generation set-points. The loss to the defender (DN operator) includes loss of voltage regulation and cost of induced load control under supply-demand mismatch caused by the attack. A 3-stage Defender-Attacker-Defender (DAD) game is formulated: in Stage 1, the defender chooses a security strategy to secure a subset of DER nodes; in Stage 2, the attacker compromises a set of vulnerable DERs and injects false generation set-points; in Stage 3, the defender responds by controlling loads and uncompromised DERs. Solving this trilevel optimization problem is hard due to nonlinear power flows and mixed-integer decision variables. To address this challenge, the problem is approximated by a tractable formulation based on linear power flows. The set of critical DER nodes and the set-point manipulations characterizing the optimal attack strategy are computed. An iterative greedy approach to compute attacker-defender strategies for the original nonlinear problem is proposed. These results provide guidelines for optimal security investment and defender response in pre- and post-attack conditions, respectively.
 
\end{abstract}

\IEEEpeerreviewmaketitle

\section{Introduction}
Integration of distributed energy resources (DERs) such as solar photovoltaic (PV) and solar thermal power generation with electricity distribution networks (DNs) is a major aspect of smart grid development. Some reports estimate that, by 2050, solar PVs will contribute up to $23.7\;\%$ of the total electricity generation in the US. Large-scale deployment of DERs can be utilized to improve grid reliability, reduce dependence on bulk generators (especially, during peak demand), and decrease network losses (at least, up to a certain penetration level)~\cite{hiskens}. Harnessing these capabilities requires secure and reliable operation of cyber-physical components such as smart inverters, DER controllers, and communication network between DERs and remote control centers. Thus, reducing security risks is a crucial aspect of the design and operation of DNs~\cite{tabuada, sinopoli, bulloDetectionAndIdentification,  karlJohansson, karlJohanssonSandberg}. This article focuses on the problem of security assessment of DNs under threats of DER node disruptions by a malicious adversary.  
 
We are specifically interested in limiting the loss of voltage regulation and supply-demand mismatch that can result from the simultaneous compromise of multiple DERs nodes on a distribution feeder. It is well known that the active power curtailment and reactive power control are two desirable capabilities that can help maintain the  operational requirements in DNs with large-scale penetration of DERs with intermittent nature~\cite{hiskens, garcia, turitsyn}. We investigate the specific ways in which these capabilities need to be built into the DER deployment designs, and show that properly chosen security strategies can protect DNs against a class of security attacks.
 
Our work is motivated by recent progress in three topics: \textbf{(T1)} Interdiction and cascading failure analysis of power grids (especially, transmission networks)~\cite{baldick, salmeron,yaoTrilevelOptimization}; \textbf{(T2)} Cyber-physical security of networked control systems~\cite{cardenas, tabuada, sinopoli, bulloDetectionAndIdentification, karlJohansson,zhangReviewer2}; and \textbf{(T3)} Optimal power flow (OPF) and control of distribution networks with DERs~\cite{low, hiskens, turitsyn}.
 
Existing work in \textbf{(T1)} employs state-of-the-art computational methods for solving large-scale, mixed integer programs for interdiction/cascade analysis of transmission networks assuming direct-current (DC) power flow models. Since our focus is on security assessment of DNs, we also need to model reactive power demand, in addition to the active power flows. In this work, we consider standard DN model with constant power loads and DERs~\cite{low, turitsyn}, but we restrict our attention to  tree networks. This enables us to obtain structural results on optimal attack strategies. We show that these structural results also provide guidelines for investment in deploying IT security solutions, especially in geographically diverse DNs.

The adversary model in this paper considers simultaneous DER node compromises by false-data injection attacks. Thanks to the recent progress in \textbf{(T2)}, similar models have been proposed for a range of cyber-physical systems~\cite{sinopoli, bulloDetectionAndIdentification}. Our model is motivated by the DER failure scenarios proposed by power system security experts~\cite{nescor}. These scenarios consider shutdown of DER systems when an external threat agent compromises the DERs by a direct attack, or by manipulating the power generation set-points sent from the control center to individual DER nodes/controllers; see \cref{fig:IEEE37NodeNetwork1}. \tcb{Indeed, the security threats to DNs are real. The recent cyber attack on Ukraine's power grid shows that an external attacker can compromise multiple DN components by exploiting commonly known IT vulnerabilities~\cite{ukraine}. Another real-world attack that is directly related to the attack model introduced in this paper was highlighted in a 2015 Congressional Research Service report~\cite{richardCampbell}. This attack was conducted by computer hackers to obtain a back-door entry to the grid. They exploited the IT systems that enable integration of DERs/renewable energy sources. }

\ifOneColumn
\def \myWidth {12cm}
\else
\def \myWidth {7cm}
\fi 
\begin{figure}[h!]
\includegraphics[width=7cm]{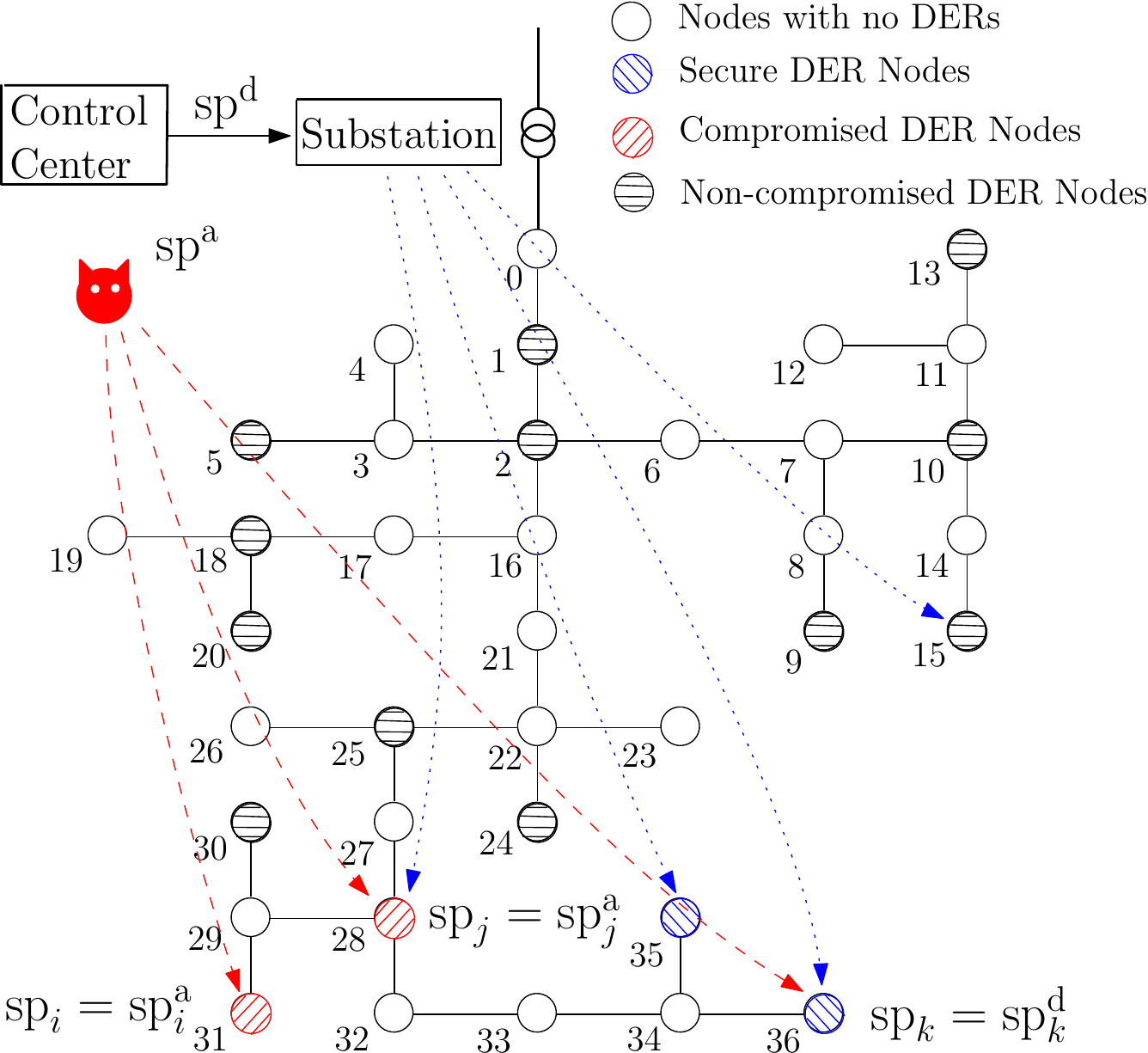}
\caption{Illustration of the DER failure scenario proposed in~\cite{nescor} on a modified IEEE 37-node network. }
\label{fig:IEEE37NodeNetwork1}
\end{figure}

%
%
In our model, the attacker's objective is to impose loss of voltage regulation to the defender (i.e., network operator), and also induce him to exercise load control in order to reduce the supply-demand mismatch immediately after the attack. The defender's primary concern in post-attack conditions is to reduce the costs due to loss of voltage regulation and load control. Hence, in our model, the line losses are assigned a relatively lesser weight. For a fixed attack, solving for a defender response via load control and control of uncompromised DERs is similar to the recent results in \textbf{(T3)}, i.e., using convex relaxations of the OPF problem.  

Our main contribution is analysis of a three-stage sequential security game posed in \cref{sec:modeling}. In Stage~1, the defender invests in securing a subset of DER nodes but cannot ensure security of all nodes due to  his  budget constraint; in Stage~2, the resource-constrained attacker compromises a subset of vulnerable DER nodes and manipulates their set-points; in Stage~3, the defender responds by regulating the supply-demand mismatch. This defender-attacker-defender (DAD) game models both strategic investment decisions (Stage 1) and operation of DN during attacker-defender interaction (Stages 2-3). Solving the  DAD game is a hard problem due to the nonlinear power flow, DER constraints, and mixed-integer decision variables.
 
In \Cref{sec:sequentialGameLPF}, we provide tractable approximations of the sub-game induced for a fixed defender security strategy, i.e., the attacker-defender interaction in Stages~2-3; see \cref{thm:lpfNpfUpfRelationship}. These approximations can be efficiently solved, and hold under the assumption of no reverse power flows, small impedances, and small line losses. Next we show structural results for the master-problem (i.e., optimal attack for fixed defender response), and the sub-problem (i.e., optimal defender response for fixed attack). For the master-problem, we derive the false set-points that the attacker will introduce in any compromised DER (\Cref{thm:attackerSetpoints}),  and also propose computational methods to solve for attack vectors, i.e., DER nodes whose compromise will cause maximum loss to the defender (\cref{prop:greedyOptimality,prop:downstream}). For the sub-problem, we utilize the convex relaxations of OPF to compute optimal defender response for a fixed attack (\cref{lem:convexRelaxationOptimalSolution}), and under a restricted set of conditions, provide a range of new set-points for the uncompromised DERs (\cref{prop:pvSetpoints}). These results lead to a greedy approach, which efficiently computes the optimal attack and defender response (\cref{algo:greedyApproach}).  We prove optimality of the greedy approach for DNs with identical resistance-to-reactance ratio (\cref{thm:greedyApproach}), and show that the approach efficiently obtains optimal attack strategy and defender response for a broad range of conditions (\cref{sec:caseStudy}).  Thanks to the structural results on optimal attack strategy, our greedy approach has significantly better computational performance than the standard techniques to solve bilevel optimization problems (e.g., Bender’s decomposition~\cite{salmeron}). Finally, we provide a characterization of the optimal security strategy for Stage 1 decision by the defender, albeit for symmetric DNs (\cref{sec:securityDesign}, \cref{thm:optimalSecureDesign}).

\tcb{In the following, the reader should note that the proofs of \cref{lem:lpfAndUpf,lem:lpfNpfRelationship,lem:convexRelaxationOptimalSolution,lem:pvCompromiseEffects,lem:optimalAttackSetIndependentOfGamma}, \Cref{prop:allPowerFlows,prop:pvSetpoints,prop:greedyOptimality,prop:downstream,prop:generalApproach,prop:complexityOfDadlpf} and 
\cref{thm:optimalSecureDesign} are provided in the online supplementary material~\cite{shelaraminOnlineSupplementary}. }

\section{Problem Formulation}
\label{sec:modeling}
\subsection{Distribution network model}

We summarize the standard network model of radial electric distribution systems~\cite{turitsyn,baran,shelaramin}. Consider a tree network of nodes and distribution lines $\G = (\N \cup \{0\},\E)$, where $\N$ denotes the set of all nodes except the substation (labeled as node $0$), and let $\NN \coloneqq \abs{\N}$. Let $V_i \in \C$ denote the complex voltage at node $i$, and $\nu_i \coloneqq\abs{V_i}^2$ denote the square of voltage magnitude. We assume that the magnitude of substation voltage $\abs{V_0}$ is constant.  Let $I_{j} \in \C$ denote the current flowing from node $i$ to node $j$ on line $(i,j)\in\E$, and $\ell_{j} \coloneqq \abs{I_{j}}^2$ the square of the magnitude of the current. A distribution line $(i,j)\in\E$ has a complex impedance $z_{j} = r_{j} +\j x_{j}$, where $r_{j} > 0$ and $x_{j} > 0$ denote the resistance and inductance of the line $(i,j)$, respectively, and $\j = \sqrt{-1}$. 

The voltage regulation requirements of the DN under \emph{nominal} no attack conditions govern that:
\begin{equation}\myEquationStyle
\label{eq:voltageConstraint}
\forall\quad i\in\N,\quad \nuMin_i \le \nu_i \le \nuMax_i,
\end{equation}
where $\nuMin_i = \abs{\Vmin_i}^2$ and $\nuMax_i = \abs{\Vmax_i}^2$ are the \emph{soft} lower and upper bounds for maintaining voltage quality at node~$i$. Additionally, voltage magnitudes under \emph{all} conditions satisfy:
\begin{equation}
\label{eq:hardLowerVoltageBound}
\forall\quad i\in\N,\quad \muMin \le \nu_i \le \muMax,
\end{equation}
where $\muMin$ and $\muMax$ are the \emph{hard} voltage safety bounds for any nodal voltage, and $0 < \muMin < \min_{i\in\N}\nuMin_i \le   \max_{i\in\N}\nuMax_i < \muMax$. 

\subsubsection{Load model} We consider constant power loads~\cite{farivar}. \footnote{\tcb{We do not consider frequency dependent loads as our analysis is limited to attacks that do not cause disturbances in system frequency; see \cref{subsec:assumptions} for our justification of constant system frequency assumption.} } Let $\sc_i \coloneqq \pc_i+\j\qc_i$ denote the power consumed by a load at node $i$, where $\pc_i$ and $\qc_i$ are the real and reactive components. Let  $\scdem_i \coloneqq \pcdem_i+\j\qcdem_i$ denote the \emph{nominal} power demanded by a node $i$, where $\pcdem_i$ and $\qcdem_i$ are the real and reactive components of $\scdem_i$. Under our assumptions, for all $i\in\N$, $\pc_i \le \pcdem_i$ and $\qc_i \le \qcdem_i$, i.e., the actual power consumed at each node is upper bounded by the nominal demand:
\begin{equation}\myEquationStyle
\label{eq:maxLoad} \forall\quad i\in\N,\quad \sc_i \le \scdem_i.
\end{equation}

\subsubsection{DER model}\footnote{We use the term DER to denote the complete DER-inverter assembly attached to a node of DN.} Let $\sg_i \coloneqq \pg_i + \j \qg_i$ denote the power generated by the DER connected to node $i$, where $\pg_i$ and $\qg_i$ denote the active and reactive power, respectively. Following ~\cite{low},~\cite{turitsyn}, $\sg_i$ is bounded by the apparent power capability of the inverter, which is a given constant $\sgmax_i$. 
We denote the DER set-point by $\sgset_i =\Re(\sgset_i) + \j \Im(\sgset_i)$, where $\Re(\sgset_i)$ and $\Im(\sgset_i)$ are the real and reactive components. The power generated at each node is constrained as follows: 
\begin{equation}\myEquationStyle
\forall\quad i\in\N, \quad \sg_i \le \sgset_i \in \S_i,  \label{eq:dgConstraint1}
\end{equation}
where $\S_i \coloneqq \{\sgset_i \in \C \quad|\quad \Re(\sgset_i) \ge 0 \text{ and } \abs{\sgset_i} \le \sgmax_i\}$. $\S \coloneqq  \prod_{i\in\N}\S_i$ denotes the set of configurable set-points. 

We denote  the net power consumed at node~$i$ by $\sn_i \coloneqq \sc_i - \sg_i$. A DN can be fully specified by the tuple $\langle\G, \abs{V_0}, z, \scdem, \sgmax\rangle$, where $z, \scdem, \sgmax$ are row vectors of appropriate dimensions, and are assumed to be constant.

 


\subsubsection{Power flow equations}
The 3-phase balanced nonlinear power flow (NPF) on line $(i,j)\in\E$ is given by~\cite{baran}: 
\begin{subequations}
\label{eq:powerFlowEquations}
\begin{alignat}{2}\myEquationStyle
\label{eq:NPFconservation} \Snpf_{j} &=\myEquationStyle \ssum_{k:(j,k)
\in\E}\Snpf_{k} +\scnpf_j - \sgnpf_j + z_{j}\ellnpf_{j} \\
\label{eq:voltageSquare2}\myEquationStyle \nunpf_j &=\myEquationStyle \nunpf_i - 2\Re(\conj{z}_{j}S_{j})+ \abs{z_{j}}^2\ellnpf_{j}\\
\myEquationStyle \ellnpf_{j} &=\myEquationStyle \frac{\abs{\Snpf_{j}}^2}{\nunpf_i},\label{eq:currentMagnitudeNpf} 
\end{alignat}
\label{eq:NPF}
\end{subequations}
\noindent where $\Snpf_{j} = \Pnpf_{j} + \j \Qnpf_{j}$ denotes the complex power flowing from node $i$ to node $j$  on line $(i,j)\in \E$, and $\bar{z}$ is the complex conjugate of $z$;  \eqref{eq:NPFconservation} is the power  conservation equation; \eqref{eq:voltageSquare2} relates the voltage drop and the power flows; and ~\eqref{eq:currentMagnitudeNpf} is the current-voltage-power relationship. For the NPF model  \eqref{eq:NPF}, we define a state as follows: 
\begin{equation*}
\label{eq:stateVector}
\xnpf \coloneqq \begin{bmatrix}\nunpf,\ellnpf,\scnpf,\sgnpf,\Snpf\end{bmatrix},
\end{equation*}
where $\xnpf \in \R_{+}^{2\NN}\times\C^{3\NN}$, and $\nunpf$, $\ellnpf$, $\scnpf$, $\sgnpf$, and $\Snpf$ are row vectors of appropriate dimensions. Let $\Fnpf$ denote the set of all states $\xnpf$ that satisfy \eqref{eq:hardLowerVoltageBound}, \eqref{eq:maxLoad}, \eqref{eq:dgConstraint1} and the NPF model \eqref{eq:NPF}, and define the set of all states with \emph{no reverse power flows} (see \cref{subsec:assumptions} for additional assumptions) as follows:
\begin{equation*}
\Xnpf \coloneqq \{\xnpf \in \Fnpf | \Snpf \ge 0\}.
\end{equation*}

The linear power flow (LPF) approximation of \eqref{eq:NPF} is:
\begin{subequations}\vspace*{-0.05cm}
\begin{alignat}{2}\myEquationStyle
 \myEquationStyle \Slpf_{j} &= \myEquationStyle \ssum_{k:(j,k)\in\E} \Slpf_{k} +\sclpf_j - \sglpf_j \label{eq:LPFConservation} \\
 \myEquationStyle \nulpf_j &= \myEquationStyle \nulpf_i - 2\Re(\conj{z}_{j}\Slpf_{j}) \label{eq:linearVoltageSquare2}\\
\myEquationStyle \elllpf_{j} &= \myEquationStyle  \frac{\abs{\Slpf_{j}}^2}{\nulpf_i},\label{eq:currentMagnitudeLpf} 
\end{alignat}\label{eq:LPF}
\end{subequations}where $\xlpf \coloneqq [\nulpf, \elllpf, \sclpf, \sglpf,\Slpf]$ is a state of the LPF model, and analogous to the NPF model, define the set of LPF states $\xlpf$ with no reverse power flows as $\Xlpf$.

\subsection{Notation and definitions}
\label{subsec:notationDefinitions}
All vectors are row vectors, unless otherwise stated. For two vectors $c$ and $d$, $c\ \emult\ d$ denotes their Hadamard product. 

Let $K_{j} \coloneqq \myFrac{r_{j}}{x_{j}}$ be the resistance-to-reactance ($\rxratio$) ratio for line $(i,j)\in\E$, and let $\underline{K}$ and $\overline{K}$ denote the minimum and maximum of the $K_{j}$s over all $(i,j) \in \E$. We say that DERs at nodes $j$ and $k$ are homogeneous with respect to each other if their  set-point configurations as well as their apparent power capabilities are identical, i.e., $\sgset_j = \sgset_k$ and $\sgmax_j = \sgmax_k$. Similarly, two loads at nodes $j$ and $k$ are homogeneous if $\scdem_j = \scdem_k$.

\begin{figure}
\centering
\resizebox{!}{2.5cm}{
\begin{tikzpicture}
	\tikzstyle{every node}=[font=\small]
	\tikzstyle{fde}=[circle,draw]
\def\y{0pt}
\def\yy{-30pt}
\def\zy{-10pt}
\def\x{40pt}
\node (0)[fde,draw,xshift=0pt,yshift=0*\y] {$0$};
\node (a)[fde,draw,xshift=\x,yshift=\y] {$a$};
\node (b)[fde,draw,xshift=2*\x,yshift=2*\y] {$b$};
\node (c)[fde,draw,xshift=3*\x,yshift=3*\y] {$c$};
\node (i)[fde,draw,xshift=4*\x,yshift=4*\y] {$i$};
\node (m)[fde,draw,xshift=5*\x,yshift=5*\y] {$m$};
\node (e)[fde,draw,xshift=2*\x,yshift=\yy] {$e$};
\node (f)[fde,draw,xshift=3*\x,yshift=\yy] {$d$};
\node (k)[fde,draw,xshift=4*\x,yshift=\yy] {$k$};

\node (g)[fde,draw,xshift=2*\x,yshift=-\yy] {$g$};
\node (j)[fde,draw,xshift=3*\x,yshift=-\yy] {$j$};

\path (b) edge (e);
\path (a) edge (g);\path (g) edge (j);
\path (b) edge (f);\path (f) edge (k);
\path (0) edge (a);\path (a) edge (b);\path (b) edge (c);\path (c) edge (i);\path (i) edge (m);


\draw [red] plot [smooth cycle] coordinates {(-5pt,10pt) (35pt,13pt) (75pt,42pt) (125pt,42pt) (135pt,30pt) (125pt,20pt) (85pt,20pt) (45pt,-10pt) (-5pt,-10pt) (-15pt,0pt)};
\draw [blue,dashed] plot [smooth cycle] coordinates {(0pt,10pt) (160pt,10pt)  (160pt,-10pt) (0pt,-10pt) };
\end{tikzpicture}}
\caption{Precedence description of the nodes for a tree network. Here, $j\prec_i k$, $e =_i k$, $b\prec k$, $\P_j = \{a,g,j\}$, $\P_i \cap \P_j = \{a \}$.} 
\label{fig:precedenceDescription}
\end{figure}
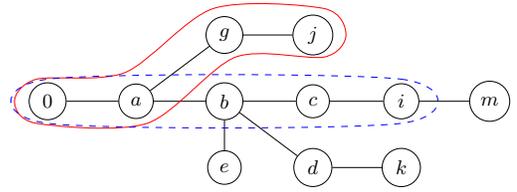

For any given node $i\in \N$, let $\P_i$ be the path from the root node to node $i$. Thus, $\P_i$ is an ordered set of nodes starting from the root node and ending at node~$i$, excluding the root node; see \cref{fig:precedenceDescription}. 
We say that node~$j$ is an \emph{ancestor} of node~$k$ ($j\prec k$), or equivalently, $k$ is a successor of $j$  \textit{iff} $\P_j \subset \P_k$. We define the \emph{relative ordering} $\preceq_i$, with respect to a ``pivot" node $i$ as follows:
\begin{itemize}
\item[-] $j$ \textit{precedes} $k$ ($ j\preceq_i k$) \textit{iff} $\P_i \cap \P_j \subseteq \P_i\cap \P_k$.

\item[-] $j$ \textit{strictly precedes} $k$ ($ j\prec_i k$) \textit{iff} $\P_i \cap \P_j \subset \P_i\cap \P_k$.

\item[-] $j$ is at the \textit{same precedence level} as $k$ ($j =_i k$) \textit{iff} \\
$\P_i \cap \P_j = \P_i \cap \P_k $.
\end{itemize} 


We define the common path impedance between any two nodes $i, j \in \N$ as the sum of impedances of the lines in the intersection of paths $\P_i$ and $\P_j$, i.e., $\zz_{ij} \coloneqq \ssum_{k\in \P_i\cap\P_j} z_{k}$, and denote  the resistive (real) and inductive (imaginary) components of $\zz_{ij}$ by 
$\rr_{ij}$ and $\xx_{ij}$, respectively. 
 
Finally, we define some useful terminology for the tree network $\G$. Let $\nl$ denote the height of $\G$, and let $\N_{\level}$ denote the set of nodes on level~$\level$ for $h = 1,2,\cdots,\nl$. For any node~$i\in\N$, $\level_i$ denotes the level of node $i$; $\child_i$ the set of children nodes of node~$i$;  $\Lambda_i$ the set of nodes in the subtree rooted at node $i$;  $\Lambda_i^j$ the set of nodes in the subtree rooted at node $i$ until level $\level_j$, where $j\in\Lambda_i$; $\Nleaf$ the set of leaf nodes, i.e., $\Nleaf \coloneqq \{j\in\N\ |\ \not\exists\ k\in\N \text{ s.t. } (j,k) \in\E \} $.

\subsection{Defender-Attacker-Defender security game}
\label{subsec:attackerDefenderGame}

%
We consider a 3-stage sequential game between a defender (network operator) and an attacker (external threat agent). 
\begin{itemize}
\item[-] \textbf{Stage 1: } The defender chooses a security strategy\\
 $u\in\Dd$ to secure a subset of DERs;

\item[-]\textbf{Stage 2: } The attacker chooses from the set of DERs that were not secured by the defender in Stage 1, and manipulates their set-points according to a strategy \\
$\psi \coloneqq  \begin{bmatrix}
\sgseta,\delta
\end{bmatrix}\in\Psi_{\arcm}(\ad)$;

\item[-]\textbf{Stage 3: } The defender responds by choosing the set-points of the uncompromised DERs and, if possible, impose load control at one or more nodes according to a strategy~
$\phi \coloneqq \begin{bmatrix}
\sgsetd, \gamma 
\end{bmatrix}\in\Phi(\ad,\psi)$. 
\end{itemize} 

\tcb{The $\dadnpf$ game is a sequential game of perfect information, i.e. each player is perfectly informed about the actions that have been chosen by the previous players. The equilibrium concept is the classical Stackelberg equilibrium. }  

In this game, $\AD$ and $\Phi(\ad,\psi)$ denote the set of defender actions in Stage 1 and 3, respectively; and $\Psi_{\arcm}(\ad)$ denotes the set of attacker strategies in Stage 2. Formally, the defender-attacker-defender $\dadnpf$ game is as follows:
\begin{equation}\label{eq:dadnpf} 
\hspace{-0.3cm}\dadnpf\;\Lnpf \coloneqq  \mmin_{\ad\in\AD} \; \mmax_{\psi\in\Psi_{\arcm}} \;  \mmin_{\phi\in\Phi} \;   \Loss(\x(\ad,\psi,\phi))  
\end{equation}
\vspace{-0.3cm}
\begin{subequations}\label{eq:dadGameConstraints}
\begin{alignat}{8}\myEquationStyle \text{s.t. }\quad &\xnpf(\ad,\psi,\phi) &&\in \X \label{eq:npfDomainConstraint}\\
 &\sc(\ad,\psi,\phi) &&= \gamma \emult \scdem \label{eq:directLoadControl}\\
&\sg(\ad,\psi,\phi) &&= \ad \emult \sgsetd + (\textbf{1}_\NN - \ad) \nonumber \\
&&&\quad \emult [\delta\emult\sgseta + (\mathbf{1}_\NN - \delta)\emult\sgsetd], \label{eq:pvSetpoints} 
\end{alignat}
\end{subequations} where \eqref{eq:directLoadControl} specifies that the \emph{actual power consumed} at node~$i$ is equal to the power demand scaled by the corresponding load control parameter $\gamma_i\in[\gammaMin_i, 1]$ chosen by the defender. 

The constraint \eqref{eq:pvSetpoints} models the net effect of defender choice $\ad_i$ in Stage 1, the attacker choice $(\sgseta_i,\delta_i)$ in Stage 2, and the defender choice $\sgsetd_i$ in Stage 3 on the \emph{actual power generated} at node~$i$. Thus, \eqref{eq:pvSetpoints} is the \emph{adversary model} of $\dadnpf$ game: the DER~$i$ is compromised \emph{if and only if} it was not secured by the defender ($\ad_i = 0$) \emph{and} was targeted by the attacker ($\aa_i=1$). Specifically, if $i$ is compromised, $\sgset_i = \sgseta_i$, where $\sgseta_i = \Re(\sgseta_i) + \j \Im(\sgseta_i)$ is the false set-point chosen by the attacker. The set-points of non-compromised DERs are governed by the defender, i.e., if DER~$i$ is not compromised $\sgset_i = \sgsetd_i$. 

\tcb{Note that the physical restriction~\eqref{eq:dgConstraint1} applies to all DER nodes, including the compromised ones. If the attacker's set-point violates this constraint, it will not be admitted by the inverter as a valid set-point. Such an attack will not affect the attack model~\eqref{eq:pvSetpoints}, and consequently it will not change the actual power generated by the DER.} Also, our adversary model assumes that the DERs' power output, $\sg$, quickly attain the set-points specified by \eqref{eq:pvSetpoints}. Thus we do not consider dynamic set-point  tracking. \footnote{Note that, under this adversary model, the impact of DER compromise is different than the impact of a natural event, e.g. cloud cover, during which $\pg = \textbf{0}$. The reactive power contribution may be non-negative during a natural event; however, as we show in \Cref{subsec:fixedDefenderResponse}, a compromised DER contributes reactive power equal to the negative of apparent power capability. }   

\tcb{During the nominal operating conditions, the network operator minimizes the line losses due to power flow on the distribution lines ($\lloss$). Typical OPF formulations mainly account for this cost. However, this objective function is not representative of the loss incurred by operator (defender) during the aforementioned attack on the DN.} We define loss function in $\dadnpf$ as follows:
\begin{equation}\label{eq:lossFunction}
\Loss(\xnpf(\ad, \psi, \phi)) \coloneqq  \llovr(\xnpf) + \lvoll(\xnpf) + \lloss(\xnpf),
\end{equation}
where $\llovr(\x)$ and $\lvoll(\x)$ model the \emph{monetary cost} to the defender due to the loss in voltage regulation and the cost of load curtailment/shedding (i.e., loss due to partially satisfied demand), respectively. The term denotes $\lloss$ the total line losses. These costs are defined as follows:
\begin{subequations}\label{eq:lossFunctionDefinitions}
\begin{alignat}{8}\myEquationStyle
\label{eq:lovrDefinition}
\llovr(\xnpf) &\coloneqq \linfinityNorm{W \emult (\nuMin - \nu)_+}\\
\label{eq:vollDefinition}
\lvoll(\xnpf) &\coloneqq  \lonenorm{\c\emult \left(1-\gamma\right) \emult \pcdem}\\
\lloss(\xnpf) &\coloneqq \lonenorm{r\emult\ellnpf},
\end{alignat}
\end{subequations} 
where $\W,\c \in \R^\NN_+$. The weight $\W_i$ is the cost of unit voltage bound violation and $\c_i$ is the cost of shedding unit load (or demand dissatisfaction) at node $i$, and $r$ denotes the vector of resistances. Note that $\llovr$ is the maximum of the weighted non-negative difference between the lower bound $\nuMin_i$ and nodal voltage square~$\nu_i$. We expect that during the attack, the defender's primary concern will be to satisfy the voltage regulation requirements, and minimize the inconvenience to the customers due to load curtailment. Thus, we assume that the weights $\W_i$ and $\c_i$ are chosen such that $\lloss$ is relatively small compared to $\llovr$ and $\lvoll$. 

Note that we added the $\lloss(\x)$ term in~\eqref{eq:lossFunction} primarily to ensure that the loss function $\Loss(\x)$ remains strictly convex function of the net demand $\sn = \sc - \sg$.  The strict convexity allows us to have a unique solution for the inner problem for fixed attacker's actions. In our computational study in \cref{sec:caseStudy}, we choose the weights $W$ and $C$ such that the line loss is negligible compared to $\llovr$ and $\lvoll$. 

\tcb{However, more generally, the loss function $\Loss(\x)$ should reflect the monetary costs incurred by the defender in maintaining the supply-demand balance and in restoring the safe operating conditions after the attack. Such a general model will contain following terms: (a) the cost of supplying additional power from the substation node to match the difference between actual power consumed by the loads and the effective DER generation ($\lsupply(\x)$); (b) the cost due to the loss of voltage regulation ($\llovr(\x)$); (c) the cost of  curtailing or shedding certain loads ($\lvoll(\x)$); (d) the cost of reactive power (VAR) control and the cost of energy spillage for the uncompromised DERs ($\lcurt(\x)$); and (e) the costs of equipment damage due to the attack ($\ldamage(\x)$).} 

\tcb{For the sake of simplicity, we do not consider $\lcurt(\x)$ and $\ldamage(\x)$ in our formulation. The choice of ignoring $\lcurt(\x)$ can be justified if we assume that the DER owners participate in VAR control, perhaps in return of a pre-specified compensation by the operator/defender. Alternatively, the DERs may be required to contribute reactive power during contingency scenarios (i.e., supply-demand mismatch during the attack). The main difficulty in modeling $\ldamage(\x)$ is that it requires relating the state vector to the probability of equipment failures. Since our focus is on security assessment of DNs, as  opposed to network reinforcement using investment in physical protection devices, we ignore this cost in our analysis. Finally, we also ignore the contribution of $\lsupply(\x)$ to the loss function, as it is likely to be dominated by $\llovr$ and $\lvoll$. }

\subsubsection*{Stage 1 [Security Investment]} The set of defender actions is: 
\begin{equation*}\myEquationStyle
\Dd \coloneqq \{\ad \in \{0,1\}^\N \ |\ \lzeronorm{u} \le B\},
\end{equation*} 
where $B \le \abs{\N}$ denotes a security budget. Since, securing control-center's communication to every DER node in a geographically diverse DN might be costly/impractical, we impose that the maximum number of nodes the defender can secure is $B$. A defender's choice  $\ad \in \Dd$ implies that a DER at node~$i$ is secure if $\ad_i = 1$ (i.e. DER at node~$i$ cannot be compromised), and vulnerable to attack if $\ad_i = 0$. Let $\Ns(\ad) \coloneqq \{i\in\N |\ad_i = 1\}$ and $\Nv(\ad) \coloneqq \N\backslash\Ns(\ad)$ denote the set of secure and vulnerable nodes, for a given $\ad$. \footnote{Note that by a ``secure" node, we mean that the DER at that node is not prone to compromise by the attacker. From a practical viewpoint, the defender can secure a DER node by investing in node security solutions such as intrusion prevention systems (IPS) \cite{cardenas}. These security solutions are complementary to the device hardening technologies that can secure the DER-inverter assembly. Our focus is on security against a threat agent interested in simultaneously compromising multiple DERs. Thus, we restrict our attention to node security solutions. }

\tcb{There are several factors which limit the defender's ability to ensure full security of DERs. First, to ensure the security of control software and network communications that support DER operations, we need cost-effective and interoperable security solutions that can be widely adopted by different entities (e.g., DER manufacturers, service providers, and owners). Secondly, the DNs are likely to inherit some of the vulnerabilities of COTS IT devices that may directly or indirectly affect DER operations. Third, the defenders (operators) need to justify the business case to deploy security solutions. Existing work on security investments in such networked environments, indicates that the operators tend to underestimate security risks~\cite{aminGalina}. Consequently, in the absence of proper regulatory impositions, they tend to underinvest in well-known security solutions. In our model, we capture the limitations imposed by these factors by introducing a security budget $B$ which restricts the maximum number of nodes the defender can secure in Stage~$1$.}


\subsubsection*{Stage 2 [Attack]} Let $\Psi_{\arcm}(\ad) \coloneqq \S(u) \times \Da(u)$ denotes the set of attacker actions for a defender's choice $\ad$, where 
\begin{align*}
\begin{aligned}\myEquationStyle
\S(u) &\coloneqq \pprod_{i\in\Nv(u)} \S_i \times \pprod_{j\in\Ns(u)} \{0+0\j \}\\
\Da(u) &\coloneqq \{\aa\in \{0,1\}^\N \ | \ \delta \le \unity-\ad, \;  \lzeronorm{\aa} \le \arcm \},
\end{aligned}
\end{align*}
and $\arcm \le \abs{\Nv}$ is the maximum number of DERs that the attacker can compromise. This limit accounts for the attacker's resource constraints (and/or restrict his influence based on his knowledge of DER vulnerabilities). The attacker \emph{simultaneously} compromises a subset of vulnerable DER nodes by introducing incorrect set-points (see the adversary model \eqref{eq:pvSetpoints}), and increase the loss $\Loss$ (see \eqref{eq:lossFunction}). The attacker's choice is denoted by $\psi \coloneqq \begin{bmatrix}
\sgseta, \aa
\end{bmatrix} \in \Psi_{\arcm}(u)$, where $\sgseta$ denotes the vector of  incorrect set-points chosen by the attacker, and $\aa\in \Da $ denotes the attack vector that indicates the subset of DERs compromised. A DER at node~$i$ is compromised if $\aa_i = 1$, and not compromised if $\aa_i=0$. 

\tcb{We assume that the attacker has full information about the DN, i.e., she knows $\langle\G, \abs{V_0}, z, \scdem, \sgset\rangle$ and maximum fraction of controllable load at each node. The attacker also knows the set of DERs secured by the defender in Stage 1 of the game, voltage regulation bounds, and defender's cost parameters (i.e. the weight $W_i$ for voltage bound violation and the cost of unit load shedding $C_i$ for each node~$i$). By assuming such an \emph{informed attacker}, we are able to focus on how the attacker uses the knowledge of the physical system toward achieving her objective. Thus, we take a conservative approach and do not explicitly consider particular mechanisms of how a security vulnerability might be exploited by the attacker. Admittedly, our attack model may be unrealistic in some scenarios; however, it allows us to identify the critical DER nodes, and characterize optimal security investment and defender response; see \Cref{sec:securityDesign}. }

\tcb{Next, we justify the attacker's resource constraint~$\arcm$. First, the DERs are likely to be heterogeneous in their capacity, design, and manufacturer type. The attacker may not have the specific knowledge to exploit vulnerabilities in all DER systems deployed on a DN. Secondly, in practice, the process of DER integration is gradual and so is the progress on implementing security solutions in the  control processes that support DER operations. The attacker's capability to compromise DERs depends on how the available threat channels vary which such a technological change. Third, the security of DNs is likely to be affected by the security practices adopted by owners of DERs. For example, the attacker's capability will be limited if the DER operations are secured by a regulated distribution utility who faces compliance checks or mandatory disclosure of known incidents. In contrast, he is more likely to gain a backdoor entry if the DN has substantial participation from a variety of third party DER owners who may not follow prudent security practices. In our analysis, we model the attacker's capability by introducing a parameter $\arcm$, which is the maximum number of DERs that the attacker can compromise. }

\subsubsection*{Stage 3 [Defender Response]} Let $\gammaMin_i \ge 0$  denote the maximum permissible fraction of load control at node~$i$, and define the set of Stage~3 defender actions:  \begin{equation*}
\Phi(\ad,\psi) \coloneqq \S \times \Gamma,  
\end{equation*}
where $\myEquationStyle
\Gamma \coloneqq \prod_{i\in\N}[\gammaMin_i,1]$. The defender chooses new set-points $\sgsetd$ of non-compromised DERs, and load control parameters $\gamma_i$ to reduce the loss $\Loss$. The defender action is modeled as a vector $\phi \coloneqq \begin{bmatrix}
\sgsetd, \gamma
\end{bmatrix} \in \Phi(\ad,\psi)$, where $\sgsetd$ (resp. $\gamma$) denotes the vector of $\sgsetd_i$ (resp. $\gamma_i$). 

\tcb{We make the standard assumption that the defender knows the nominal demand (i.e., the demand in pre-attack conditions) using measurements collected from the DN nodes. We also assume that the defender can distinguish between compromised and non-compromised DERs. In heavy loading conditions, the defender expects the output of a non-compromised DER to lie in the first quadrant (see \cref{fig:voltVarControl} in \cref{subsec:fixedDefenderResponse}), i.e.  it contributes positive active and reactive power to the DN. A simple technique to detect compromised DERs is whether the inverter output lies in the fourth quadrant.}

\subsection{Assumptions about the DN model}\label{subsec:assumptions}

In general, $\dadnpf$ is a non-convex, non-linear,  tri-level optimization problem with mixed-integer decision variables. Hence, it is a computationally hard problem. Our goals are: \begin{itemize}
\item[(i)]  to provide structural insights about the optimal attacker and defender strategies of the $\dadnpf$ game; 
\item[(ii)]  to approximate the non-linear (hard) problem by formulating computationally tractable variants based on linear power flow models. 
\end{itemize} To address these goals we make the following assumptions:
\begin{subassumption}\label{assm:voltageLimits}
\textbf{Voltage quality: } In no attack (nominal)  conditions, both $\Xnpf$ and $\Xlpf$ satisfy the voltage quality  bounds~\eqref{eq:voltageConstraint}. 
\end{subassumption} 
\begin{subassumption}\label{assm:hardLowerBoundOfVoltage}
\textbf{Safety: } Safety bounds \eqref{eq:hardLowerVoltageBound} are always satisfied,  i.e., \small $\forall\; (\ad,\psi,\phi) \in \Dd\times\Psi\times\Phi$,  $\forall\; \xnpf(\ad,\psi,\phi)\in\Xnpf$, $\muMin\unity \le \nu \le \muMax\unity$. \normalsize
\end{subassumption} 
\begin{subassumption}\label{assm:noReversePowerFlows}\textbf{No reverse power flows: } Power flows from node $0$ towards the downstream nodes, i.e., $\Slpf \ge 0$. This implies that $\forall\;\xlpf\in\Xlpf, \; \nulpf \le \nu_0\unity$; similarly, for NPF model. 
\end{subassumption} 
\begin{subassumption}
\label{assm:lowImpedance}\textbf{Small impedance: } All power flows are in the per unit (\textit{p.u.}) system, i.e., $\nu_0 = 1$ and $\forall\;(i,j)\in\E,\;\abs{S_j} < 1$. Furthermore, the resistances and reactances are small, i.e., 
\begin{equation*}
 \forall\  (i,j)\in \E, r_{j} \le \myFrac{\muMin^2}{4\muMin+8} < 1,\ x_{j} \le \myFrac{\muMin^2}{4\muMin+8}< 1, 
\end{equation*}
and the common path resistances and reactances are also smaller than 1, i.e., $\rr_{ii} \le 1 \text{ and } \xx_{ii} \le 1 \ \forall\ i\in\N$. 
\end{subassumption}
\begin{subassumption}\label{assm:smallLineLosses}
\textbf{Small line losses: } The line losses are very small compared to power flows, i.e., $\forall\;\xnpf\in\Xnpf,\;z\emult\ell \le \slr S$, where $\slr$ is a small positive number. \footnote{Equivalently, $\slr$ is an upper bound on the maximum ratio of the magnitudes of line losses and the power flows, i.e., $\slr =  \max_{(i,j) \in \E,\Pnpf_j \ne 0,\Qnpf_j \ne  0}\max\big(\sfrac{r_j\ellnpf_j}{\Pnpf_j}, \sfrac{x_j\ellnpf_j}{\Qnpf_j}\big)$. 
\tcb{Thus, $\slr$ can be determined by setting the values of  loads to the corresponding nominal demands, and then computing the line losses and power flows for nominal conditions.}}
\end{subassumption}

\subassmref{assm:voltageLimits}-\subassmref{assm:hardLowerBoundOfVoltage} are standard assumptions. \subassmref{assm:noReversePowerFlows} assumes that the  DER penetration level is such that the net demand is always positive. In real-world DNs, both $r_j$s and $x_j$s are typically around 0.01 \subassmref{assm:lowImpedance}. Also, residential load power factors  ($\sfrac{\pc_j}{\abs{\sc_j}}$) are in range of  0.88-0.95. For these values, one can show that $\slr \approx 0.05$~\subassmref{assm:smallLineLosses}. We will denote \subassmref{assm:voltageLimits}-\subassmref{assm:smallLineLosses} by \hspace{0.5cm} \mbox{\setcounter{assmcounter}{-1}\begin{assumption}\label{assm:throughout}.\end{assumption}} 
 


%

\tcb{In addition to the aforementioned assumption, we also assume that \emph{(a)} the node $0$ is an infinite bus; \emph{(b)}  the voltage $\nu_0$ is constant, and \emph{(c)} the system frequency is constant. 

These assumptions are standard in the steady state power flow analyses, and can be justified as follows: Our focus is on the security assessment of DNs that have substation nodes with high enough ramp rates in supplying $\sim 50\;MW$ power (typical for medium-voltage (MV) substations). That is,  any supply-demand imbalance of the order of $50\;MW$ can be cleared relatively quickly by the substation; hence the infinite substation bus assumption.} 

\tcb{The assumption \emph{(b)} is typical in OPF formulations and we make it for the sake of mathematical convenience. Indeed, as a consequence of attack, there will be a net reduction in the substation voltage relative to the pre-attack value $\nu_0$. This effect is due to a higher net demand after the Stage 3 of the game. To meet this  additional demand, higher currents will flow through the distribution lines, resulting in even higher drops in the nodal voltages than what we obtained using the computational approach detailed in \Cref{subsec:fixedDefenderResponse}. Thus, our estimate of the optimal loss is actually a lower bound on the true value of optimal loss that the defender would face when the substation voltage drops after the attack.}

\tcb{To justify assumption \emph{(c)}, we argue that even large-scale penetration of DERs is not likely to achieve a generation capacity beyond $50\;MW$ from a single DN. Even in the worst case, i.e. when all the DERs are simultaneously disconnected, their impact on the system frequency will be negligible. } 
%

Next, we choose $\slrrr$ as follows
\begin{equation}
\slrrr \coloneqq (1-\slr)^{-\Pmax} - 1,
\end{equation}
where $H$ is the height of the tree DN and $\slr$ is chosen as above. Now, consider another linear power flow model (which we call the $\UPF$ model): 
\begin{subequations}
\begin{alignat}{2}\myEquationStyle
\label{eq:UPFConservation} \Supf_{j} &= \ssum_{k:(j,k)\in\E}\Supf_{k} + (1+\slrrr)(\scupf_j - \sgupf_j) \\
\label{eq:linearVoltageSquareUpf}\nuupf_j &= \nuupf_i - 2\Re(\conj{z}_{j}\Supf_{j})\\
\ellupf_{j} &= \myFrac{\abs{\Supf_{j}}^2}{\nuupf_i},\label{eq:currentMagnitudeUpf} 
\end{alignat}\label{eq:UPF}
\end{subequations} 
\noindent and $\myEquationStyle\xupf \coloneqq \begin{bmatrix}
\nuupf, \ellupf, \scupf, \sgupf, \Supf\end{bmatrix}$ is a state of $\UPF$ model, and $\Xupf$ is the set of all states $\xupf$ with no reverse power flows.  (Note that for $\slrrr = 0$, \eqref{eq:UPF} becomes \eqref{eq:LPF}.)

\tcb{We also note that both $\lpf$ and $\UPF$ models ignore the line losses term $z_j\ell_j$ in the power balance equation (5a),  and the term $\abs{z_j}^2\ell_j$ in the  voltage drop equation (5b). The power flows obtained by ignoring these terms approximate the non-linear power flow (NPF) model calculations under the assumption $(\mathbf{A0})_3$, i.e., the line impedances are very small $\abs{z_j} \ll 1$. Under the assumption $(\mathbf{A0})_2$, i.e. no reverse power flows, the $\lpf$ provides a lower bound on the line power flows, and an upper bound on the nodal voltages of the standard DistFlow model~\cite{farivar},\cite{low}. The main use of $\UPF$ model is that it provides an \emph{upper bound} on the line power flows and a lower bound on the nodal voltages; see \cref{prop:allPowerFlows} in \Cref{subsec:upperAndLowerBoundsOnLnpf}.}

We will consider two variants of the $\dadnpf$ game \eqref{eq:dadnpf}-\eqref{eq:dadGameConstraints}: 
\begin{subequations}
\begin{equation*}
\dadlpf\;\Llpf \coloneqq  \mmin_{\ad\in\AD} \; \mmax_{\psi\in\Psi_{\arcm}} \;  \mmin_{\phi\in\Phi} \;   \Losslpf(\xlpf(\ad,\psi,\phi)) \label{eq:dadlpf}
\end{equation*}
\vspace{-0.3cm}
\begin{align*}\label{eq:lpfDomainConstraint}
\begin{aligned}\myEquationStyle \text{s.t. }\quad &\xlpf(\ad,\psi,\phi) \in \Xlpf, \eqref{eq:directLoadControl}, \eqref{eq:pvSetpoints},
\end{aligned}
\end{align*}
\end{subequations}
and
\begin{subequations}
\begin{equation*}
\color{black}\dadupf\;\Lupf \coloneqq  \mmin_{\ad\in\AD} \; \mmax_{\psi\in\Psi_{\arcm}} \;  \mmin_{\phi\in\Phi} \;   \Lossupf(\xupf(\ad,\psi,\phi)) \label{eq:dadupf}
\end{equation*}
\vspace{-0.3cm}
\begin{align*}\label{eq:upfDomainConstraint}
\begin{aligned}\myEquationStyle \text{s.t. }\quad &\xupf(\ad,\psi,\phi) \in \Xupf, \eqref{eq:directLoadControl}, \eqref{eq:pvSetpoints}, 
\end{aligned}
\end{align*}
\end{subequations} \normalsize
where $\Losslpf(\xlpf) \coloneqq \llovr(\xlpf) + \lvoll(\xlpf)$, and $\Lossupf(\xupf) \coloneqq \llovr(\xupf) + \lvoll(\xupf)$ are the loss functions for $\dadlpf$ and $\dadupf$, respectively. Note that the loss functions $\Losslpf$ and $\Lossupf$ do not have the line losses term. 
\normalsize
The optimal loss $\Loss$ of  $\dadlpf$ and $\dadupf$ are denoted by  $\Llpf$ and $\Lupf$, respectively. Our results in \mytextsection\ref{sec:sequentialGameLPF}-\ref{sec:securityDesign} show that \textbf{(a)} $\dadlpf$ (resp. $\dadupf$) help  provide under (resp. over) approximation of $\dadnpf$; and \textbf{(b)} the derivation of structural properties of optimal strategies in both $\dadlpf$ and $\dadupf$ is analogous to one another.

\tcb{We will, henceforth, abuse the notation, and use $\Psi$ and $\Phi$ to denote $\Psi_{\arcm}(\ad)$ and $\Phi(\ad,\psi)$, respectively. For a summary of notations, see \Cref{tab:notationsTable} in the Appendix. }

\section{Attacker-Defender Sub-game}
\label{sec:sequentialGameLPF}

In this section, we consider the sub-game (Stages~2 and~3) induced by a fixed defender security strategy $\ad$ in Stage 1:  
\begin{align*}
\begin{aligned}\myEquationStyle
\adnpf \quad \Lunpf\  \coloneqq\  &\text{max}_{\psi\in\Psi} \  \text{min}_{\phi\in\Phi} \ \Loss(\x(\ad, \psi,\phi)) \quad \text{s.t. }\quad  \eqref{eq:dadGameConstraints} \end{aligned}
\end{align*}
\normalsize
Analogous to the variants of $\dadnpf$, $\dadlpf$ and $\dadupf$, we define two variants of the sub-game $\adnpf$:  $\adlpf$ (resp. $\adupf$) with $\Xlpf$ (resp. $\Xupf$) in \eqref{eq:npfDomainConstraint}. The optimal losses of $\adlpf$ and $\adupf$ are denoted by $\Lulpf$ and  $\Luupf$), respectively. 

For simplicity and \emph{without loss of generality}, we focus on case for $\ad = \textbf{0}$; i.e., no node is secured by the defender in Stage~1. With further abuse of notation, for a  strategy profile $(\textbf{0}, \psi, \phi)$, we denote $\xnpf(\textbf{0}, \psi, \phi)$ by $\xnpf(\psi,\phi)$ as the solution of NPF model. Similarly, redefine $\xlpf(\psi,\phi)$ and $\xupf(\psi,\phi)$. We also drop the superscript~$u$ from $\Lunpf$, $\Lulpf$ and $\Luupf$. 
 
Following the computational approach in the literature to solve (bilevel) interdiction problems  \cite{baldick}, \cite{kevinwood}, we define the master-problem $\admpnpf$ (resp. sub-problem $\adspnpf$) for fixed $\phi \in \Phi$ (resp. fixed $\psi \in \Psi$): 
\begin{align*}
\begin{aligned}\myEquationStyle
&\admpnpf \quad  \psistarnpf(\phi) &&\in  \aargmax_{\psi\in\Psi} &&\Loss(\xnpf(\psi,\phi)) \quad &&\text{s.t.}\quad \eqref{eq:dadGameConstraints},\\
&\adspnpf \quad \phistarnpf(\psi) &&\in  \aargmin_{\phi\in\Phi} &&\Loss(\xnpf(\psi,\phi))\quad &&\text{s.t.}\quad \eqref{eq:dadGameConstraints}.  
\end{aligned}
\end{align*}
Similarly, define master- and sub- problems $\admplpf$ and $\adsplpf$ (resp. $\admpupf$ and $\adspupf$)  for the variants $\adlpf$ (resp. $\adupf$). 

\Cref{subsec:upperAndLowerBoundsOnLnpf} focuses on bounding the optimal loss for $\adnpf$ with the losses in $\adlpf$ and $\adupf$. The master- and sub- problems are addressed in \cref{subsec:fixedAttackerStrategy} and \cref{subsec:fixedDefenderResponse}, respectively. This leads to a computationally efficient iterative approach in \cref{subsec:greedy} to solve the sub-games $\adnpf$, $\adlpf$, $\adupf$.  \cref{fig:technicalResultsFlowChart} provides an  outline of results in this section. 

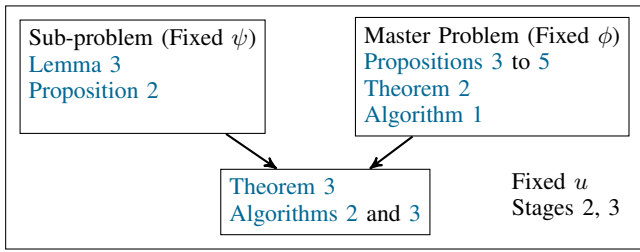
\begin{figure}[h]
\centering
\resizebox{!}{!}{
\begin{tikzpicture}
	\tikzstyle{every node}=[font=\small]
	\tikzstyle{fde}=[circle,draw]
	\tikzset{
	    >=stealth',
	    punkt/.style={
	           rectangle,
	           rounded corners,
	           draw=black, very thick,
	           text width=6.5em,
	           minimum height=2em,
	           text centered},
	    pil/.style={
	           ->,
	           thick,
}
	}
\def\y{0pt}
\def\yy{-30pt}
\def\zy{-10pt}
\def\x{40pt}

\node[draw,text depth = 3cm,minimum width=8.5cm,font=\normalsize,align=left] (main){};

\node[draw,align=left](three) at ([xshift=5.2em,yshift=-2.7em]main.north west){ Sub-problem (Fixed $\psi$) \\ 
 \cref{lem:convexRelaxationOptimalSolution} \\ \cref{prop:pvSetpoints} \\ };
 
\node[draw,align = left](two) at ([xshift=-5.8em,yshift=-2.7em]main.north east){Master Problem (Fixed $\phi$) \\ \cref{prop:greedyOptimality,prop:downstream,prop:generalApproach} \\ \cref{thm:attackerSetpoints} \\ \cref{algo:greedyAlgorithm}};

\node[draw,align=left](final) at ([xshift=0em,yshift=1.8em]main.south){ \cref{thm:greedyApproach} \\ 
\cref{algo:greedyAlgorithm2,algo:greedyApproach}};

\node[align=left](one) at ([xshift=-3em,yshift=2em]main.south east) {Fixed $\ad$\\ Stages 2, 3}; 


\path (three) edge[pil] (final);
\path (two) edge[pil] (final);

\end{tikzpicture}}
\caption{Outline of technical results in \cref{sec:sequentialGameLPF}.}\label{fig:technicalResultsFlowChart}
\end{figure}


\subsection{Upper and Lower Bounds on $\Lnpf$}
\label{subsec:upperAndLowerBoundsOnLnpf}

\begin{theorem}\label{thm:lpfNpfUpfRelationship} Let $(\psistarnpf,\phistarnpf)$, $(\psistarlpf,\phistarlpf)$ and $(\psistarupf,\phistarupf)$ be optimal solutions to $\adnpf$, $\adlpf$ and $\adupf$, respectively; and denote the optimal losses by $\Lnpf$, $\Llpf$, $\Lupf$,  respectively. Then,  
\begin{equation}\myEquationStyle\label{eq:lpfNpfUpfRelationship}
\Llpf \le \Lnpf \le \Lupf + \myFrac{\muMin\NN}{2\muMin+4}.
\end{equation} 
\end{theorem}

To prove  \cref{thm:lpfNpfUpfRelationship}, we first state \Cref{lem:lpfAndUpf,lem:lpfNpfRelationship}, and \cref{prop:allPowerFlows} that relates $\xnpf(\psi,\phi)$, $\xlpf(\psi,\phi)$, and $\xupf(\psi,\phi)$: 
\begin{lemma}\label{lem:lpfAndUpf}
Consider a fixed $(\psi, \phi)\in\Psi\times\Phi$. The following holds: $\scnpf = \sclpf = \scupf$, $\sgnpf = \sglpf = \sgupf$, and
\begin{subequations}
\begin{alignat}{8}
\label{eq:lpfAndUpf}\Supf &= (1+\slrrr)\Slpf \\
\nuupf - \nu_0\unity &= (1+\slrrr)(\nulpf - \nu_0\unity) \label{eq:nulpfupfRelationship}
\end{alignat}
\end{subequations}
\begin{subnumcases}{\hspace{-3.4cm}\forall\ (i,j)\in\E\ }
 \hspace{-0.7cm} & $\Snpf_j = \sum_{k\in\Lambda_j} \sn_k + z_k\ellnpf_k$\label{eq:recursiveNPFConservation}\\  \hspace{-0.7cm} & $\Slpf_j = \sum_{k\in\Lambda_j} \sn_k$\label{eq:recursiveLPFConservation}
\end{subnumcases}

\begin{subnumcases}{\hspace{-0.8cm}\forall\ j\in\N\ }
\hspace{-0.7cm} &$\nulpf_j = \nu_0 - 2\sum_{k\in\N} \Re(\conj{\zz}_{jk} \sn_k)$ \label{eq:recursiveVoltageSquareLpf} \\
  \hspace{-0.7cm} &$\nuupf_j = \nu_0 - 2(1+\slrrr)\sum_{k\in\N} \Re(\conj{\zz}_{jk} \sn_k)$\label{eq:recursiveVoltageSquareUpf}\\
\hspace{-0.7cm} &$\nulpf_j = \nu_0 - 2\sum_{k\in\P_j} \Re(\conj{z}_{k} \Slpf_k)$ \label{eq:recursiveVoltageSquareLpfToRoot} \\
\hspace{-0.7cm} &$\nuupf_j = \nu_0 - 2\sum_{k\in\P_j} \Re(\conj{z}_{k} \Supf_k)$. \label{eq:recursiveVoltageSquareUpfToRoot}
\end{subnumcases}
\normalsize
\end{lemma}

\begin{lemma}\label{lem:lpfNpfRelationship}
For a fixed $ (\psi, \phi) \in \Psi \times \Phi$, 
\begin{equation}\myEquationStyle
\forall\quad (i,j)\in\E, \quad \Snpf_j \le \myFrac{\Slpf_j}{(1-\slr)^{\nl-\abs{\P_j} + 1}}.
\end{equation}
\end{lemma}

\begin{proposition} \label{prop:allPowerFlows}
For a fixed strategy profile $ (\psi, \phi)  \in \Psi\times\Phi$, 
\begin{equation*}
\Slpf \le \Snpf \le \Supf, \quad \nulpf \ge \nunpf \ge \nuupf, \quad \elllpf \le \ellnpf \le \ellupf. 
\end{equation*}

Hence, 
\small 
\begin{align}
\myEquationStyle
\begin{rcases}
\hspace{-0.3cm} \llovr(\xlpf) \le \llovr(\xnpf) \le \llovr(\xupf) \\
\hspace{-0.3cm} \lvoll(\xlpf) = \lvoll(\xnpf) = \lvoll(\xupf) \\ 
\hspace{-0.3cm} \lloss(\xlpf) \le \lloss(\xlpf) \le  \lloss(\xupf)
\end{rcases} \implies \Loss(\xlpf) \le \Loss(\xnpf) \le \Loss(\xupf). \label{eq:optimalValueRelationship}
\end{align}
\normalsize
\end{proposition}

\Cref{prop:allPowerFlows} implies that  any attack $\psi$ that increases $\Llpf$ in $\adlpf$ (relative to the no attack case), also increases $\Lnpf$ in $\adnpf$ and $\Lupf$ in $\adupf$, respectively. The converse need not be true, i.e., an attack that increases $\Lnpf$ in $\adnpf$ (resp. $\Lupf$ in $\adupf$) need not increase $\Llpf$ in $\adlpf$ (resp. $\Lnpf$ in $\adnpf$). Similarly, any defender response $\phi$ that reduces $\Lupf$ (resp. $\Lnpf$), also reduces $\Lnpf$ (resp. $\Llpf$). Again, the converse statements do not apply here.

\begin{proof}[Proof of \textbf{\Cref{thm:lpfNpfUpfRelationship}}]


For any $\xnpf\in\Xnpf$, \small 
\begin{align}\label{eq:lineLossUpperBound}
\begin{aligned}
\hspace{-0.5cm}\lloss(\xnpf) \stackrel{\eqref{eq:currentMagnitudeNpf}}{=}  \ssum\limits_{(i,j)\in\E}\myFrac{r_j(\Pnpf_j^2+\Qnpf_j^2)}{\nunpf_i} \stackrel{\subassmref{assm:hardLowerBoundOfVoltage},  \subassmref{assm:lowImpedance}}{\le} \myFrac{2}{\muMin}\ssum\limits_{(i,j)\in\E} r_j \stackrel{\subassmref{assm:lowImpedance}}{\le} \myFrac{\muMin\NN}{2\muMin+4}.
\end{aligned}
\end{align} \normalsize
Hence, 
\begin{align*} \small 
\begin{aligned}\myEquationStyle
\Lupf & = \Lossupf(\xupf(\psistarupf,\phistarupf(\psistarupf))) \\
&\ge \Lossupf(\xupf(\psistarnpf,\phistarupf(\psistarnpf))) && (\text{by optimality of } \psistarupf) \\
&\ge \Lossupf(\xnpf(\psistarnpf,\phistarupf(\psistarnpf))) && (\text{by \cref{prop:allPowerFlows}})\\
&\stackrel{\eqref{eq:lineLossUpperBound}}{\ge}  \Loss(\xnpf(\psistarnpf,\phistarupf(\psistarnpf))) - \myFrac{\muMin\NN}{2\muMin+4} &&  \\
&\ge  \Loss(\xnpf(\psistarnpf,\phistarnpf(\psistarnpf))) - \myFrac{\muMin\NN}{2\muMin+4}  && (\text{by optimality of } \phistarnpf) \\
& = \Lnpf  - \myFrac{\muMin\NN}{2\muMin+4}.
\end{aligned}
\end{align*} \normalsize

Similarly, one can show $\Lnpf \ge \Llpf$. 
\end{proof}

\cref{thm:lpfNpfUpfRelationship} implies that the value of the sub-game $\adnpf$ with NPF can be lower (resp. upper) bounded by the value of $\adlpf$ (resp. $\adupf$). Our subsequent results show that both $\adlpf$ and $\adupf$ admit computationally efficient solutions. 

\subsection{Optimal defender response to fixed attacker strategy $\psi$}
\label{subsec:fixedAttackerStrategy}

We consider the sub-problem $\adspnpf$ of computing optimal defender response $\phistarnpf(\psi)$ for a fixed attack $\psi$.

The following Lemma shows that $\adspnpf$ is a Second-Order Cone Program (SOCP), and hence, can be solved efficiently. 
\begin{lemma} \label{lem:convexRelaxationOptimalSolution} Let $\Xcpf \coloneqq conv(\Xnpf)$ denote the set of states $\xnpf$ satisfying \eqref{eq:hardLowerVoltageBound}-\eqref{eq:dgConstraint1}, \eqref{eq:NPFconservation}, \eqref{eq:voltageSquare2}, and the relaxation of \eqref{eq:currentMagnitudeNpf}: 

For a fixed $\psi \in \Psi$, the problem of minimizing $\Loss(\xnpf(\psi,\phi))$ subject to $\xnpf\in\Xcpf$, \eqref{eq:directLoadControl}, \eqref{eq:pvSetpoints} is a SOCP. Its optimal solution is also optimal for $\adspnpf$.  

\end{lemma}

For fixed $\psi$ (attack) and fixed load control parameter $\gamma$ (e.g. when changing $\gamma$ is not allowed), \cref{prop:pvSetpoints} below provides a range of optimal defender set-points $\sgsetlpf^{d\star}$ and $\sgsetupf^{d\star}$ for LPF and $\UPF$ models, respectively. Note that, if $\gamma$ is fixed, $\lvoll(\xlpf)$ is also fixed. Then, the defender set-points can be chosen  by using $\llovr(\xlpf)$ as a loss function, instead of $\Losslpf(\xlpf)$. Similar argument holds for $\Lossupf(\xupf)$. 

 \def \di {d_i}
 \def \myti {\theta_i}
 \def \dstari {{d_i^\star}}
 \def \mytstari {{\theta_i^\star}}
 \def \mytti {\widetilde{\theta}_i}
\begin{proposition}\label{prop:pvSetpoints}

If we fix $\gamma\in \Gamma$ in $\adsplpf$, then $\forall i\in\N$,
\begin{equation*}\label{eq:pvSetpointsGeneral} \myEquationStyle
 \aa_i = 0 \implies \abs{\sgsetlpf^{d\star}_i} = \sgmax_i, \quad  \angle \sgsetlpf_i^{d\star} \in [\arccot \kmax,\arccot \kmin]. 
\end{equation*} 
\normalsize 
\noindent Furthermore, if the DN has identical $\rxratio \equiv \k$ ratio, then 
\begin{equation}\label{eq:optimalDefenderSetpoints}\myEquationStyle
 \aa_i = 0 \implies \myEquationStyle\abs{\sgsetlpf^{d\star}_i} = \sgmax_i,\quad  \angle \sgsetlpf_i^{d\star}  = \arccot K. 
\end{equation}
\normalsize 
\noindent Similar results hold for $\adspupf$. 

\end{proposition}

\subsection{Optimal attack under fixed defender response $\phi$}
\label{subsec:fixedDefenderResponse}

Now, we focus on the master problem $\admpnpf$, i.e., the problem of computing optimal attack for a fixed defender response $\phi$. The following Theorem characterizes the optimal attacker set-point, denoted by  $\sgsetnpf^{a\star}_i = \Re(\sgsetnpf^{a\star}_i) + \j\Im(\sgsetnpf^{a\star}_i)$, when $\aa_i = 1$ (i.e. DER at node~$i$ is targeted by the attacker).   
\begin{theorem}\label{thm:attackerSetpoints} Consider $\admpnpf$ for a fixed $\delta \in \Da$ (i.e., the DERs compromised by the attacker are specified by $\aa$ and the only decision variables in $\admpnpf$ are $\sgsetanpf$). Then
\begin{equation}\label{eq:attackerSetpoints}
\forall\; i\in\N \text{ s.t. } \aa_i = 1,\quad \sgsetnpf^{a\star}_i = 0 -\j\sgmax_i. 
\end{equation} 

\vspace{0.05cm}Same holds for both $\admplpf$ and $\admpupf$.

\normalsize
\end{theorem}

\begin{proof}
If $\aa_i = 1$, then $\pgnpf_i = \pglpf_i = \Re(\sgset_i) = \Re(\sgset^{a\star}_i)$.

We first prove the simpler case for $\admplpf$. From \eqref{eq:LPF}, one can check that as functions of $\pglpf_i$, $\Plpf$ is strictly decreasing, $\Qlpf$ is constant, and $\nulpf$ is strictly increasing. Hence, $\Losslpf(\psi,\phi_f)$ is strictly increasing in $\pglpf_i$ (because $\llovr$ is non-decreasing as $\nulpf$ is decreasing; $\lvoll$ is constant). Hence, to minimize the loss $\Loss$, the attacker chooses  $\Re(\sgsetlpf^{a\star}_i) = 0$. Similarly, $\Im(\sgsetlpf^{a\star}_i) = -\j\sgmax_i$. Similarly, we can show that in $\admpupf$, $\sgsetupf^{a\star} = \mathbf{0}-\j\sgmax$.  

\ifArxivVersion
For the proof of $\sgsetastarnpf = \mathbf{0}-\j\sgmax$, please refer to the supplementary material at the end of the document. 
\else
To complete the proof we need to argue that the same conclusion (i.e., $\sgsetastarnpf = \mathbf{0}-\j\sgmax$) also holds for $\admpnpf$; we refer the reader to~\cite{shelaraminOnlineSupplementary} for this argument. 
\fi 
\end{proof} 

\begin{figure}[htbp!]
\centering
\includegraphics[height=6cm]{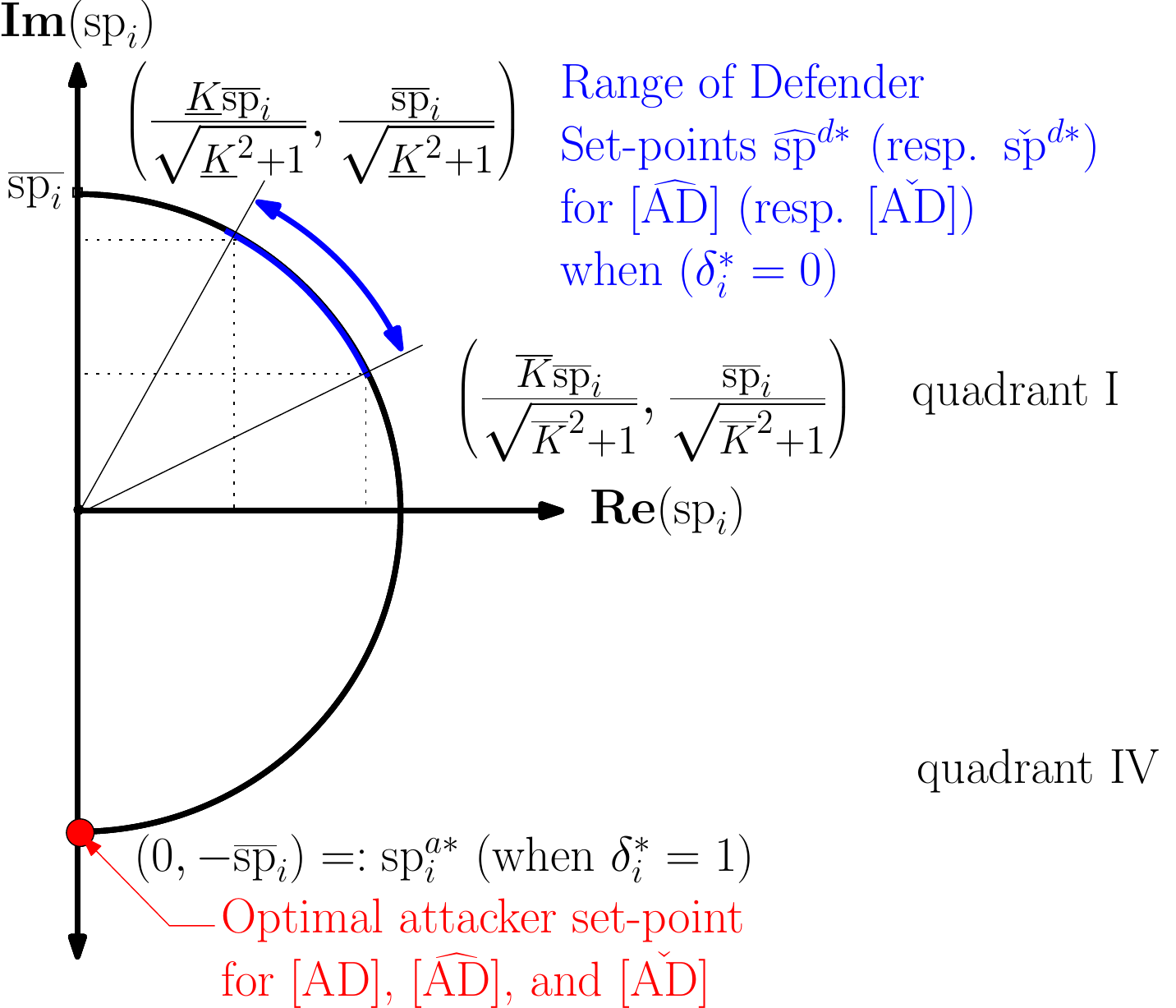}
\caption{Optimal attacker set-points (\Cref{thm:attackerSetpoints}) and range for optimal defender set-points (\Cref{prop:pvSetpoints}).}
\label{fig:voltVarControl}
\end{figure}

\Cref{fig:voltVarControl} shows the optimal attacker set-point $\sgset^{a\star}_i$ for $\aa^\star_i = 1$, and the defender set-points for the DERs for $\aa^\star_j = 0$. 

Thanks to  \Cref{thm:attackerSetpoints}, $\sc$ and $\sg$ are determined by $\aa$ and $\phi$ (since optimal $\sgsetastarnpf$ is given by \eqref{eq:attackerSetpoints}). Thus, for given $(\aa,\phi)$, loss function can be denoted as $\Loss(\xnpf(\begin{bmatrix}
\textbf{0}-\j\sgmax,\aa
\end{bmatrix},\phi))$; and $\adnpf$ can be restated as follows: 
\begin{align*}
\begin{aligned}
\Lnpf\quad = \quad & \mmax_{\aa\in\Da}\; \mmin_{\phi\in\Phi}\;  \Loss(\xnpf(\aa,\phi)) \myEquationStyle \quad  \text{ s.t. } \quad  \eqref{eq:dadGameConstraints},  \eqref{eq:attackerSetpoints}. 
\end{aligned}
\end{align*}
Same holds for $\adlpf$ (resp. $\adupf$) and $\admplpf$ (resp. $\admpupf$). Note that the attacker actions on DERs may not be limited to an incorrect set-point attack. For example, the attacker can simply choose to disconnect the DER nodes by choosing $\sgseta = \mathbf{0} + \mathbf{0}\j$. However, \cref{thm:attackerSetpoints} shows that the attacker will induce more loss to the defender by causing the DERs to withdraw maximum reactive power rather than simply disconnecting them. 

 Let $\Delta_j(\nulpf_i)$ (resp. $\Delta_\aa(\nulpf_i)$) be the change in voltage at node $i$ caused due to compromise of DER at node~$j$ (resp. compromise of DERs due to attack vector $\aa$.) Similarly, define $\Delta_j(\nuupf_i)$ and $\Delta_\aa(\nuupf_i)$. We now state a useful result: 

\begin{lemma}\label{lem:pvCompromiseEffects}
If $\phi$ is fixed, then 
\begin{subnumcases}{\hspace{-1cm}\forall\  i,j\in\N}
\hspace{-0.8cm}& $\Delta_j(\nulpf_i) = 2\Re(\conj{Z}_{ij}(\sgsetd_j + \j\sgmax_j))$ \label{eq:recursiveLPFVoltageSquare}\\
\hspace{-0.8cm}& $\Delta_j(\nuupf_i) = 2(1+\slrrr)\Re(\conj{Z}_{ij}(\sgsetd_j + \j\sgmax_j))$ \label{eq:recursiveUPFVoltageSquare}
\end{subnumcases}
\begin{subnumcases}{\hspace{-3.1cm}\forall\ \aa \subseteq \Da}
\hspace{-0.8cm} & $\Delta_\aa(\nulpf_i) = \ssum_{j:\aa_j=1} \Delta_j(\nulpf_i)$ \label{eq:sumVoltageSquareLPF}\\
\hspace{-0.8cm} & $\Delta_\aa(\nuupf_i) = \ssum_{j:\aa_j=1} \Delta_j(\nuupf_i)$. \label{eq:sumVoltageSquareUPF}
\end{subnumcases}
\end{lemma}


For a fixed $\phi\in\Phi$, let $\Dimlpf(\phi)$ be the set of optimal attack vectors that maximize voltage bounds violation under $\lpf$ at a pivot node, say~$i$. Formally, 
\begin{align}\small
\begin{aligned}
\hspace{-0.3cm}\Dimlpf(\phi) \coloneqq \argmax_{\aa\in\Da} \W_i(\nuMin_i-\nulpf_i) \; \text{ s.t. } \xlpf(\aa,\phi) \in\Xlpf, \eqref{eq:directLoadControl}, \eqref{eq:pvSetpoints}
\end{aligned}
\end{align}\normalsize
Also, let 
\begin{equation} \myEquationStyle \label{eq:defCandidateOptimalAttackSet}
\Dastarlpf(\phi) \coloneqq \bigcup_{i\in\N} \  \Dimlpf(\phi) 
\end{equation} denote the set of candidate optimal attack vectors, and $\aailpf \in \Dimlpf(\phi)$ denote any vector in $\Dimlpf$. Similarly, define $\Dimupf(\phi)$, $\Dastarupf(\phi)$, and $\aaiupf$.

Using \Cref{lem:pvCompromiseEffects}, \Cref{algo:greedyAlgorithm} computes optimal $\aastarlpf$ to maximize $\llovr$ for a fixed defender action $\phi\in\Phi$ \cite{shelaramin}. In each iteration, the Algorithm selects one node as a pivot node. For a pivot node, say~$i$, a set of target nodes $\aailpf$ is determined by selecting $\arcm$ nodes with largest $\Delta_j(\nulpf_i)$ (see \Cref{algo:helperProcedures} in Appendix). Applying \Cref{lem:pvCompromiseEffects}, the final nodal voltage at the current pivot node~$i$ is given by $\nulpf_i - \Delta_{\aailpf}(\nulpf_i)$. The attack strategy that maximizes $\llovr$ is the set $\aaklpf$ corresponding to a pivot node~$k$ that admits maximum voltage bound violation when DERs specified by $\aaklpf$ are compromised. \tcb{\Cref{algo:greedyAlgorithm} repeatedly calls procedure \Cref{algo:helperProcedures}, considering each node as the pivot node, and hence, requires $\O(n^2 \log n)$ time.} 
 
 \begin{algorithm}[h!]
 
 \begin{algorithmic}[1]
 \small
 \State $\aastarlpf(\phi) \gets $\Call{OptimalAttackForFixedResponse}{$\phi$} 
 \Procedure{OptimalAttackForFixedResponse}{$\phi$}
 \State Compute state vector for no attack $\xlpf(\textbf{0},\phi) \in \Xlpf$ 
 \For {$i \in \N$} 

		\State $\aailpf \gets  \textsc{GetPivotNodeOptimalAttack}(i,\sgsetd)$, and calculate $\Delta_{\aailpf}(\nulpf_i)$ using  Lemma.~\ref{lem:pvCompromiseEffects}
		
		\State Calculate new voltage value $\nulpf_i' \gets  \nulpf_i - \Delta_{\aailpf}(\nulpf_i)$
 
 \EndFor
 
 \State $k \gets  \argmax_{i\in\N}\W_i(\nuMin_i - \nulpf_i') $ 
 
 \State \Return $\aalpf \gets \aaklpf$ (Pick $\aaklpf$ which maximally violates \eqref{eq:voltageConstraint})
 \EndProcedure
 \Procedure{GetPivotNodeOptimalAttack}{$i$, $\sgsetd$}
   	\State $(\J,\N_{\gi},\m') \gets $ \Call{OptimalAttackHelper}{$i$, $\sgsetd$}
  	\State Randomly choose $\arcm-m'$ nodes from  $\N_{\gi}$ to form $\N'$
  	\State \Return $\aailpf \in\Da$ such that $\aailpf_k = 1 \iff k\in \J \cup \N'$
  \EndProcedure
  
 \end{algorithmic}
 \caption{Optimal Attack for Fixed Defender Response}
 \label{algo:greedyAlgorithm}
 \normalsize
 \end{algorithm}
 
The following proposition argues that \Cref{algo:greedyAlgorithm} computes the optimal attack vectors for $\admplpf$ and $\admpupf$. 

 \begin{proposition} \label{prop:greedyOptimality}
For a fixed $\phi\in\Phi$, if $\aalpf$ is the optimal attack vector computed by \Cref{algo:greedyAlgorithm}, then $\aalpf$ is also an optimal attack vector of $\admplpf$. 
  Same holds for $\admpupf$.    
 \end{proposition}

 We now show that the effect of DER compromise at either node $j$ or $k$ on the node $i$ depends upon the locations of nodes $j$ and $k$ relative to node $i$. The following Proposition states that if node $j$ is upstream to node $k$ relative to the pivot node $i$ ($j\prec_ik$), then the DER compromise at node $k$ impacts on $\nulpf_i$ more than the DER compromise on node $j$; and if $j=_i k$, then the effect of DER compromise at $j,k$ on $\nulpf_i$ is identical. 
 \begin{proposition} \cite{shelaramin} Consider $\admplpf$. Let nodes $i,j,k\in\N$ where $i$ is the pivot node, $\sgsetd_j = \sgsetd_k$, and $\sgmax_j = \sgmax_k$. \ifpropDownstream If $j \prec_i k$ (resp. $j =_i k$), then $\Delta_j(\nulpf_i) < \Delta_k(\nulpf_i) $ (resp. $\Delta_j(\nulpf_i) = \Delta_k(\nulpf_i) $). \else If $j \preceq_i k$, then $\Delta_j(\nulpf_i) < \Delta_k(\nulpf_i) $. \fi  
 Same holds true for $\admpupf$. 
 \label{prop:downstream}
 \end{proposition}
 
 \ifNotArxivVersion
 \Cref{prop:downstream} implies that, broadly speaking, compromising downstream DERs is advantageous to the attacker than compromising the upstream DERs. In other words, compromising DERs by means of clustered attacks are more beneficial to the attacker than distributed attacks. Consequently, our results (see \cref{sec:securityDesign}) on security strategy in Stage~1 show that the defender should utilize his security strategy to deter cluster attacks. 
 
 \fi 

We, now, state a result that connects the optimal attack strategies for $\admplpf$ and $\admpupf$.
\begin{proposition}\label{prop:generalApproach} 
For a fixed $\phi \in \Phi$, the following holds:

1) The sets of candidate optimal attack vectors that maximizes voltage bound violations  under $\lpf$ and $\UPF$ are identical, i.e., 
\begin{equation}
\Dastarlpf(\phi) \equiv \Dastarupf(\phi). 
\end{equation} 

2) Furthermore, assume that $\nuMin_i = \nuMin_j \eqqcolon \nuMin$ and $W_i = W_j \eqqcolon W\ \forall\ i,j\in\N$. Also, let the sets of optimal attack strategies for $\admplpf$ and $\admpupf$ be denoted by $\Psistarlpf(\phi)$ and $\Psistarupf(\phi)$, respectively. Let $\psistarlpf \in \Psistarlpf(\phi)$ and $\psistarupf \in \Psistarupf(\phi)$ be any two attack strategies. Now, if 
 \begin{equation}
 \llovr(\xlpf(\psistarlpf,\phi)) > 0  \quad\text{and}\quad  \llovr(\xupf(\psistarupf,\phi)) > 0, \label{eq:positiveLovrCond}
 \end{equation} then the sets of optimal attack strategies for $\admplpf$ and $\admpupf$  are identical, i.e.,  
 \begin{equation}
\Psistarlpf(\phi) \equiv \Psistarupf(\phi). 
 \end{equation}
 \end{proposition}
  
As we will see in \cref{subsec:greedy}, \cref{prop:generalApproach} forms the basis of our overall computational approach.

\subsection{A greedy approach for solving $\adlpf$, $\adupf$ and $\adnpf$}\label{subsec:greedy}

We now utilize results for sub- and master-problems to solve $\adnpf$. Consider the following assumption:
\begin{assumption}\label{assm:identicalRX}
DN has identical $\rxratio \equiv \k$ ratio, i.e., $\forall j\in\N, \k_j = \k$. 
\end{assumption}
In this subsection, we present an algorithm to solve $\adlpf$ and $\adupf$ under \assmref{assm:throughout} and  \assmref{assm:identicalRX}, and then propose its extension, a greedy iterative approach, for solving $\adnpf$ under the  general case.

Under \assmref{assm:throughout} and \assmref{assm:identicalRX}, the optimal defender set-points $\sgsetdstarlpf$ and $\sgsetdstarupf$ are as specified by \Cref{prop:pvSetpoints}, and hence fixed. For fixed optimal $\sgsetdstarlpf$ (resp. $\sgsetdstarupf$), we can solve the problem $\adlpf$ (resp. $\adupf$) by using Benders Cut method \cite{kevinwood}. However, we present a computationally faster algorithm,   \Cref{algo:greedyAlgorithm2} that computes attacker's candidate optimal attack vectors $\Dastarlpf$ (resp. $\Dastarupf$) using \Cref{lem:pvCompromiseEffects}.  

\begin{lemma}\label{lem:optimalAttackSetIndependentOfGamma}
Under \assmref{assm:throughout}, \assmref{assm:identicalRX}, for any two fixed $\gammalpf^1, \gammalpf^2 \in \Gamma$, $\myEquationStyle
\Dastarlpf(\small\begin{bmatrix}
\sgsetdstarlpf, \gammalpf^1 
\end{bmatrix}\normalsize) = \Dastarlpf(\small\begin{bmatrix}
\sgsetdstarlpf, \gammalpf^2 
\end{bmatrix}\normalsize)$. Same holds true for $\Dastarupf$. 
\end{lemma}

Given $\sgsetdlpf\in\S$, it can be checked that \Cref{algo:greedyAlgorithm2}, in fact, computes $\Dimlpf(\sgsetdlpf)$, and $\Dastarlpf(\sgsetdlpf) = \bigcup_{i\in\N} \Dimlpf(\sgsetdlpf)$ is the set of candidate optimal attack vectors. \tcb{The cardinality of the set $\Dimlpf$ (Line 4) in the worst-case can be as high as $\O(e^{\frac{n}{e}})$. Therefore, computing $\Dastarlpf$ can take $\O(n\exp{(\frac{n}{e})})$ time in the worst-case.} 

\Cref{algo:greedyAlgorithm2} computes the set of attacks $\Dastarlpf(\sgsetdstarlpf)$, and iterates over each $\aalpf \in \Dastarlpf(\sgsetdstarlpf)$. In each iteration, since $\sgsetd = \sgsetdstarlpf$ is fixed, the sub-problem $\adsplpf$ reduces to an LP over the variable $\gamma$. Let $\gammastarlpf(\aalpf)$ be the solution to the LP. Then, $\phistarlpf(\aalpf) = \small\begin{bmatrix}
\sgsetdstarlpf, \gammastarlpf(\aalpf) \end{bmatrix}\normalsize$ is the optimal solution to $\adsplpf$. Choosing $\aastarlpf = \argmax_{\aalpf\in\Dastarlpf} \Loss(\xlpf(\aalpf, \phistarlpf(\aalpf))) $, \Cref{algo:greedyAlgorithm2} computes the solution to be $(\aastarlpf, \phistarlpf(\aastarlpf))$ to the problem $\adlpf$. Similarly, we can use \Cref{algo:greedyAlgorithm2} to solve  $\adupf$.

\small
\begin{algorithm}[h!]

\begin{algorithmic}[1]\small
\State $(\aastarlpf, \phistarlpf, \Llpf) \gets $ \Call{Greedy-One-Shot}{\null}
\Procedure{Greedy-One-Shot}{\null}

	\State $\Llpf = 0$, $\aastarlpf = \textbf{0}$, $\gammastarlpf = \textbf{1}$, $\sgsetdstarlpf$ as in \Cref{prop:pvSetpoints}
	\State Let $\Dimlpf = \footnotesize\textsc{GetPivotNodeOptimalAttackSet}\small (i,\sgsetdstarlpf)$ 
	\State $\Dastarlpf = \bigcup_{i\in\N} \Dimlpf$ 
	\State For each $\aalpf \in \Dastarlpf$, compute $\gammastarlpf(\aalpf)$ by solving  $\adsplpf$ as an LP in $\gamma$. Let $\phistarlpf(\aalpf) =  \begin{matrix}
			 \sgsetdstarlpf, \gammastarlpf(\aa)
				\end{matrix}))$ 
	\State Let $\aastarlpf\coloneqq \argmax\limits_{\aalpf \in \Dastarlpf} \Losslpf(\xlpf(\aalpf, \begin{matrix}
			\gammastarlpf(\aalpf), \sgsetdstarlpf
			\end{matrix}))$  
	\State \Return $\aastarlpf, \phistarlpf = \phistarlpf(\aastarlpf), \Llpf = \Losslpf(\xlpf(\aastarlpf,\phistarlpf)) $ 
\EndProcedure
\Procedure{GetPivotNodeOptimalAttackSet}{$i$, $\sgsetd$}
	\State $(\J, \N_{\gi}, \m') \gets  $ \Call{OptimalAttackHelper}{$i$, $\sgsetd$}
	\State \Return $\Dim \gets \{\aalpf\in\Da | \aalpf_k = 1 \text{ \emph{iff} } k\in \J \cup \N', \text{ where } \N' \subseteq \N_{\gi} \text{ and } \abs{\N'} = M-m' \}$	
\EndProcedure 

\end{algorithmic}
\caption{Solution to $\adlpf$ for DNs with identical $\rxratio$} \label{algo:greedyAlgorithm2}
\end{algorithm}
\normalsize

\begin{theorem}\label{thm:greedyApproach}
Under \assmref{assm:throughout}, \assmref{assm:identicalRX}, let $(\aalpf, \philpf)$ be  a solution computed by \Cref{algo:greedyAlgorithm2}. Then $(\aalpf, \philpf)$ is also an optimal solution to $\adlpf$. Similar result holds for $\adupf$. 
\end{theorem}
\begin{proof} 
\def \aaistarlpf {\widehat{\aa}^{i\star}}
Under \assmref{assm:identicalRX}, $\sgsetd = \sgsetdstarlpf$ is fixed (\Cref{prop:pvSetpoints}). Then, for any $\gamma \in \Gamma$, by \Cref{lem:optimalAttackSetIndependentOfGamma} and \Cref{prop:greedyOptimality}, the optimal attack $\aastarlpf$ belongs to the set $\Dastarlpf(\sgsetdstarlpf)$.  \Cref{algo:greedyAlgorithm2} iterates over the attack vectors $\aa \in \Dastarlpf$, computes $\gammastarlpf(\aa)$ by solving an LP, and calculates the loss  $\Losslpf(\xlpf(\aa, \phistarlpf(\aa)))$. Finally, it returns the solution corresponding to the maximum loss. Similar logic applies for optimal solution of $\adupf$. 
\end{proof}

We, now, describe an iterative greedy approach to compute the solution to $\adnpf$  that uses the optimal attacker strategy for fixed defender response (refer \Cref{algo:greedyAlgorithm}). 

\Cref{algo:greedyApproach} initializes $\phi_c$ to the optimal defender response under no attack. In the first step of the iterative approach, the attacker assumes some  defender response $\phi_c$ to be fixed, and computes the optimal attack strategy $\aa_c(\phi_c)$ using the greedy  \Cref{algo:greedyAlgorithm}. Then in the second step, the defender computes a new defense strategy $\phi_c$ optimal for fixed $\aa_c$ by solving the SOCP, and updates the defender response. If $\Loss(\xnpf(\aa_c,\phi_c)) > \Loss(\xnpf(\aastar,\phistarnpf))$, then the current best solution $(\aastar,\phistarnpf)$ is updated to $(\aa_c,\phi_c)$. Then in the next iteration, the attacker uses this new defender response to update his attack strategy, and so on and so forth. If this $\aa_c$ has already been discovered in some previous iteration, the algorithm terminates successfully, with $\aastar,\phistarnpf$ as the required optimal attack plan, and the corresponding optimal defense. The algorithm terminates unsuccessfully if the number of iterations exceeds a maximum limit.


\def \ds {\Upsilon}


\begin{algorithm}[h]
\small
\begin{algorithmic}[1]
\def \lstar {\l^\star}
\def \phistar {\phi^\star}
\def \aac {\aa_c}
\def \phic {\phi_c}
\State $(\aastar, \phistar, \Lnpf) \gets $ \Call{Greedy-Iterative}{\null}
\Procedure{Greedy-Iterative}{}
\State Let $\aastar \gets \textbf{0}, \lstar  \gets  0, \aac \gets \textbf{0}, iter \gets  0, \ds  \gets  \emptyset, \phic, \phistar, \ds $
\State \label{state:computeDefenderResponse1} For $\aa = \aac$, compute $\phistarnpf$ by solving SOCP $\adspnpf$  (\cref{lem:convexRelaxationOptimalSolution})
\State $\phic \leftarrow \phistarnpf, \Lnpf^\star  \leftarrow \Loss(\xcpf(\aa,\phistarnpf))$

\For {$iter  \gets  0,1,\dots,maxIter$}
	
	\State $\aac  \gets $ \Call{OptimalAttackForFixedResponse}{$\phi_c$}
		
		\State \hspace{-0.1cm}\myComment{If  $\aac$ previously found, successfully terminate}	
		\State \textbf{if} $\aac \in \ds$ \textbf{then return } $\aastar, \phistar$   
		\State \textbf{else }$\ds = \ds \cup \{\aac \}$ \hspace{-0.6cm} \myComment {\footnotesize  Store the current best attack vector}
		
	\State \label{state:computeDefenderResponse2} Compute $\phic$ by solving SOCP $\adspnpf$ \cref{lem:convexRelaxationOptimalSolution}
	
	\If {$\Loss(\xcpf(\aac,\phic)) > \lstar$ }
		\State $\aastar  \gets  \aac, \phistar  \gets  \phic, \Lnpf^\star  \gets  \Loss(\xcpf(\aa,\phistarnpf))$
	\EndIf
	
\EndFor \hspace{1cm} \myComment { Maximum Iteration Limit reached}

\State Return $\aastar, \phistar, \Lnpf^\star$ \hspace{0.3cm} \myComment { Return the last best solution}

\EndProcedure \hspace{0.7cm} \myComment { Algo terminates unsuccessfully}
\end{algorithmic}
\caption{Iterative Algorithm for Greedy Approach}\label{algo:greedyApproach}
\normalsize
\end{algorithm}

Note that in each iteration, the size of $\ds$ increases by 1, hence, the algorithm is bound to terminate after exhausting all possible attack vectors. 

\Cref{prop:generalApproach} and \Cref{thm:greedyApproach} can be applied for  any $\ad \in \Dd$, since if the DN has identical $\rxratio$ ratio, $\sgsetd$ are also fixed.  

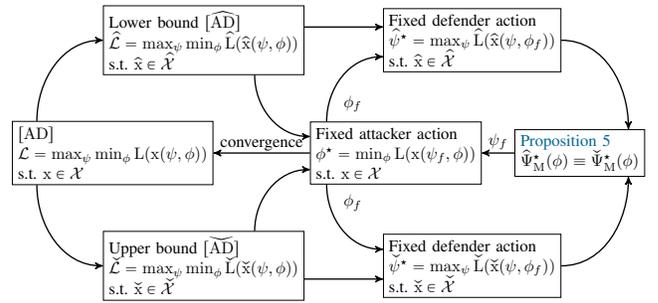
\begin{figure}[h]
\centering
\resizebox{8.5cm}{!}{
\begin{tikzpicture}
	\tikzstyle{every node}=[font=\normalsize]
	\tikzstyle{fde}=[circle,draw]
	\tikzset{
	    >=stealth',
	    punkt/.style={
	           rectangle,
	           rounded corners,
	           draw=black, very thick,
	           text width=6.5em,
	           minimum height=2em,
	           text centered},
	    pil/.style={
	           ->,
	           thick,
}
	}
	
\def\myXshift{4.5em}
\def\myXXXshift{4em}
\def\myXXshift{-3.5em}
\def\myYshift{2.2em}

\node[text depth = 3cm,minimum width=8.5cm,font=\normalsize,align=left] (main){};

\node[draw,align=left](lb) at ([xshift=\myXshift,yshift=\myYshift]main.north west){ Lower bound $\adlpf$\\
$\Llpf = \max_\psi\min_\phi\Losslpf(\xlpf(\psi,\phi))$\\
s.t. $\xlpf \in \Xlpf$ };

\node[draw,align=left](fdalb) at ([xshift=\myXXshift,yshift=\myYshift]main.north east){Fixed defender action\\
$\psistarlpf = \max_\psi\Losslpf(\xlpf(\psi,\phi_f))$\\
s.t. $ \xlpf \in \Xlpf $ };
 
 \node[draw,align=left](ub) at ([xshift=\myXshift,yshift=-\myYshift]main.south west){ Upper bound $\adupf$\\
 $\Lupf = \max_\psi\min_\phi\Lossupf(\xupf(\psi,\phi))$\\
 s.t. $\xupf \in \Xupf$ };
 
 \node[draw,align=left](fdaub) at ([xshift=\myXXshift,yshift=-\myYshift]main.south east){Fixed defender action\\
 $\psistarupf = \max_\psi\Lossupf(\xupf(\psi,\phi_f))$\\
 s.t. $ \xupf \in \Xupf $ };
 
  \node[draw,align=left](org) at ([xshift=-1em,yshift=0em]main.west){$\adnpf$ \\ 
  $\Lnpf = \max_\psi\min_\phi\Loss(\xnpf(\psi,\phi))$\\
  s.t. $\xnpf \in \Xnpf$ };
  
  \node[draw,align=left](faa) at ([xshift=\myXXXshift,yshift=0em]main.center){Fixed attacker action\\
$\phistarnpf = \min_\phi\Loss(\xnpf(\psi_f,\phi))$\\
s.t. $\xnpf \in \Xnpf$ };
  
    \node[draw,align=left](result) at ([xshift=3em,yshift=0em]main.east){
 \cref{prop:generalApproach}\\
$\Psistarlpf(\phi) \equiv \Psistarupf(\phi)$  };
 
 \node[align=left](phi1) at ([xshift=1.4em,yshift=3em]main.center){$\phi_f$};
 
  \node[align=left](phi2) at ([xshift=1.4em,yshift=-3em]main.center){$\phi_f$};

\draw[pil] (org.north west)+(1.5em,0) to [pil,out=90,in=180] (lb.west);
\draw[pil] (org.south west)+(1.5em,0) to [pil,out=270,in=180] (ub.west);

\draw[pil] ([yshift=0.8em]lb.east) to [] ([yshift=0.8em]fdalb.west);
\draw[pil] ([yshift=-0.8em]ub.east) to [] ([yshift=-0.8em]fdaub.west);

\draw[pil] (faa) -> (org) node[midway,yshift=0.5em]{convergence};
\draw[pil] (result) -> (faa) node[midway,yshift=0.7em]{$\psi_f$};

\draw[pil] ([yshift=0em]fdalb.east) to [out=0,in=90] ([xshift=-1em]result.north east);
\draw[pil] ([yshift=0em]fdaub.east) to [out=0,in=270] ([xshift=-1em]result.south east);

\draw[pil] ([xshift=-3em]lb.south east) to [out=270,in=180] ([yshift=1em]faa.west);
\draw[pil] ([xshift=-3em]ub.north east) to [out=90,in=180] ([yshift=-1em]faa.west);

\draw[pil] ([xshift=1em]faa.north west) to [out=90,in=180] ([yshift=-1em]fdalb.west);
\draw[pil] ([xshift=1em]faa.south west) to [out=270,in=180] ([yshift=1em]fdaub.west) ;

\end{tikzpicture}}
\caption{Overall computational approach.} \label{fig:generalApproach}
\end{figure}
  

 Our overall computational approach to solving the problem $\adnpf$, thus far, can be summarized as in \cref{fig:generalApproach}. Given an instance of the problem $\adnpf$, we first solve the problems $\adlpf$ and $\adupf$. For this, we employ an iterative procedure that iterates between the master- and sub- problems. For a fixed attacker action we determine the optimal defender response $\phi$ for the $\adspnpf$ using the convex relaxation of \eqref{eq:currentMagnitudeNpf}. Then, for the fixed defender response $\phi$, we compute the optimal attacker strategies $\psistarlpf$ and $\psistarupf$ by solving $\admplpf$ and $\admpupf$, respectively. \cref{prop:generalApproach} provides us an useful result that $\psistarnpf(\phi) \coloneqq \psistarlpf = \psistarupf$. This optimal attacker strategy $\psistarnpf(\phi)$ is then fed back to the master- problem $\admpnpf$. This procedure is repeated until we reach a convergence or we exceed the maximum iteration limit. 

\section{Securing DERs To Worst-Case Attacks}
\label{sec:securityDesign}

In this section, we consider the defender problem of optimal security investment in Stage 1. For simplicity, we restrict our attention to DNs that satisfy the following assumption: 
\begin{assumption}\label{assm:balancedTree}
\textbf{Symmetric Network.} For every $i\in\N$, for any two nodes $j,k\in\child_i$, $\Lambda_j$ and $\Lambda_k$ are symmetrically identical about node~$i$. That is, $z_j = z_k$, $\abs{\child_j} = \abs{\child_k}$, $\scdem_j = \scdem_k$, $\nuMin_j = \nuMin_k$, $\W_j = \W_k$, and $\c_j = \c_k$. However, all the DERs are homogeneous, i.e., $\forall\;j,k\in\N,\; \sgmax_j = \sgmax_k$. 
\end{assumption}

\def \ns {n_s}
\def \adone {\ad^1}
\def \adtwo {\ad^2}

Let $B$ be a fixed security budget. Let $\ad, \adnew \in \Dd$, $\ad \ne \adnew$, be two security strategies. Strategy $\ad$ is \emph{more secure} than strategy $\adnew$ (denoted by $\ad \securePreceq \addash$) under NPF (resp. LPF), if $\Lunpf \le \Lunpfnew$ (resp. $\Lulpf \le \Lulpfnew$). Finally, we  ask what is the best security strategy $\adstar$, such that for $\ad = \adstar$, $\Lunpf$ is minimized. \cref{fig:designs} shows two possible security strategies $\adone$ (\cref{fig:design1}) and $\adtwo$ (\cref{fig:design2}), and gives a generic security strategy (\cref{fig:design2Tree}). If we compare $\adone$ and $\adtwo$, while transitioning from $\adone$ to strategy $\adtwo$, 3 secure nodes in $\Lambda_2$ subtree go up a level each, while 3 secure nodes in $\Lambda_3$ subtree go down a level each. Then, between $\adone$ and $\adtwo$, which strategy is more secure? In this section, we provide insights about optimal security strategies 
under \assmref{assm:balancedTree}, which help show that $\adtwo$   is more secure than $\adone$.

\def \hh {2cm} 
\begin{figure}[h!]
\centering 
\subfloat[Security strategy $\adone$. $\Ns(\adone) = \{3,5,6,7,10,11\}$.]{
\includegraphics[width=3.5cm,height=\hh]{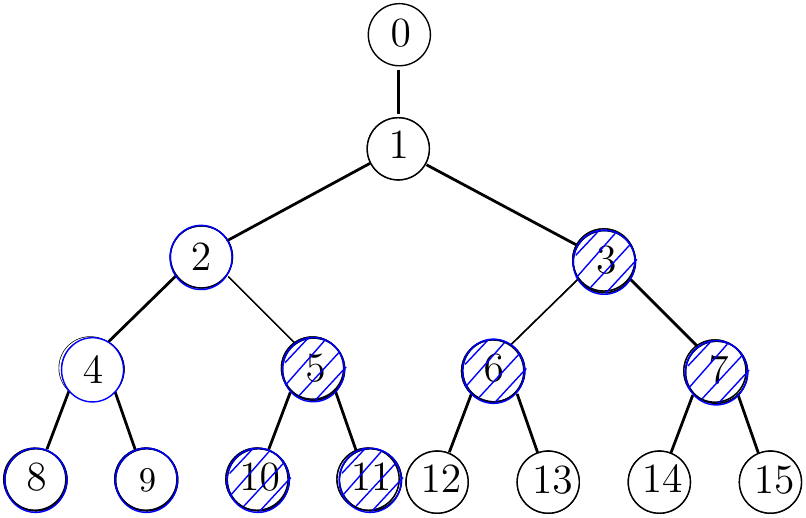}\label{fig:design1}
}\hspace{0.5cm}
\subfloat[Security strategy $\adtwo$. $\Ns(\adtwo) = \{2,4,5,6,12,14\}$.]{
\includegraphics[width=3.5cm,height=\hh]{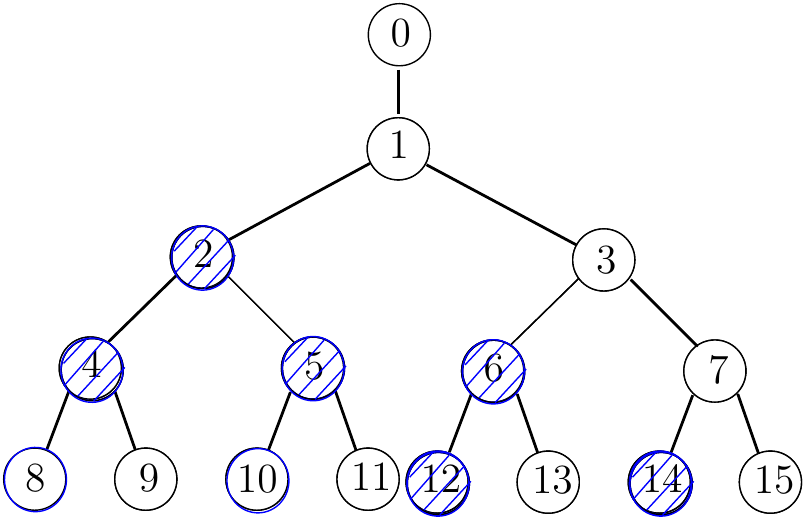}\label{fig:design2}
}\\
\subfloat[A generic security strategy on a tree DN.]{\includegraphics[width=6cm]{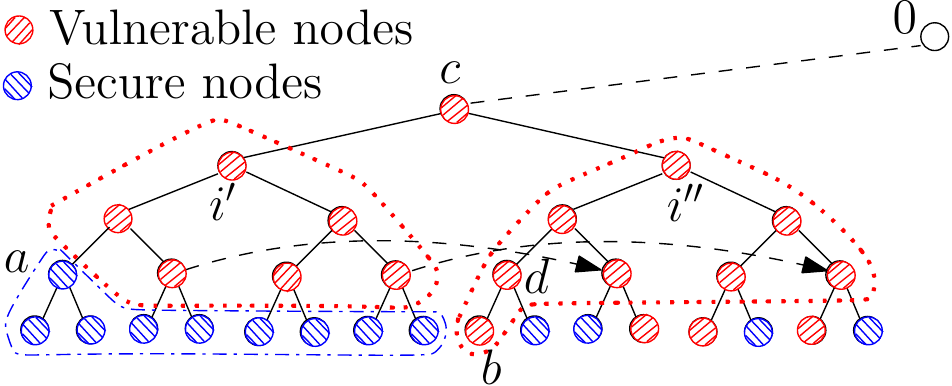}\label{fig:design2Tree}}
\caption{Different defender security strategies. } 
\label{fig:designs}
\end{figure} 

 \Cref{algo:securityAlgorithm} computes an optimal security strategy $\dadlpf$ under \assmref{assm:throughout}-\assmref{assm:balancedTree}. It initially assigns all nodes to be vulnerable. Then, DER nodes are secured sequentially in a bottom-up manner towards the root node. If the security budget is not adequate to secure a full level, the nodes in that level are uniformly secured and the remaining nodes are not secured. \tcb{Under all the assumptions of \Cref{algo:securityAlgorithm}, it takes $\O(n)$ time. }

\def \nnh {\alpha} 
\begin{algorithm}[h!]

\begin{algorithmic}[1]
\small
\State $\adstarlpf \gets $\Call{OptimalSecurityStrategy}{\null}
\Procedure{OptimalSecurityStrategy}{\null}

\def \step {step}
\State $\ns \gets 0, \level \gets \Pmax$, $\ulpf \gets \textbf{0}$ \myComment{Initialize all nodes to vulnerable nodes}
\State For each $\level\in [1,2,\dots,\Pmax]$, let $\nnh_\level \gets \ssum_{j = \level}^\Pmax \abs{\N_j}$ 
\State Let $\level' \gets \argmax_{h\in [1,\dots,\Pmax]:\nnh_\level\ge M} h$
\State Let $\forall\;\level\in[\level',\dots,\Pmax],\forall\;i\in\N_\level$, $\ulpf_i \gets 1$.  
\State Let $\N_{\level'}' \subseteq \N_{\level'}$ be a set of uniformly chosen $M-\nnh_{\level'+1}$ nodes on level $\level'$. 
\State For each $i\in\N_{\level'}'$, $\ulpf_i \gets 1$ 
%
%
\State \Return $\ulpf$
\EndProcedure
\end{algorithmic}
\normalsize
\caption{Optimal security strategy}
\label{algo:securityAlgorithm}
\end{algorithm}

In the following theorem, we show that the security strategy computed by \Cref{algo:securityAlgorithm} is an optimal solution to the Stage 1 of the $\dadlpf$ and $\dadupf$ problem. 
\begin{theorem}\label{thm:optimalSecureDesign}
Assume \assmref{assm:throughout},  \assmref{assm:identicalRX}, \assmref{assm:balancedTree}. Let  $\adstarlpf$ be the security strategy computed by \cref{algo:securityAlgorithm}. Furthermore, with $\ad = \adstarlpf$, let $(\psistarlpf,\phistarlpf)$ be the solution computed by \cref{algo:greedyAlgorithm2}. Then,  $(\adstarlpf,\psistarlpf,\phistarlpf)$ is an optimal solution to $\dadlpf$. Similar result holds for $\dadupf$.  
\end{theorem}

Finally, we state the following result:
\begin{proposition} 
1) Under  \assmref{assm:throughout}, \assmref{assm:identicalRX}, \assmref{assm:balancedTree}, $\Dastarlpf$ can be partitioned into at most $\abs{\Nv} \le \NN$ equivalence classes of attack vectors, one for each vulnerable node considered as pivot node. Any two attack vectors in the same equivalence class has identical impact on the corresponding pivot node. Additionally, any two equivalence classes can be considered homomorphic transformations of each other.   

2) Under \assmref{assm:throughout}, \assmref{assm:identicalRX}, if $\forall\;i,j,k\in\N$ such that $\sgmax_j > 0$ and  $\sgmax_k > 0$, $\Delta_j(\nulpf_i) \ne \Delta_k(\nulpf_i)$, then $\abs{\Dastarlpf} = \abs{\Nv}$, i.e., if for any pivot node, no two DERs have identical impact on the pivot node due to their individual DER compromises, then each equivalence class is a singleton set, and hence, the set for candidate optimal attack vectors is at most of size $\abs{\Nv}$. 
	
\label{prop:complexityOfDadlpf}
\end{proposition} 
%

By \cref{thm:optimalSecureDesign}, we can compute the optimal security investment $\aastarlpf$ in $\mathcal{O}(\NN)$, and by \cref{prop:complexityOfDadlpf}, for fixed $\aastarlpf$, we can compute the optimal attacker strategy $\psistarlpf$ in $\mathcal{O}(\NN)$. Finally, for fixed $\aastarlpf$ and $\psistarlpf$, we can compute the optimal defender response $\phistarlpf$ in $\mathcal{O}(poly(\NN))$. Hence, we can compute the optimal solution for $\dadlpf$, in $\mathcal{O}(poly(\NN))$. Same holds for $\dadupf$.

Admittedly, our structural results on optimal security investment in Stage~$1$ of the game are specific to assumption~\assmref{assm:balancedTree}. Future work involves extending these results to a general radial DN with heterogeneous DER nodes. A key aspect in effort will be to understand how the defender's net value of securing an individual DER node depends on its capacity and location in the DN. 

\section{Computational Study}
\label{sec:caseStudy}

We describe a set of computational experiments to evaluate the performance of the iterative Greedy Approach (GA) in solving $\adnpf$; see \Cref{algo:greedyApproach}. We again assume $\ad = \textbf{0}$. We compare the optimal attack strategies and optimal defender set-points obtained from GA with the corresponding solutions obtained by conducting an exhaustive search (or Brute Force (BF)), and by implementing the Benders Cut (BC) algorithm. We refer the reader to \cite{baldick}, \cite{kevinwood}, for the BC algorithm adopted here.
The abbreviations BC-LPF and BC-NPF denote the solutions obtained by applying optimal attack strategies from $\adlpf$ to LPF and NPF, respectively. Importantly, the experiments illustrate the impact of attacker's resource ($\arcm$) and defender's load control capability $\gammaMin$ on the optimal value of $\adnpf$. The code for this computational study can be obtained by contacting the authors. 

\subsubsection*{Network Description}$\ $ Our prototypical DN is a modified IEEE 37-node network; see \cref{fig:IEEE37NodeNetwork1}. We consider two variants of this network: homogeneous and heterogeneous. \textbf{Homogeneous Network ($\ghm$)} has 14 homogeneous DERs with randomly assigned node locations, loads with equal nominal demand, and lines with identical $\rxratio$ ratio. \ifparameterTable \Cref{tab:parameterTable} (in appendix) lists the parameter values of $\ghm$. \else Each line has impedance of $z_j = (0.33 + 0.38\j)\;\Omega$. The nominal demand at each node $i$ is $\scdem_i = 15\;kW+\j4.5\;kvar$. The apparent power capability of each DER node $i$ is $\sgmax_i = 11.55\;kVA$. The nominal voltage at node $0$ is $\abs{V_0} = 4\;kV$. The cost of load control is $C = 7\;\$$ per $kW$. \fi 
\textbf{Heterogeneous Network ($\ght$)} has same topology as $\ghm$, but has heterogeneous DERs (chosen at random from 3 different DER apparent power capabilities), heterogeneous loads, and lines with different $\rxratio$ ratios. The locations of DER nodes, the total nominal generation capacity, and the total nominal demand in $\ght$ is roughly similar to the corresponding values for $\ghm$.  

\noindent \textbf{DER output vs $\arcm$.} \cref{fig:37pgqg} compares the DER output ($\sg$)  of uncompromised DERs that form part of defender response in $\ghm$ and $\ght$ for different $\arcm$. When $\arcm = 0$ (no attack), there are no voltage violations, and the defender minimizes $\lloss$, which results in $\pg > \qg$. For $\arcm > 0$, the voltage bounds may be violated. To limit $\llovr$, the defender responds by increasing $\qg$; and the output of uncompromised DERs lie in a neighborhood of $\theta = \arccot \rxratio$. For the case of $\ght$ (\cref{fig:37heteropgqg}), the set-points of the uncompromised DERs are more spread out to achieve voltage regulation over different $\rxratio$ ratios (\cref{prop:pvSetpoints}). In \cref{fig:37heteropgqg} the three semi-circles correspond to the uncompromised DERs with different apparent power capabilities. These observations on the defender response validate \cref{prop:pvSetpoints}. 
  
\def\hssss{\hspace{-0.2cm}}
\def\hs{\hspace{-0.65cm}}
\def\ww{4cm}
\def\hh{4cm}
\begin{figure}[h]
\centering 
\subfloat[Homogeneous network.]{
\includegraphics[width=\ww,height=\hh]{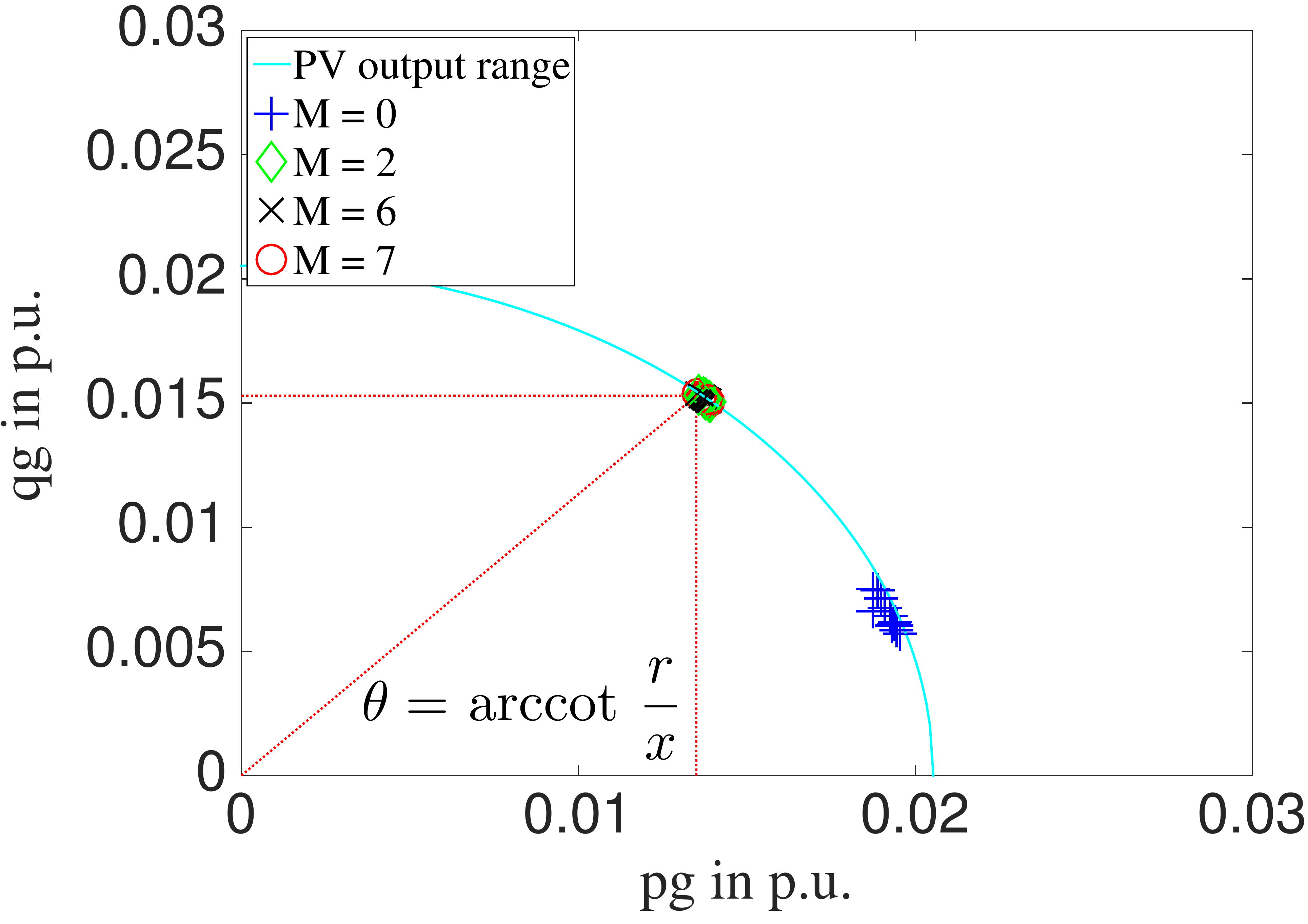}
\label{fig:37homopgqg}}\hfill
\subfloat[Heterogeneous network.]{
\includegraphics[width=\ww,height=\hh]{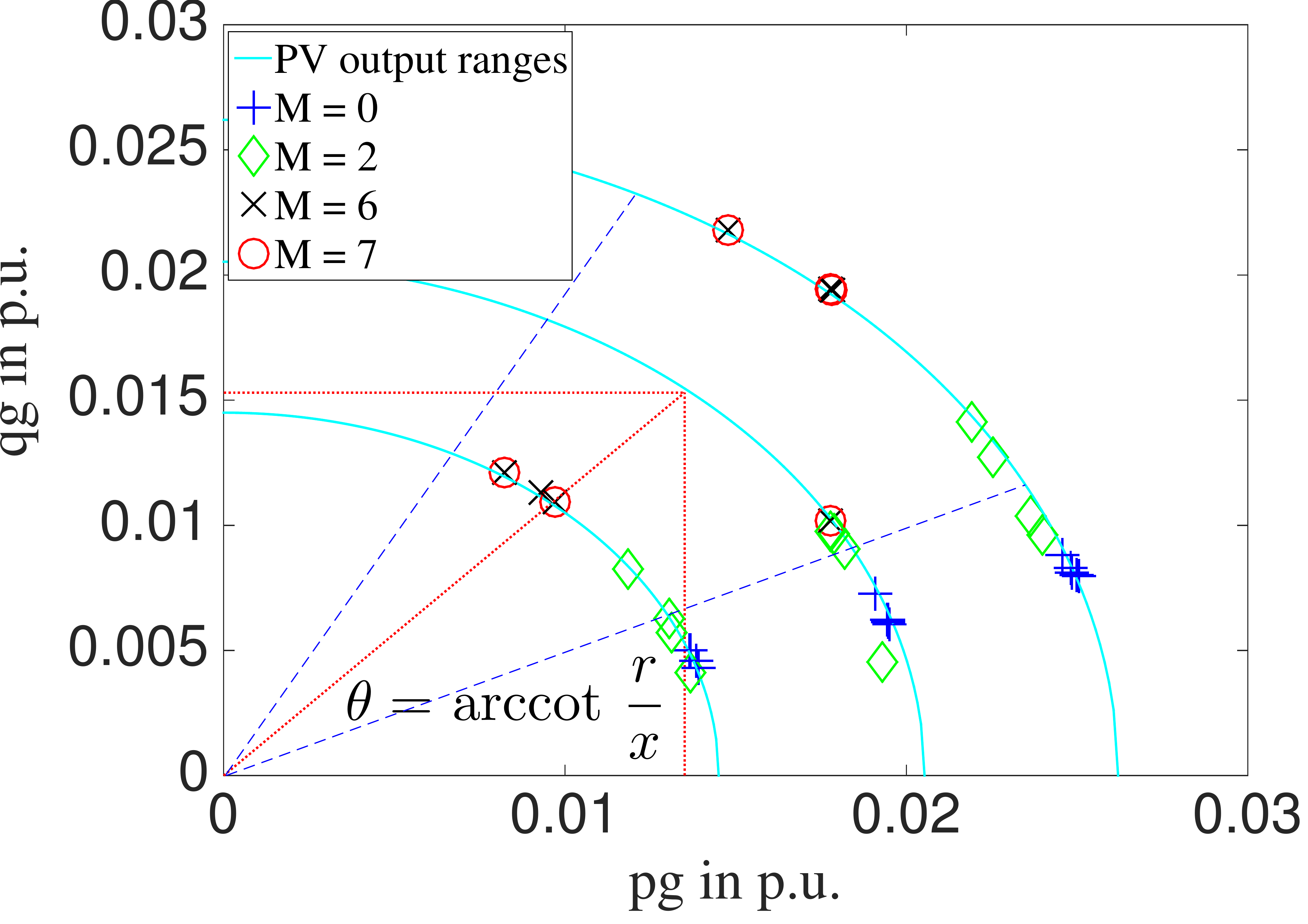}
\label{fig:37heteropgqg}}
\caption{Reactive power vs Real power output of DERs.}
\label{fig:37pgqg}
\end{figure}

\noindent \textbf{GA \textit{vs.} BC-NPF, BC-LPF and BF.} \cref{fig:homogenous} compares results obtained from~\emph{BC-NPF},~\emph{GA}, and~\emph{BF} on $\ghm$. We consider two cases with the maximum controllable load percentage $\gammaMin=50\;\%$ and $\gammaMin=70\;\%$. For each case, we vary $\arcm$ from $0$ to $\abs{\Nv} = 14$; and also vary $\wcratio$ ratios to capture the effect of different weights on the terms  $\llovr$ and $\lvoll$. 
 
\tcb{In our study, we chose $C_i= 7\;cents/kWh$, converted appropriately to the per unit system. \footnote{\tcb{From a practical viewpoint, the weights $C_i$ can be obtained from the operator's rate compensation scheme for load control. For example, North Star Electric \cite{northStar} provides a compensation of $9.1 \; cents$ to their customers for $1\; kWh$ of load curtailment. One can argue that the net cost of shedding unit load should be adjusted to reflect the fact that the defender supplies additional power during the attack to meet the consumers' demand. }} The ratio $\wcratio = 2$ roughly corresponds to the maximum $\sfrac{W}{C}$ ratio for which the defender does not exercise load control, because the cost of doing load control is too high, i.e., at optimum defender response $\gammastarnpf = \unity$. In contrast, $\wcratio = 18$  roughly corresponds to the minimum $\sfrac{W}{C}$ ratio for which the defender exercises maximum load control (i.e. $\gammastarnpf = \gammaMin$). We also consider an intermediate ratio, $\wcratio = 10$. }


\def\ww{4.3cm}
\def\hh{4.3cm}
\begin{figure}
\subfloat[$\llovr$ vs $\arcm,\quad\gammaMin = 0.5$.]{
\includegraphics[width=\ww,height=\hh]{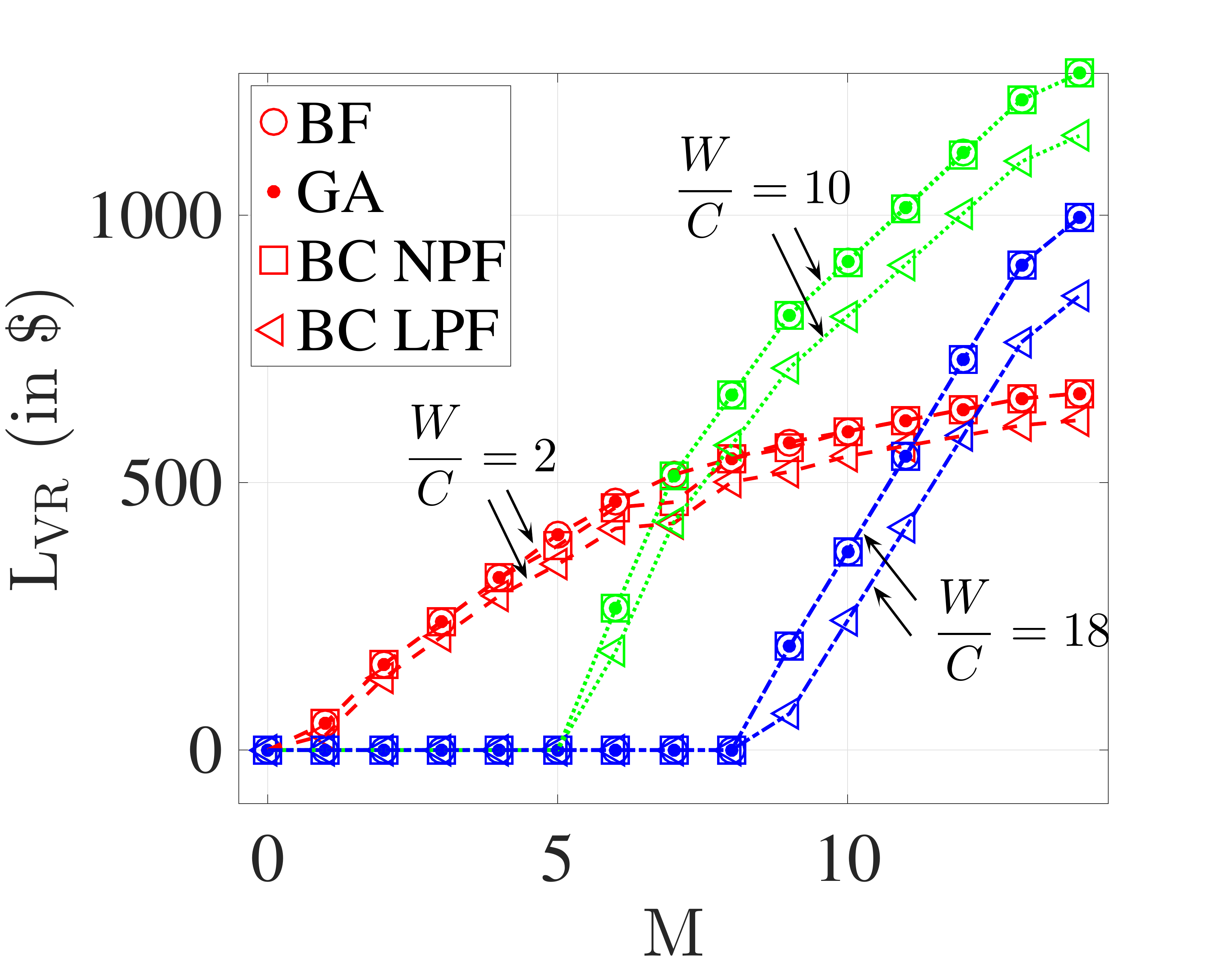}
\label{fig:lovrVsDelta_37_50}
}
\subfloat[$\llovr$ vs $\arcm,\quad\gammaMin = 0.7$.]{
\includegraphics[width=\ww,height=\hh]{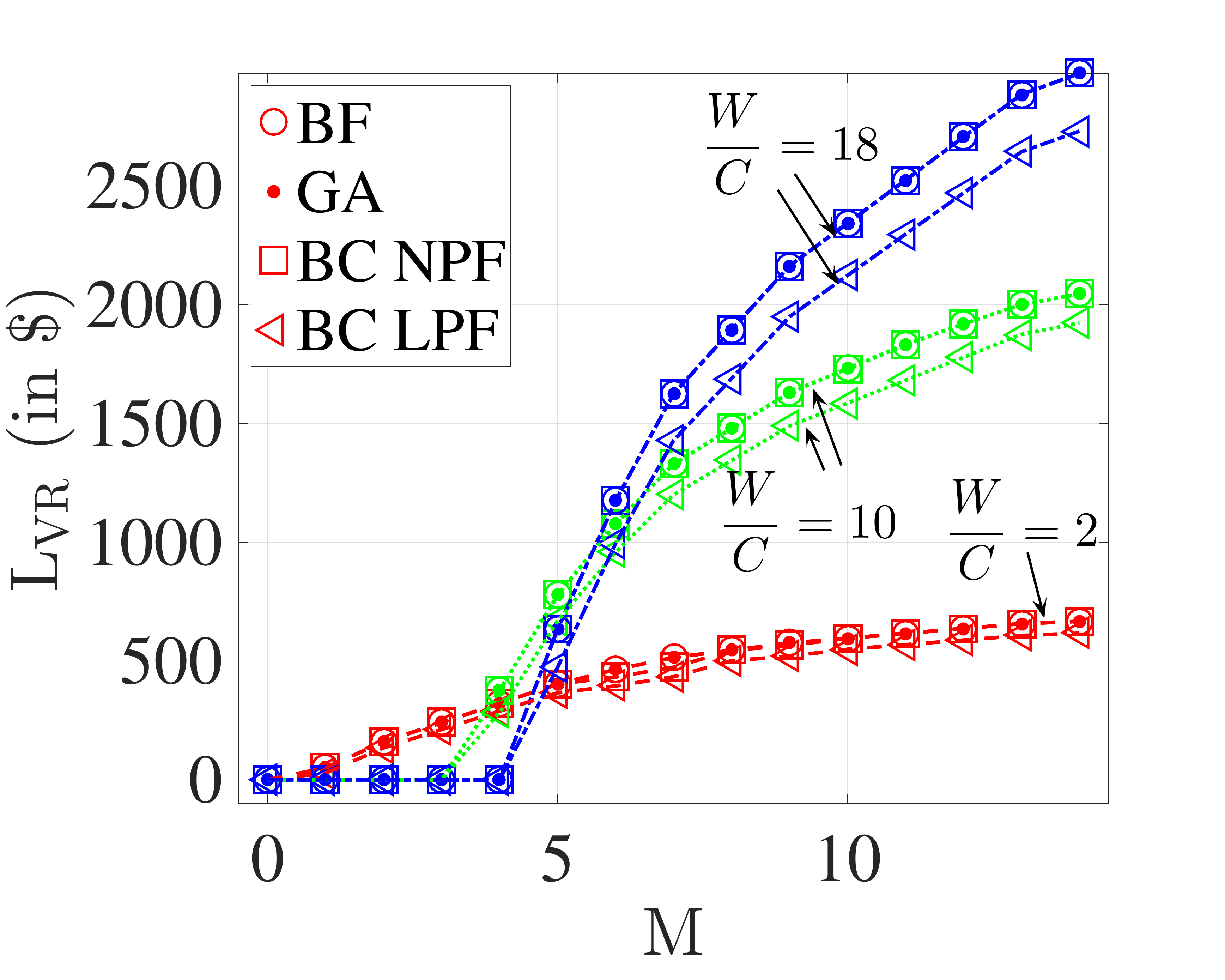}
\label{fig:lovrVsDelta_37_70}
}\\
\subfloat[$\lvoll$ vs $\arcm,\quad\gammaMin = 0.5$.]{
\includegraphics[width=\ww,height=\hh]{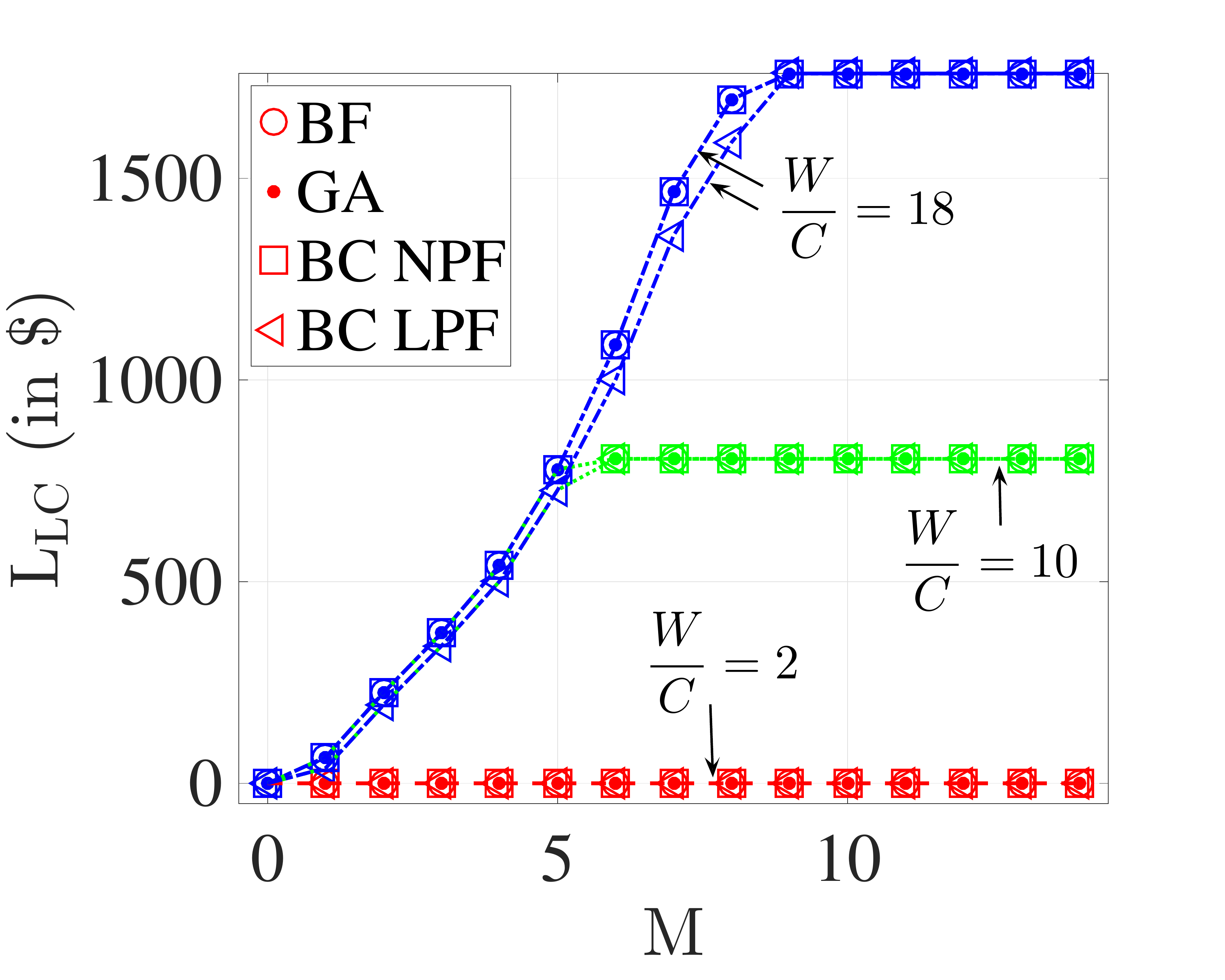}
\label{fig:vollVsDelta_37_50}
}
\subfloat[$\lvoll$ vs $\arcm,\quad\gammaMin = 0.7$.]{
\includegraphics[width=\ww,height=\hh]{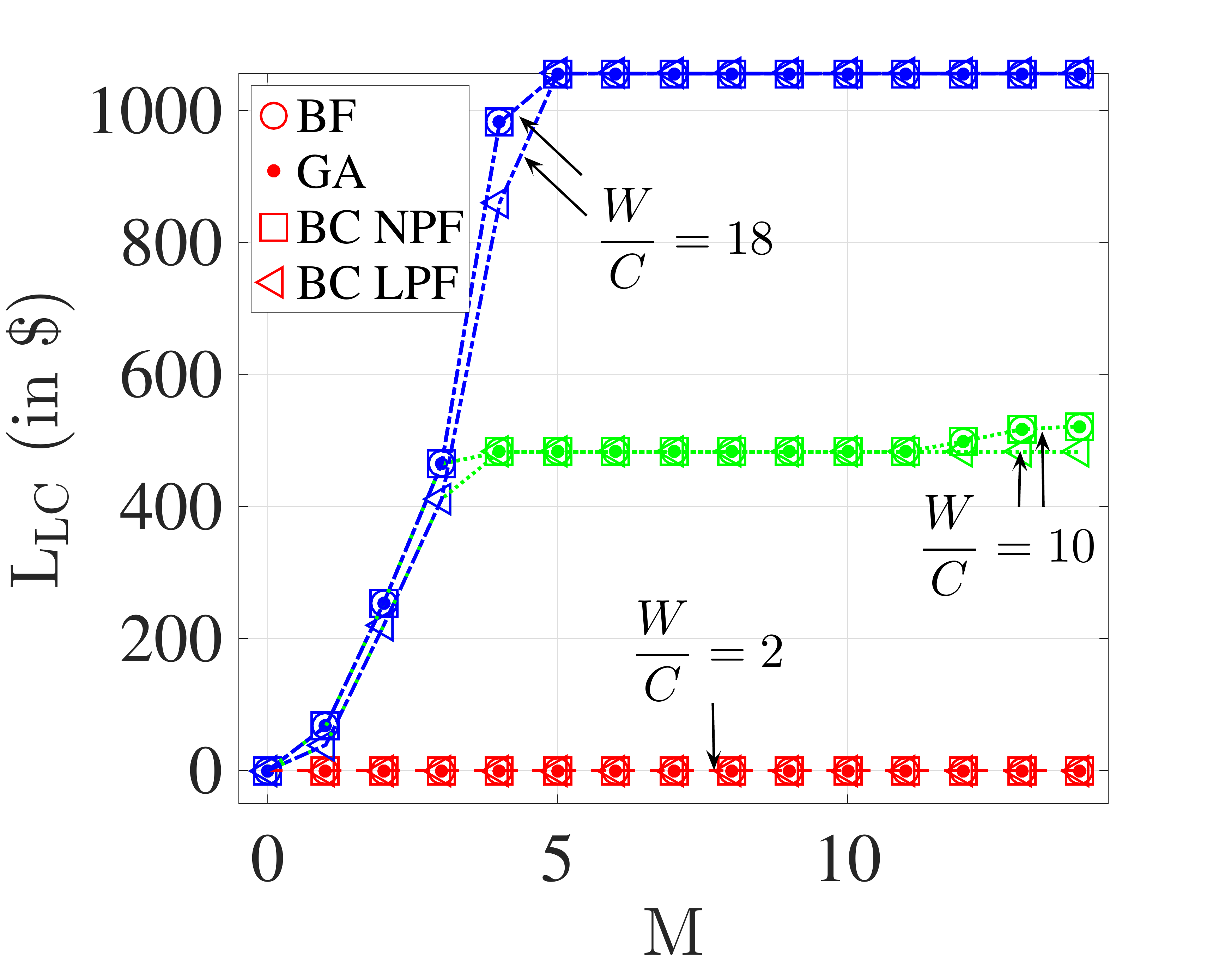}
\label{fig:vollVsDelta_37_70}
}
\caption{$\llovr$ and $\lvoll$ vs $M$ for $\ghm$. The results of $\ght$ are more or less similar to those of $\ghm$. }\label{fig:homogenous}
\end{figure}

 
\noindent $\mathbf{\Loss}$ \textit{versus} $\mathbf{\arcm}$. Both $\llovr$ and $\lvoll$ are zero when there is no attack. As $\arcm$ increases, one or both $\llovr$ and $\lvoll$ start increasing. This indicates that as more DERs are compromised, the defender incurs $\llovr$, and in addition, he imposes load control to better regulate the DN. Indeed, after the false set-points (\Cref{thm:attackerSetpoints}) are used to compromise  DERs, the net load in the DN increases. Without load control, the voltages at some nodes drop below the lower bounds, increasing $\llovr$. Hence, the defender exercises  load control, and changes the set-points of uncompromised DERs to limit the total loss. 

Perhaps a more interesting observation is that as $\arcm$ increases,  $\lvoll$ first increases rapidly but then flattens out (\cref{fig:vollVsDelta_37_50,fig:vollVsDelta_37_70}). This can be explained as follows: depending on the $\sfrac{W}{C}$ ratio, there is a subset of downstream loads that are beneficial in terms of the value that the defender can obtain by controlling them. That is, if the loads belonging to this subset are  controlled, the decrease in $\llovr$ outweighs the increase in $\lvoll$, hence, the defender imposes load control on these downstream loads to reduce the the total loss. In contrast, controlling the loads outside this subset, increases $\lvoll$ more than the decrease in $\llovr$. Hence, the defender satisfies the demand at these loads fully. The $\lvoll$ increases  until load control capability in the subset of  beneficial downstream loads to the defender is fully exhausted. The size of this subset depends on the $\sfrac{W}{C}$ ratio. The higher the ratio, the larger the size of the subset of the loads beneficial to the defender. Hence, the value of $\arcm$, at which the $\lvoll$ cost curve flattens out, increases as the $\sfrac{W}{C}$ ratio increases.  

The cost curve for $\llovr$ also shows interesting behavior as the number of compromised DER nodes increases (\cref{fig:lovrVsDelta_37_50,fig:lovrVsDelta_37_70}). The marginal increase in $\llovr$ for every additional DER compromised reduces as $\arcm$ increases. This observation can be explained by the fact that the attacker prefers to compromise downstream nodes over upstream ones (\cref{prop:downstream}).  Initially, the attacker is able to rapidly increase $\Loss$ by compromising more beneficial downstream nodes. However, as the downstream nodes are eventually exhausted, the attacker has to target the relatively less beneficial upstream nodes. Hence, the reduction in marginal increase of $\llovr$. 

In $\llovr$ plots, for small $\arcm$, $\wcratio = 2$ curves are lower than the $\wcratio=10$ curves which in turn are lower than the $\wcratio=18$ curves. But, for larger $\arcm$, this order reverses. The $\arcm$ where these lines cross each other decreases, as the $\gammaMin$ increases (see \cref{fig:lovrVsDelta_37_50,fig:lovrVsDelta_37_70}). The reason is for some intermediate value of $\arcm$, the defender exhausts the load control completely, and then the $\Loss$ increases at rates in the same order of increasing $\wcratio$ values.  

Our computational study also validates that the GA is more efficient than BC method because GA calculates the exact impact the DER compromises will have on a pivot node. In contrast, BC overestimates the impact of DER compromises that are not the ancestors to the pivot nodes. Therefore, the feasible region probed by BC at every iteration is larger than the feasible region probed in the corresponding iteration of GA. Hence, although GA converges to a solution in 2-3 iterations, BC in most cases does not converge to the optimal solution even in 200 iterations.

\section{Concluding Remarks}
\label{sec:conclusion}

We focused on the security assessment of radial DNs for an adversary model in which multiple DERs (in this case, DER nodes) are compromised. The adversary can be a threat agent, who can compromise the operation of DERs, or a malicious insider in the control center. We considered a composite loss function that primarily accounts for the attacker's impact on voltage regulation and induced load control. The security assessment problem is formulated as a three-stage Defender-Attacker-Defender ($\dadnpf$) sequential game. Our main technical contributions include: (i) Approximating the $\dadnpf$ game that has nonlinear power flow model and mixed-integer decision variables with tractable formulations based on linear power flow; and (ii) characterization of structural properties of security investments in Stage 1 and the optimal attack in Stage 2 (i.e., the choice of DER node locations and the choice of false set-points).
 
Future work includes: (a) Extending \Cref{thm:lpfNpfUpfRelationship,thm:attackerSetpoints} to cases where reverse power flows are permissible (e.g., when the DN is not under heavy loading conditions and the attacker can cause DER generation to exceed the demand); (b) Designing greedy algorithm to solve [AD] and proving optimality guarantees of \Cref{thm:greedyApproach,thm:optimalSecureDesign} for DNs with heterogeneous $\rxratio$ ratio, and heterogeneous DERs or loads. 

\tcb{Finally, note that we do not consider cascading failures in our paper. 
However, our analysis can be extended to a form of cascading failures within DNs reported by Kundu and Hiskens~\cite{kunduOvervoltage}. They study synchronous tripping of the loads (specifically, plug-in electric vehicles chargers) leading to over-voltages in the DN. Our result on optimal DER attack can be used to create voltage violations at some nodes. If these violations are too high, certain loads may start to trip. After sufficiently large number of loads trip, the attacker can further manipulate the DER setpoints to their maximum power generation capacity. In the absence of new loads, this may lead to overvoltages, as described in~\cite{kunduOvervoltage}.}

\bibliographystyle{IEEEtran}
\bibliography{IEEEfull.bib}

\appendix


\ifArxivVersion
\begin{footnotesize}
\begin{table}
\centering
\begin{footnotesize}
\begin{tabular}{|c|c|}
\hline
Parameters & Values \\
\hline\hline
$r + \j x$ & $(0.33 + 0.38\j)\;\Omega$\\
\hline
$\pcdem_i $ & $15\;kW$\\
\hline
$\qcdem_i $ & $4.5\;kvar$\\
\hline
$\sgmax_i $ & $11.55\;kVA$\\
\hline
$\abs{V_0}$ & $4\;kV$\\
\hline
$C$ & $7\;\$$ per $kW$\\
\hline
\end{tabular}
\end{footnotesize}
\caption{Parameters of the Homogeneous Network}\label{tab:parameterTable}\normalsize
\end{table}
\end{footnotesize}
\fi 
\def \msmall {\normalsize}
\newenvironment{mysmall}[1]{\msmall #1}{}

\begin{mysmall}
For a pivot node~$i\in\N$, \cref{algo:helperProcedures} computes a sequence of sets of nodes in decreasing order of  $\Delta_j(\nulpf_i)$ values. This sequence is used to compute the optimal attacks that maximize voltage bounds violation at node~$i$. 

\begin{algorithm}[H]
\begin{algorithmic}[1] 
\begin{small}
\Procedure{OptimalAttackHelper}{$i$, $\sgsetdlpf$} 
 	\State For each $j\in\N$ compute $\Delta_j(\nulpf_i)$ using \Cref{lem:pvCompromiseEffects}
 	\State Create a sequence of sets  $\{\Ni_j \}_{j=1}^{\NN}$ such that \begin{enumerate}
\item[i)] $\N = \bigcup_{j=1}^\NN \Ni_j$, $\forall\;1\le j,k,\le \NN, \; \Ni_j \cap \Ni_k = \emptyset$
\item[ii)] if $1 \le l \le \NN,\;j,k\in\Ni_l$, then $\Delta_j(\nulpf_i) =  \Delta_k(\nulpf_i)$, and
\item[iii)] if $1 \le l < m \le \NN,\;j\in\Ni_l,k\in\Ni_m$, then $\Delta_j(\nulpf_i) > \Delta_k(\nulpf_i)$. 
 	\end{enumerate}
 	
 	\State Let, for $j\in[1,\dots,\NN], \mi_j \gets \abs{\Ni_j}$, $\Mi_j \coloneqq \sum_{k=1}^{j-1} \mi_k$. 
 	\State Let $\gi \gets  \argmin_{j\in[1,\dots,\NN], \Mi_j \ge M} \quad j$. 
 	\State $\J \gets \bigcup_{j=1}^{\gi-1}$, $\N_{\gi}$, $\m' = M - \Mi_{\gi-1}$ 
 	\State \Return  $\J$, $\N_{\gi}$, $\m'$
\EndProcedure

\end{small} 
\end{algorithmic}
\caption{Helper procedure}\label{algo:helperProcedures}
\end{algorithm} 
\end{mysmall}

\begin{table}[h]
\centering
\begin{small}
\def \Ilpf {\widehat{I}}
\def \Iupf {\widecheck{I}}
\def \Vlpf {\widehat{V}}
\def \Vupf {\widecheck{V}}
\def \xx {0.5cm}
\def \xxx {5.0cm}
\def \xxxx {7.0cm}
\def \xxxxx {0.5cm}
\def \xxxxxx {8.1cm}
\def \xxxxxxx {6.5cm}
\def \myRowHeight {[5ex]}
\setlength\minrowclearance{0.7pt}
\tcb{\begin{tabular}{|p{\xxxxx}|p{\xxx}|p{\xx}|p{\xxxxx}|}
\hline
\hline
$\j$ & \multicolumn{3}{|p{\xxxx}|}{$\j = \sqrt{ -1}$ complex square root of -1} \\
\hline
\hline
\multicolumn{4}{|c|}{Network parameters} \\
\hline
\hline
\ifArxivVersion 
$\N$ & \multicolumn{3}{|p{\xxxx}|}{set of nodes} \\
$\E$ & \multicolumn{3}{|p{\xxxx}|}{set of edges} \\
$\G$ & \multicolumn{3}{|p{\xxxx}|}{tree topology $\G = (\N,\E)$} \\
$r_j $ & \multicolumn{3}{|p{\xxxx}|}{resistance of line $(i,j) \in \E$}  \\
$x_j $ & \multicolumn{3}{|p{\xxxx}|}{reactance of line  $(i,j) \in \E$ }\\ \fi 
$z_j $ & \multicolumn{3}{|p{\xxxx}|}{impedance $z_j = r_j + \j x_j$ of line  $(i,j) \in \E$ }\\
$H $ & \multicolumn{3}{|p{\xxxx}|}{height of the tree }\\
\ifArxivVersion $\N_h $ & \multicolumn{3}{|p{\xxxx}|}{set of nodes on level $h \in 1,2,\cdots,H$}\\
$h_i $ & \multicolumn{3}{|p{\xxxx}|}{level of node $i$}\\
\ifArxivVersion $\child_i $ & \multicolumn{3}{|p{\xxxx}|}{set of children nodes of node $i$}\\ \fi 
$\Lambda_i$ & \multicolumn{3}{|p{\xxxx}|}{subtree rooted at node $i \in \N $}\\
$\Lambda_i^j $ & \multicolumn{3}{|p{\xxxx}|}{subtree rooted at node $i\in\N$ until level $h_j$ for $j\in\Lambda_i$}\\ \fi 
$\P_i$ & \multicolumn{3}{|p{\xxxx}|}{path from the root node to node $i$} \\
$Z_{ij}$ & \multicolumn{3}{|p{\xxxx}|}{$Z_{ij}\coloneqq \sum_{k\in\P_i\cap\P_j}z_k$ common path impedances between nodes $i$ and $j$} \\
\hline
\hline
\multicolumn{4}{|c|}{Power flow notations} \\
\hline
\hline
$\npf$ & Nodal quantities of node $i \in \N $  & $\lpf$ & $\UPF$ \\
\hline
$\scdem_i $ & complex power demand at node $i$ & $ - $ & $ - $ \\
$\sc_i $ & complex power consumed at node $i$ & $\sclpf_i $ & $\scupf_i $ \\
$\sg_i $ & complex power generated at node $i$ & $\sglpf_i $ & $\sgupf_i $ \\
$\sgset_i$ & complex power set-point of DER $i$ & $\sgsetlpf_i $ & $\sgsetupf_i $ \\
$V_i$ & complex voltage at node $i$ & $\Vlpf_i$ & $\Vupf_i$  \\
$\nu_i$ & square of voltage magnitude at node $i$ & $\nulpf_i$ & $\nuupf_i$  \\
$\nuMin_i,\nuMax_i$ & \emph{soft} lower and upper bounds on square of  voltage magnitude at node $i$ & &   \\
\hline
\hline
$\npf$ & Edge quantities of edge $(i, j) \in \E $  & $\lpf$ & $\UPF$ \\
\hline
$S_j $ & complex power flowing on line $(i,j)$ & $\Slpf_i $ & $\Supf_i $ \\
\ifArxivVersion $I_j$ & complex current flowing on line $(i,j)$ & $\Ilpf_i $ & $\Iupf_i $ \\ \fi 
$\ell_j $ & square of magnitude of current $I_j$ & $\elllpf_i $ & $\ellupf_i $ \\
$\xnpf $ & $\xnpf = (\nunpf,\ellnpf,\scnpf,\sgnpf,\Snpf)$ - state vector & $\xlpf $ & $\xupf $ \\
\hline
\hline
\multicolumn{4}{|c|}{Attacker model} \\
\hline
\hline
$\delta_i$ & $\delta_i = 1$ if DER $i$ is compromised & $- $ & $- $ \\
$\sgsetanpf_i$ & attacker set-point of DER $i$ & $\sgsetalpf_i $ & $\sgsetaupf_i $ \\
$\psi$ & $\psi\coloneqq (\sgsetanpf,\delta)$ attacker strategy & $\psilpf$ & $\psiupf$ \\
\hline
\hline
\multicolumn{4}{|c|}{Defender model} \\
\hline
\hline
$\gammaMin_i$ & max. allowed fraction of load control & $- $ & $-$ \\
$\gammanpf_i$ & fraction of load control at load $i$ & $\gammalpf_i $ & $\gammaupf_i $ \\
$\sgsetdnpf_i$ & defender set-point of DER $i$ & $\sgsetdlpf_i $ & $\sgsetdupf_i $ \\
$\phi$ & $\phi\coloneqq (\sgsetdnpf,\gamma)$ defender strategy & $\philpf$ & $\phiupf$ \\
\hline
\end{tabular}}
\caption{Table of Notations.} \label{tab:notationsTable}
\end{small}
\end{table}

\begin{IEEEbiography}[{\includegraphics[width=1in,height=1.25in,clip,keepaspectratio]{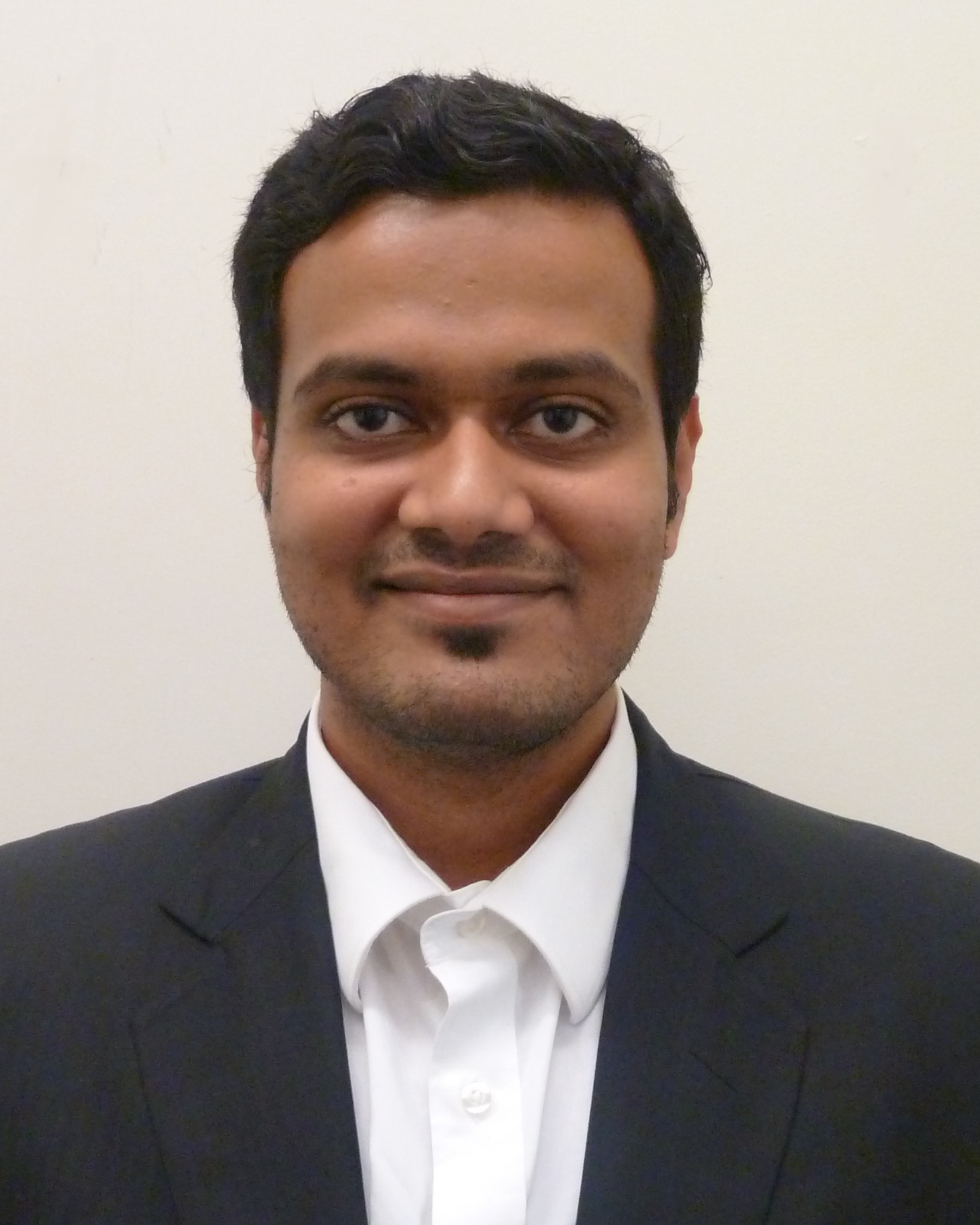}}]{Devendra Shelar}
is pursuing his Ph.D. at the Center for Computational Engineering (CCE), Massachusetts Institute of Technology (MIT). He is interested in developing tools using large-scale optimization and game theory to improve the resiliency of cyber-physical systems to failures. His current focus is on the secure and efficient operation of power systems with high penetration of distributed energy resources.  He  received his Dual B.Tech. \& M.Tech. in Computer Science and Engineering from the Indian Institute of Technology Bombay in 2012, M.S. in Transportation Engineering from MIT in 2016.
\end{IEEEbiography}

\begin{IEEEbiography}[{\includegraphics[width=1in,height=1.25in,clip,keepaspectratio]{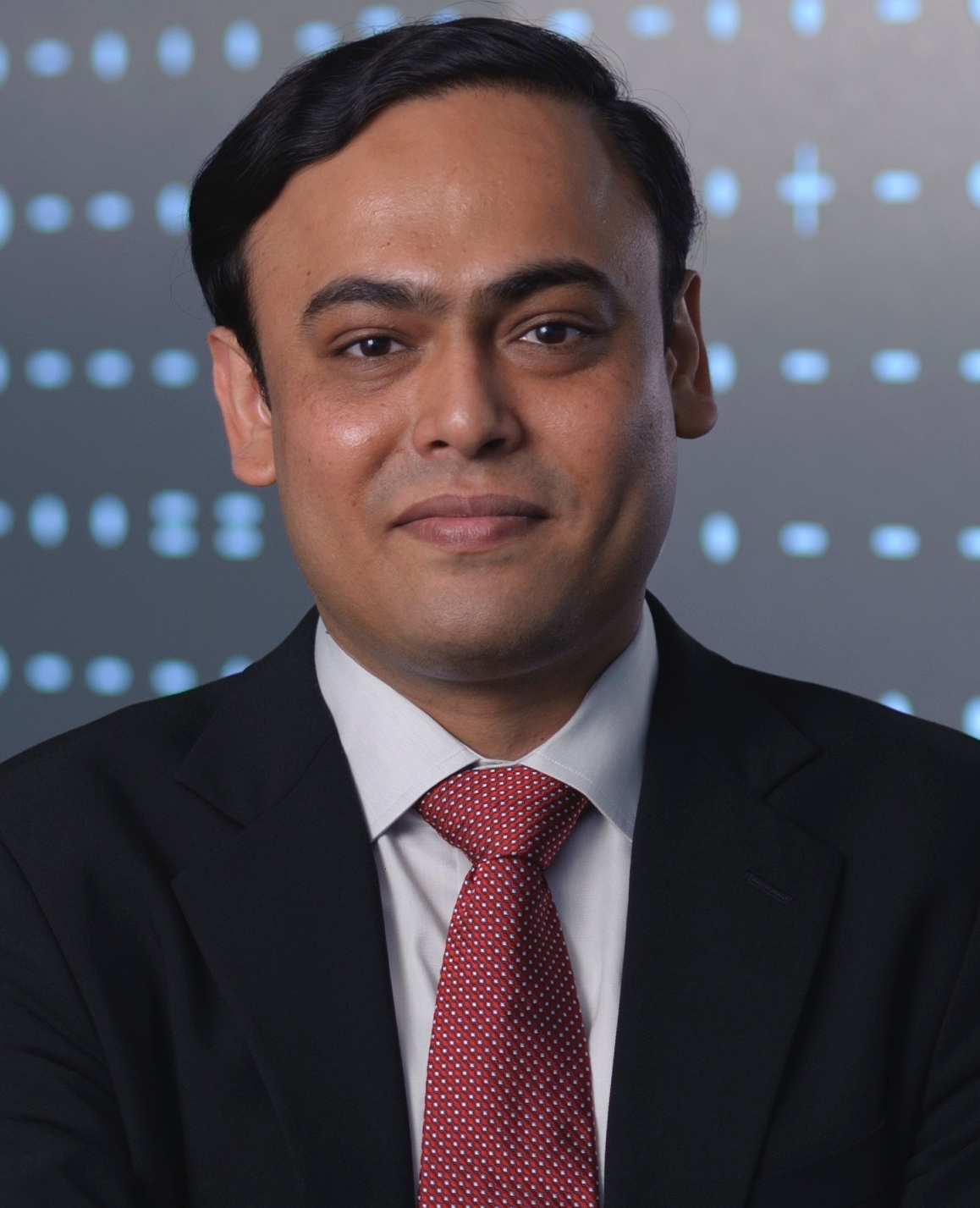}}]{Saurabh Amin}
is Robert N. Noyce Career Development Assistant Professor in the Department of Civil and Environmental Engineering, Massachusetts Institute of Technology (MIT). His research focuses on the design and implementation of high confidence network control algorithms for infrastructure systems. He works on robust diagnostics and control problems that involve using networked systems to facilitate the monitoring and control of large-scale critical infrastructures, including transportation, water, and energy distribution systems. He also studies the effect of security attacks and random faults on the survivability of networked systems, and designs incentive-compatible control mechanisms to reduce network risks. Dr. Amin received his Ph.D. in Systems Engineering from the University of California, Berkeley in 2011. He is a recipient of NSF CAREER award, and Google Faculty Research award. 
\end{IEEEbiography}
  
\ifArxivVersion
\clearpage 
\section*{Supplementary Material}


\begin{proof}[\textbf{Proof of \cref{lem:lpfAndUpf}}]
Recursively apply the power flow models \eqref{eq:NPF}, \eqref{eq:LPF}, and  \eqref{eq:UPF}, from the root node to leaf nodes.
\end{proof}

\begin{proof}[\textbf{Proof of \cref{lem:lpfNpfRelationship}}]

We apply induction from  leaf nodes to the root node. 
\textbf{Base case:} For any leaf node $k \in \Nleaf$,
\begin{align*}
\begin{aligned}\myEquationStyle
z_k\ellnpf_k &\stackrel{\subassmref{assm:smallLineLosses}}{\le} \slr \Snpf_k &&\stackrel{\eqref{eq:recursiveNPFConservation}}{=} \slr(\sn_k + z_k\ellnpf_k) \\
\therefore z_k\ellnpf_k &\le \myFrac{\slr\sn_k}{1-\slr} &&\stackrel{\eqref{eq:recursiveLPFConservation}}{=} \myFrac{\slr\Slpf_k}{1-\slr}.
\end{aligned}
\end{align*}
Now, for any $j\in\N \backslash \Nleaf$, 
\begin{alignat}{8}\myEquationStyle
z_j\ellnpf_j \stackrel{\subassmref{assm:smallLineLosses}}{\le} \slr\Snpf_j &\stackrel{\eqref{eq:NPFconservation}}{=} \slr\big[\ssum_{k:(j,k)\in\E} \Snpf_k + \sn_j+z_j\ellnpf_j \big] \nonumber \\
\therefore z_j\ellnpf_j &\le \myFrac{\slr}{1-\slr} \big[\ssum_{k:(j,k)\in\E} \Snpf_k + \sn_j\big]. \nonumber
\end{alignat}
Adding $\sum \Snpf_k + \sn_j$ on both the sides:
\begin{alignat}{8}\myEquationStyle \underbrace{\ssum_{k:(j,k)\in\E} \Snpf_k + \sn_j+z_j\ellnpf_j}_{\Snpf_j} &\le \myFrac{1}{1-\slr} \big[\ssum_{k:(j,k)\in\E} \Snpf_k + \sn_j\big]. \nonumber
\end{alignat}
\textbf{Inductive step:} By inductive hypothesis (IH) on $\child_j$, 
\begin{alignat}{8}\myEquationStyle
 \Snpf_j &\stackrel{\text{(IH)}}{\le} \myFrac{1}{(1-\slr)^{\nl-\abs{P_k} + 2}}\big[\ssum_{k:(j,k)\in\E} \Slpf_k + \sn_j \big] \nonumber \\
 \vspace{-1cm}&\stackrel{}{=}  \myFrac{\Slpf_j}{(1-\slr)^{\nl-\abs{P_j} + 1}} \quad (\because \abs{\P_j} = \abs{\P_k} - 1). \nonumber
\end{alignat}
\end{proof}

\begin{proof}[\textbf{Proof of \cref{prop:allPowerFlows}}] The inequalities $\Slpf \le \Snpf$ and $\nulpf \ge \nunpf$ are proved in \cite{exactConvexRelaxation}.

The rest of the proof of \Cref{prop:allPowerFlows} utilizes two lemmas. From \Cref{lem:lpfNpfRelationship}, for any $(i,j)\in\E$, 
\begin{equation}\label{eq:npfAndUpfPowerFlow}\myEquationStyle
S_j \le \frac{\Slpf_j}{(1-\slr)^{\Pmax - \abs{\P_j} + 1}} \le \frac{\Slpf_j}{(1-\slr)^{\Pmax}} = (1+\slrrr) \Slpf_j \stackrel{\eqref{eq:lpfAndUpf}}{=} \Supf_j.
\end{equation}
For nodal voltages, 
\begin{alignat}{8} \myEquationStyle
\nu_j & \myEquationStyle \stackrel{\eqref{eq:voltageSquare2}}{=} \nu_i - 2\Re(\bar{z}_jS_j) + \abs{z}_j^2 \ell_j \nonumber \\
\myEquationStyle & \myEquationStyle \stackrel{}{\ge} \nu_i - 2\Re(\bar{z}_jS_j) \nonumber \\
\myEquationStyle \label{eq:voltageNpfUpf}& \myEquationStyle \stackrel{\eqref{eq:npfAndUpfPowerFlow}}{\ge} \nu_i - 2\Re(\bar{z}_j\Supf_j).
\end{alignat}
Applying \eqref{eq:voltageNpfUpf} recursively from the node~$j$ till root node:
\begin{align*}
\begin{aligned}\myEquationStyle
\nu_j &\ge \nu_0 - 2\ssum_{k\in\P_j}\Re(\bar{z}_k\Supf_k) \stackrel{\eqref{eq:recursiveVoltageSquareUpfToRoot}}{=} \nuupf_j.
\end{aligned}
\end{align*}
Thus, $\Slpf_j \le \Snpf_j \le \Supf_j \text{ and } \nulpf_j \ge \nunpf_j \ge \nuupf_j$. Furthermore, 
\begin{align*}
\begin{aligned}\myEquationStyle
\Slpf_j \le \Snpf_j \le \Supf_j \stackrel{\subassmref{assm:noReversePowerFlows}}{\implies} \myEquationStyle \abs{\Slpf_j}^2 \le \abs{\Snpf_j}^2 \le \abs{\Supf_j}^2 \\
\implies \myFrac{\abs{\Slpf_j}^2}{\nulpf_j} \le \myFrac{\abs{\Snpf_j}^2}{\nunpf_j} \le \myFrac{\abs{\Supf}^2}{\nuupf_j} \implies \elllpf \le \ellnpf \le \ellupf. 
\end{aligned}
\end{align*}
Finally, \eqref{eq:optimalValueRelationship} immediately follows from \eqref{eq:lossFunction}, \eqref{eq:lossFunctionDefinitions}, and \eqref{eq:lpfNpfUpfRelationship}. 
\end{proof}

\begin{proof}[\textbf{Proof of \Cref{lem:convexRelaxationOptimalSolution}}] 

Let $\adspcpf$ denote the following problem:
\begin{align}
\begin{aligned}
&\adspcpf \quad \phistarcpf(\psi) &&\in  \argmin_{\phi\in\Phi} \Loss(\xnpf(\psi,\phi))\quad \\
& && \text{s.t. } \xlpf(\ad,\psi,\phi) \in \Xcpf, \eqref{eq:directLoadControl}, \eqref{eq:pvSetpoints}. 
\end{aligned}
\end{align} 

\subassmref{assm:hardLowerBoundOfVoltage} implies that a feasible solution exists for $\adspnpf$. Since, $\Xnpf \subset \Xcpf$, a feasible solution $\xcpf\in\Xcpf$ also exists for $\adspcpf$. 

Let $(\phicpf,\ellcpf)$ denote the decision variables for $\adspcpf$.  Note that, for a fixed $\psi$, $\xcpf$ is affine in the variables $(\phicpf,\ellcpf)$, and can be very efficiently computed using \eqref{eq:NPFconservation} and \eqref{eq:voltageSquare2}. 

Now, $\Loss$ is convex in $\phicpf$ (because the $\llovr$ is a maximum over affine functions, $\lvoll$ is affine in $\phicpf$, and $\lloss$ is affine in $\ellcpf$. Also, $\Phi$ is a convex compact set. Further, for a fixed $\phi$, $\Loss$ is strictly increasing in $\ellcpf$ (because, $\llovr$ is non-decreasing in $\ellcpf$ as $\nucpf$ is affine decreasing in $\ellcpf$; $\lvoll$ does not change with $\ellcpf$; $\lloss$ is strictly increasing in $\ell$). From Theorem~1 \cite{farivar}, $(\phistarcpf,\ellstarcpf)$ can be computed using a SOCP. To argue that $\ellstarcpf$ satisfy \eqref{eq:currentMagnitudeNpf}, assume for contradiction that $\exists\;(i,j)\in\E, \text{s.t. } \ellstarcpf_j > \sfrac{\abs{\Sstarcpf_j}^2}{\nustarcpf_i}$. Then, construct $(\phistarnpf,\ellcpf')$ such that $\forall\;j\in\N\;:j\ne i,\; \ellcpf_j' = \ellstarcpf_j$, and $\ellcpf_i' = \sfrac{\abs{\Sstarcpf_j}^2}{\nustarcpf_i} $. Since  $\forall\;(j,k)\in\E,\sfrac{\abs{\Scpf_k}^2}{\nucpf_j}$ is strictly decreasing in $\ellcpf_i$, $\forall\;(j,k)\in\E:\;\ellcpf_k' \ge \sfrac{\abs{\Sstarcpf_k}^2}{\nustarcpf_j} > \sfrac{\abs{\Scpf_k'}^2}{\nucpf_j'}$. Hence, one can further minimize the loss function by choosing a new feasible solution $(\phistarnpf,\ellcpf')$, thus violating the optimality of $(\phistarcpf,\ellstarcpf)$.
\end{proof}

\begin{proof}[\textbf{Proof of \cref{prop:pvSetpoints}}]
\def \istar {j} 

Let $(\di, \myti)$ denote $\sgsetdlpf_i$ in the polar coordinates, i.e., $\di = \abs{\sgsetdlpf_i}, \myti = \angle \sgsetdlpf_i$. 

For $\aa_i = 0$, $\sgsetlpf_i = \sgsetdlpf_i$. Then from  \eqref{eq:recursiveVoltageSquareLpf}, $\forall\;j\in\N$, 
\begin{equation} \label{eq:partialNuTheta}
 \nulpf_j = \nulpf_j' + 2\di (\rr_{ij} \cos\myti + \xx_{ij}\sin\myti), 
\end{equation}
where $\nulpf_j' = \nu_0 - 2\ssum_{k\in\N,k\ne j} \Re(\conj{\zz}_{jk}\sn_k) - 2\Re(\conj{\zz}_{ij}\sc_j)$. Note that $\nulpf_j'$ does not depend on $(\di, \myti)$. 

It is clear from \eqref{eq:partialNuTheta} that $\nulpf_j$ is greater if $\myti \in [0,\sfrac{\pi}{2}]$ than if $\myti \in [-\sfrac{\pi}{2},0]$. Furthermore, the impedances are positive. Hence, $\forall\; j, \partial_{\di}\nulpf_j = 2(\rr_{ij}\cos\myti + \xx_{ij}\sin\myti) > 0$. Hence, $\partial_{\di}\llovr > 0$. But, from \eqref{eq:dgConstraint1}, $\di \le \sgmax_i$. Hence, $\dstari = \sgmax_i$. Further, $\partial_{\myti} \nulpf_j = 2\di(-\rr_{ij}\sin\myti + \xx_{ij}\cos\myti)$. 
\[
\partial_{\myti} \nulpf_j \quad 
\begin{cases}
 > 0 & \text{ if } \quad \myti \in [0, \arccot (\sfrac{\rr_{ij}}{\xx_{ij}}) ) \\
 = 0 & \text{ if } \quad \myti = \arccot (\sfrac{\rr_{ij}}{\xx_{ij}}) ) \\
 < 0 & \text{ if }  \quad\myti \in (\arccot (\sfrac{\rr_{ij}}{\xx_{ij}}) ), \sfrac{\pi}{2}] 
\end{cases}
\]
Now, $\arccot \kmax \le \arccot (\sfrac{\xx_{ij}}{\rr_{ij}}) \le \arccot \kmin$. Hence, 
\begin{equation}\label{eq:partialNuTheta2} 
\forall\quad j\in\N,\quad \partial_{\myti} \nulpf_j  \quad 
\begin{cases}
 < 0 & \text{ if } \myti > \quad \arccot \kmin \\
 > 0 & \text{ if } \myti < \quad \arccot \kmax \\
\end{cases}
\end{equation} 

\noindent Suppose, for contradiction, $\myti^\star \not\in [\arccot \kmax,\arccot \kmin]$. Holding all else equal, for $\myti = \mytti$, let $\nulpf({\mytti})$ and $\llovr({\mytti})$ be the $\nulpf$ and $\llovr$. From \eqref{eq:partialNuTheta2}, for any $\mytti \in [\arccot \kmax, \arccot \kmin]$, $\nulpf({\mytti}) >  \nulpf({\mytti}^\star) $. Since, $\llovr > \lloss \ge 0$, $\llovr({\mytti}) < \llovr({\mytti}^\star)$, violating the optimality of ${\mytti}^\star$. Furthermore, under identical $\rxratio$ ratio, $\kmin = \kmax = K$, which implies $\myti = \arccot K$. 
\end{proof}

\begin{claim}
\cref{thm:attackerSetpoints} also holds for $\admpnpf$. 
\end{claim}
\begin{proof}
Now, we prove the case for $\admpnpf$ by contradiction. Suppose that there exists $i\in\N \text{ s.t. } \Re(\sgset^{a\star}_i) > 0$. Then we can construct another attacker strategy $\psistarcpf = [\aanew,\sgsetanew]$ that can further maximize $\Loss$, such that $\Re(\sgset^{a\star}_i) = 0$, holding all else equal, i.e., $\aanew = \delta, \forall\;j\in\N,\;\Im(\sgsetanew_j) = \Im(\sgset^{a\star}_j), \forall\;j\in\N:j\ne i,\; \Re(\sgsetanew_j) = \Re(\sgset^{a\star}_j)$.  

Let $(\sgseta,\ell)$ be the decision variables for $\admpnpf$, as for fixed $\phi$, the other decision variables $\Pnpf,\Qnpf,\nunpf$ can then be written as affine functions of $(\sgseta,\ell)$ from \eqref{eq:NPF}. Let $(\sgsetastarnpf,\ellstarnpf)$ (resp. $(\sgsetacpf,\ellcpf)$) be the solution to $\admpnpf$ when $\psi = \psistarnpf$ (resp. $\psi = \psicpf$).  

Let $\f \in \R_+^\NN$ such that, for any $(i,j)\in\E$, $f_{j}(\sgseta,\ell) \coloneqq \frac{\Pnpf_{j}^2+\Qnpf_{j}^2}{\nunpf_i}$. Let $\fstar = \f(\sgsetastarnpf, \ellstarnpf)$, $\fnew = \f(\sgsetacpf,\ellstarnpf)$, and $\fcpf = \f(\sgsetacpf,\ellcpf)$.

Since $(\sgsetastarnpf,\ellstarnpf)$ and $(\sgsetacpf,\ellcpf)$ are solutions to $\admpnpf$, they satisfy \eqref{eq:currentMagnitudeNpf}. Hence,  $\fstar = \ellstarnpf$, and $\fcpf = \ellcpf$. Furthermore, it can be checked that $\fnew > \fstar$. We want to show that $\fcpf > \fnew$. Assume that $\fcpf > \fnew$. Then, $\fcpf > \fstar$. Hence, $\Loss(\xcpf) > \Loss(\xstarnpf)$, (because,  $\llovr(\xcpf) > \llovr(\xstarnpf)$, $\lvoll(\xcpf) =  \lvoll(\xstarnpf)$, $\lloss(\xcpf) > \lloss(\xnpf) $). However, this is a contradiction, as it violates the optimality of $\sgsetastarnpf$.  By similar logic, we can show that $\forall\;i\in\N,\;\Im(\sgsetastarnpf_i) = -\sgmax_i$. 
 
We now prove that $\fcpf > \fnew$, with the help of an illustrative diagram (see \cref{fig:currentShift}).  
\begin{figure}[h!]
\includegraphics[width=7cm]{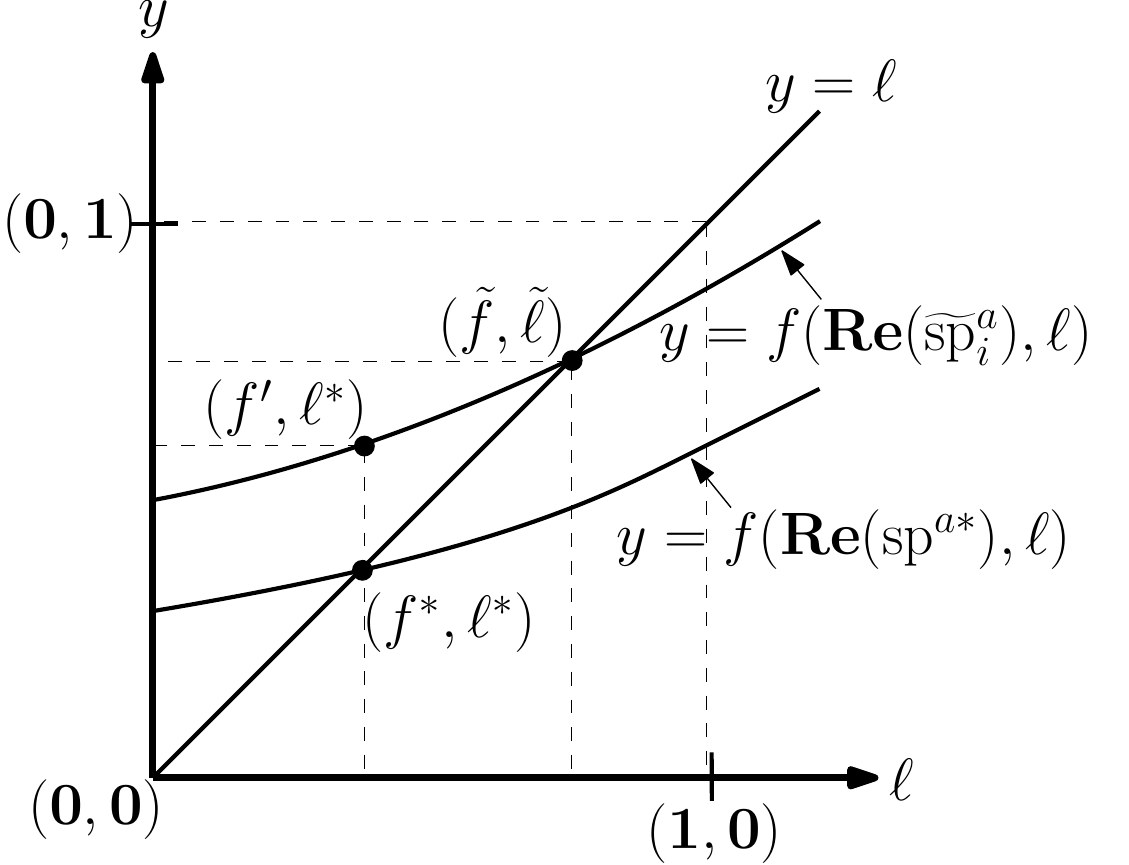}
\caption{Illustrative diagram showing how $\ell$ changes with $\sgseta$} \label{fig:currentShift}
\end{figure}

Note that from \eqref{eq:NPF}, one can show that for any $\ell$, $\f(\Re(\sgsetanew,\ell)) > \f(\Re(\sgsetastarnpf,\ell))$. Now, consider the $(j,k)^{th}$ entry of Jacobian $\textbf{J}_\f(\ell)$.
\begin{align*}
\begin{aligned}
\partial_{\ell_k} f_{j} &= \myFrac{\nu_i \big(2P_{j}\partial_{\ell_k}P_j + 2Q_{j}\partial_{\ell_k}Q_j\big)}{\nu_i^2} - \myFrac{(P_{j}^2+Q_{j}^2)\partial_{\ell_k} \nu_i}{\nu_i^2} 
\end{aligned}
\end{align*}
\begin{align*}
\begin{aligned}
\therefore \quad 0 \stackrel{\assmref{assm:noReversePowerFlows}}{\le} \partial_{\ell_k} f_{j} & \stackrel{\assmref{assm:lowImpedance}}{<} \myFrac{(2r_{k} + 2x_{k})}{\nu_i} + \myFrac{(4\rr_{ik}r_{k} + 4\xx_{ik}x_{k})}{\nu_i^2} \\
\implies \quad 0 \le \partial_{\ell_k} f_{j} & \stackrel{\assmref{assm:lowImpedance}}{<} (r_{k} + x_{k})(\sfrac{2}{\mu}+\sfrac{4}{\mu^2}) \quad  \stackrel{}{\le} 1 \\
\implies \quad 0 \le \partial_{\ell_k} f_{j} & < 1. 
\end{aligned}
\end{align*}

At $\ell = \textbf{0}$, $\f > \textbf{0}$, and each entry of Jacobian $\textbf{J}_\f(\ell)$ is positive and smaller than 1. Hence, $\f$ intersects the hyper-plane $y = \ell$, exactly once. Furthermore, $\f(\Re(\sgsetanew,\ell)) > \f(\Re(\sgsetastarnpf,\ell))$. Hence, we can conclude that $\ellcpf = \fcpf > \fnew > \fstar = \ellstarnpf$. 
\end{proof}

\begin{proof}[\textbf{Proof of \cref{prop:greedyOptimality}}]

Note that for fixed $\phi\in\Phi$, maximizing $\Losslpf(\xlpf(\aalpf,\phi))$ (resp. $\Lossupf(\xupf(\aalpf,\phi))$) is equivalent to maximizing $\llovr(\xlpf(\aalpf,\phi))$ (resp. $\llovr(\xupf(\aaupf,\phi))$). 
Let $\aastarlpf$ be the optimal solution to $\admplpf$. 

\textbf{Case (i).} $\llovr(\xlpf(\aastarlpf,\phi)) = 0$.  Then \Cref{algo:greedyAlgorithm} computes $\aastarlpf$ trivially, because $\llovr(\xlpf(\aastarlpf,\phi)) \ge  \llovr(\xlpf(\aalpf,\phi)) \ge 0$. Hence, $\llovr(\xlpf(\aalpf,\phi)) = \llovr(\xlpf(\aastarlpf,\phi)) = 0$. 

\textbf{Case (ii).} $\llovr(\xlpf(\aastarlpf,\phi)) > 0$.  
Let $\nulpf_j(\aa,\phi)$ denote the nodal voltage at node~$j$ after the attack $\aa$. Since, $\aalpf = \aaklpf$ , for some pivot node~$k \in \N$ (see \Cref{algo:greedyAlgorithm}), $\aaklpf$ maximally violates  \eqref{eq:voltageConstraint} over all $\aailpf$, i.e.,
\begin{equation}\label{eq:optimalityOfaaklpf}
\forall\;i\in\N,\quad \nuMin_k - \nulpf_k(\aaklpf,\phi) \ge \nuMin_i - \nulpf_i(\aailpf,\phi), 
\end{equation}
where $\aailpf$ is the optimal pivot node attack  as computed by \Cref{algo:greedyAlgorithm} for node~$i$, i.e.,
\begin{equation} \label{eq:optimalityOfaailpf}
\forall\;i\in\N,\;\forall\;\aa\in\Da \quad \nuMin_i - \nulpf_i(\aailpf,\phi) \ge \nuMin_i - \nulpf_i(\aa,\phi). 
\end{equation}

\noindent Let $i = \argmax_{j\in\N}\W_j (\nuMin_j-\nulpf_j(\aastarlpf,\phi))_+$.  Furthermore, since $ \llovr(\xlpf(\aastarlpf,\phi)) > 0$, 
\begin{equation} \label{eq:positiveLovr}
\llovr(\xlpf(\aastarlpf,\phi))  = \W_i (\nuMin_i-\nulpf_i(\aastarlpf,\phi))
\end{equation} \normalsize  
\begin{align*}
\begin{aligned}
\therefore \llovr(\xlpf(\aalpf, \phi)) & =  \llovr(\xlpf(\aaklpf,\phi))\\
&\stackrel{\eqref{eq:lovrDefinition}}{=} \mmax_{j\in\N}\W_j (\nuMin_j-\nulpf_j(\aaklpf,\phi))_+ \\
& \stackrel{\eqref{eq:optimalityOfaaklpf}}{\ge} \W_i(\nuMin_i-\nulpf_i(\aailpf,\phi))  \\ 
& \stackrel{\eqref{eq:optimalityOfaailpf}}{\ge} \W_i(\nuMin_i-\nulpf_i(\aastarlpf,\phi)) \\
& \stackrel{\eqref{eq:positiveLovr}}{=}  \llovr(\xlpf(\aastarlpf,\phi)). 
\end{aligned}
\end{align*}
 \normalsize 
Furthermore, for a fixed $\phi$ $\lvoll(\xlpf(\aalpf,\phi)) = \lvoll(\xlpf(\aastarlpf,\phi))$. Hence, 
$\quad 
\Losslpf(\xlpf(\aalpf,\phi)) \ge  \Losslpf(\xlpf(\aastarlpf,\phi))$. 
\end{proof}

\begin{proof}[\textbf{Proof of \Cref{lem:pvCompromiseEffects}}]
Let $\sgset_j$ be the DER set-point of node $j$ before the attack. If $\sgset_j$ is the pre-attack set-point, let $\Delta_j(\sgset_j)$ denote the change in the set-point of DER~$j$  after it is compromised. By \cref{thm:attackerSetpoints}, $\Delta(\sgset_j) = \sgset_j - \sgseta_j =  \sgset_j - (0 - \j\sgmax_j) = \sgset_j + \j\sgmax_j$; and by  linearity in \eqref{eq:recursiveVoltageSquareLpf},
\begin{align*}
\begin{aligned}
\Delta_j(\nu_i) = 2\Re(\conj{Z}_{ij}\Delta_j(\sgset_j)) = 2\Re(\conj{Z}_{ij}(\sgsetd_j + \j\sgmax_j)).
\end{aligned}
\end{align*}
Again, by invoking the linearity in \eqref{eq:recursiveVoltageSquareLpf}, \eqref{eq:sumVoltageSquareLPF} follows. 

\noindent Similarly, one can show \eqref{eq:recursiveUPFVoltageSquare} and \eqref{eq:sumVoltageSquareUPF}. 
\end{proof}

\begin{proof}[\textbf{Proof of \cref{lem:optimalAttackSetIndependentOfGamma}}]
The computation of $\Dastarlpf(\phi)$ depends on $\Delta_j(\nulpf_i)$ values which depend only on $\sgsetd$, and not on $\gammalpf$ (see \Cref{lem:pvCompromiseEffects}). 
\end{proof}

\begin{proof}[\textbf{Proof of \cref{prop:downstream}}]
When $\aa_j = 1$, i.e., the DER $j$ is compromised, only the power supplied at node $j$ changes. Using \eqref{eq:recursiveLPFVoltageSquare}, we get, 
\begin{align*}
\begin{aligned}
\therefore \Delta_j(\nulpf_i) &= 2\Re(\conj{\zz}_{ij}\Delta(\sgsetd_j)) \\
&= 2\Re(\conj{\zz}_{ij}(\sgsetd_j + \j\sgmax_j)). \\
\end{aligned}
\end{align*}

\begin{alignat*}{2}
\text{Now, } j \prec_i k &\implies \P_i\cap\P_j \subset \P_i\cap\P_k \implies \zz_{ij} < \zz_{ik}. 
 \end{alignat*}
\begin{small}
\begin{equation*}
\begin{split}
\therefore \Delta_j(\nulpf_i) &= 2\Re(\conj{\zz}_{ij}(\sgsetd_j+ \j\sgmax_j))\\
& < 2\Re(\conj{\zz}_{ik}(\sgsetd_k+ \j\sgmax_k)) = \Delta_k(\nulpf_i)
\end{split}
\end{equation*}
\end{small}
Similarly, we can prove the case for $j =_i k$. 

\noindent Under the $\UPF$ model, \small 
$\myEquationStyle
\Delta_j(\nuupf_i) = 2(1+\slrrr)\Re(\conj{\zz}_{ij}(\sgsetd_j+ \j\sgmax_j))$. \normalsize
The rest of the proof follows similarly. 
\end{proof}

\begin{remark}

 \Cref{prop:downstream} implies that, broadly speaking, compromising downstream DERs is advantageous to the attacker than compromising the upstream DERs. The following illustrative example suggests that compromising DERs by means of clustered attacks are more beneficial to the attacker than distributed attacks. 
 
 \begin{example}
 Consider the $\admplpf$ with $M = 2$ instantiated on the DN in  \cref{fig:precedenceDescription}. Assume that all loads and DERs are homogeneous, all lines have equal impedances, i.e., $\forall i\in\N, \sc_i = \sc_a$, $\sgsetd_i = \sgsetd_a$, $\sgmax_i = \sgmax_a$, $z_i = z_a$. By \Cref{prop:pvSetpoints}, the outputs of all the DERs are fixed and identical to each other. 
 
 Let $\alpha = 2(\Re(\conj{z}_a(\sc_a-\sgsetd_a)))$, and $\beta = 2(\Re(\conj{z}_a (\Re(\sgsetd_a) + \j (\Im(\sgsetd_a)+\sgmax_a))))$. Then $\nu$ values for different attack vectors are given in \Cref{tab:clusterAttacks}. The optimal attack compromises nodes $i$ and $m$, which is a cluster attack. 
 
 \begin{table}[h]
 \tiny
 \centering
 \begin{tabular}{|p{1.3cm}|c|c|c|}
 \hline
 Attacked Nodes & $\nu_m$ & $\nu_j$ & $\nu_k$\\
 \hline
  $\emptyset$& $\nu_0 - 23\alpha$ & $\nu_0 - 13\alpha$ & $\nu_0 - 20\alpha$ \\
 \rowcolor{lightgray}  $\{i,m\}$ & $\nu_0-23\alpha-9\beta$& $\nu_0-13\alpha-2\beta$& $\nu_0-20\alpha-4\beta$\\
  $\{j,m\}$ & $\nu_0-23\alpha-6\beta$& $\nu_0-13\alpha-4\beta$& $\nu_0-20\alpha-3\beta$\\
  $\{k,m\}$ & $\nu_0-23\alpha-7\beta$& $\nu_0-13\alpha-2\beta$& $\nu_0-20\alpha-6\beta$\\
  $\{g,j\}$ & $\nu_0-23\alpha-2\beta$& $\nu_0-13\alpha-5\beta$& $\nu_0-20\alpha-2\beta$\\
  $\{d,k\}$ & $\nu_0-23\alpha-4\beta$& $\nu_0-13\alpha-2\beta$& $\nu_0-20\alpha-7\beta$\\
  \hline
 \end{tabular}
 \caption{$\nu$ vs Different Attack Combinations.}\label{tab:clusterAttacks}
 \end{table}
 \normalsize
 \end{example}
 Consequently, our results (see \cref{sec:securityDesign}) on security strategy in Stage~1 show that the defender should utilize his security strategy to deter cluster attacks. 
 
\end{remark}

 \begin{proof}[\textbf{Proof of \cref{prop:generalApproach}}]
 For a fixed defender action $\phi$, we have from \eqref{eq:nulpfupfRelationship} that $\forall\ j\in\N, \Delta_j(\nuupf_i) = (1+\epsilon) \Delta_j(\nulpf_i)$. Hence, the sequence of partitions of the nodes for every pivot node is the same in both the LPF and the $\UPF$ model. 
 Hence, $\Dastarlpf(\phi) = \Dastarupf(\phi)$. 
 
 Now, for any $\psi_1,\psi_2 \in\Psi$, 
  \begin{align}
  \begin{aligned}
  \lvoll(\xlpf(\psi_1,\phi)) &= \lvoll(\xlpf(\psi_2,\phi)) \quad \text{ and }\\
  \lvoll(\xupf(\psi_1,\phi)) &= \lvoll(\xupf(\psi_2,\phi)). \label{eq:sameVollCost}
  \end{aligned}
 \end{align}
  Suppose $\psistarlpf$ is not an optimal solution to $\admpupf$. Then, 
  \begin{small}
  \begin{alignat*}{8}
  &\Lupf(\xupf(\psistarupf,\phi)) &&> \Lupf(\xupf(\psistarlpf,\phi))\\
  \stackrel{\eqref{eq:sameVollCost}}{\iff} & \llovr(\xupf(\psistarupf,\phi)) &&> \llovr(\xupf(\psistarlpf,\phi)) \\
  \iff & \max_{i\in\N}W_i(\nuMin_i - \nuupf_i(\psistarupf,\phi))_+ && > \max_{j\in\N}W_j(\nuMin_j - \nuupf_j(\psistarlpf,\phi))_+ \\
  \stackrel{\eqref{eq:positiveLovrCond}}{\iff} & \max_{i\in\N}(\nuMin - \nuupf_i(\psistarupf,\phi)) &&> \max_{j\in\N}(\nuMin - \nuupf_j(\psistarlpf,\phi)) \\
  \iff & \max_{i\in\N}(\nu_0 - \nuupf_i(\psistarupf,\phi)) &&> \max_{j\in\N}(\nu_0 - \nuupf_j(\psistarlpf,\phi)) \\
  \stackrel{\eqref{eq:nulpfupfRelationship}}{\iff} &   (1+\epsilon)\linfinityNorm{(\nu_0 - \nulpf_i(\psistarupf,\phi))} && > (1+\epsilon)\linfinityNorm{(\nu_0 - \nulpf_j(\psistarlpf,\phi))} \myEquationStyle \normalsize\\
  \stackrel{\eqref{eq:positiveLovrCond}}{\iff}  & \max_{i\in\N}W_i(\nuMin_i - \nulpf_i(\psistarupf,\phi))_+ &&> \max_{j\in\N}W_j(\nuMin_j - \nulpf_j(\psistarlpf,\phi))_+ \\
  \iff & \llovr(\xlpf(\psistarupf,\phi)) &&> \llovr(\xlpf(\psistarlpf,\phi)) \\
  \stackrel{\eqref{eq:sameVollCost}}{\iff} &\Llpf(\xlpf(\psistarupf,\phi)) &&> \Llpf(\xlpf(\psistarlpf,\phi)). 
  \end{alignat*}
  \end{small}
  Hence, the contradiction that $\psistarlpf$ is an optimal solution to $\admplpf$. 
  Similarly, we can show that $\psistarupf$ is an optimal solution to $\admpupf$. 
  \end{proof}

To prove \Cref{thm:optimalSecureDesign}, we first introduce \cref{prop:childNodeSecure,prop:design2Better,prop:type3secure}. 
\noindent Consider any security strategy $\ad\in \AD$ such that \small
\begin{equation}\label{eq:originalSecureDesign}\myEquationStyle
\ad = \begin{bmatrix}
\underbracket{\ad_1}_1 & \underbracket{\ad_2}_2 &  \dots & \underbracket{\hspace{3pt}1\hspace{3pt}}_a &  \dots &  \underbracket{\hspace{3pt}0\hspace{3pt}}_b & \dots &  \underbracket{\ad_\NN}_\NN
\end{bmatrix}.
\end{equation} \normalsize

\noindent \tcb{Construct $\adnew$ from $\ad$ by only flipping the bits at nodes $a$ and $b$ as follows:}
\begin{equation} \small 
\label{eq:newSecureDesign}\myEquationStyle
\adnew = \begin{bmatrix}
\underbracket{\ad_1}_1 & \underbracket{\ad_2}_2 &  \dots & \underbracket{\hspace{3pt}0\hspace{3pt}}_a &  \dots &  \underbracket{\hspace{3pt}1\hspace{3pt}}_b & \dots &  \underbracket{\ad_\NN}_\NN
\end{bmatrix},
\end{equation} \normalsize
i.e., $\adnew_i = \ad_i\quad\forall\quad i\in\N\backslash\{a,b\}$. Similarly, let $\aa \in \Da(u)$ such that $\aa_a = 0, \aa_b = 1$; and construct $\aanew$ from $\aa$ as in \eqref{eq:newSecureDesign} such that $\aanew_i = \aa_i\quad\forall\quad i\in\N\backslash\{a,b\}$. Note that, $\aanew \in \Da(\adnew)$.

We use \cref{prop:childNodeSecure,prop:design2Better,prop:type3secure} to compare the security strategies $\u$ and $\adnew$ under various conditions. Refer to \cref{fig:design2Tree} for the purpose of proofs of \cref{prop:childNodeSecure,prop:design2Better,prop:type3secure}.
\begin{proposition}\label{prop:childNodeSecure}
Assume \assmref{assm:throughout},  \assmref{assm:identicalRX}, \assmref{assm:balancedTree}. Let $\u\in\Dd$ (resp. $\adnew\in\Dd$) be as in \eqref{eq:originalSecureDesign} (resp. \eqref{eq:newSecureDesign}). If $b \in \Lambda_a$, then $\u \securePreceq \adnew$. 
\end{proposition} 

\begin{proof}[\textbf{Proof of \cref{prop:childNodeSecure}}]
Let $(\aastar,\phistar)$ and $(\aastarnew,\phistarnew)$, denote the optimal solutions of $\adlpf$ with $\u = \ulpf$ (resp. $\u = \adnew$). \assmref{assm:identicalRX} $\implies \sgsetdstar$ is fixed (\cref{prop:pvSetpoints}). Hence, $\phistar$ depends only on $\aastar$, and not $\ad$. Then, let $\phistar(\aa)$ denote optimal defender response to $\aa$. We want to show $\Lulpfnew \le \Lulpf$. 

 \textbf{Case ${\aastarnew_a = 0}$.} Then $\aastarnew\in\Da(\ad)$. Thus,   
 \begin{align*}
\begin{aligned}
\Lulpfnew &= \Losslpf(\xlpf(\adnew,\aastarnew,\phistar(\aastarnew))) \\
 &= \Losslpf(\xlpf(\ad,\aastarnew,\phistar(\aastarnew))) \\
 & \le \Losslpf(\xlpf(\ad,\aastar,\phistar(\aastar))),
\end{aligned}
 \end{align*} where the inequality follows due to the optimality of $\aastar$. 
  
\textbf{Case ${\aastarnew_a = 1}$.} Let $\aa\in\Da(\ad):\aa_a = 0,\aa_b=1,\forall i\in\N\backslash \{a,b\}, \aa_i = \aastarnew_i$.  We have assumed that $b\in\Lambda_a$; see~\cref{fig:design2Tree}. Therefore, $\forall\;i\in\N,\; a \preceq_i b$. Hence, by \cref{prop:downstream}, 
\begin{equation*}
\forall\; i\in\N,\;\Delta_b(\nulpf_i) \ge \Delta_a(\nulpf_i). 
\end{equation*}  
Then, by \cref{lem:pvCompromiseEffects}, for fixed $\phi$,  
$\Delta_{\aa}(\nulpf) \ge \Delta_{\aastarnew}(\nulpf)$. Hence,
\begin{align*}
\begin{aligned}
\Lulpfnew &= \Losslpf(\xlpf(\adnew,\aastarnew,\phistar(\aastarnew))) \\
 & \le \Losslpf(\xlpf(\adnew,\aastarnew,\phistar(\aa))) \\
 & \le \Losslpf(\xlpf(\ad,\aa,\phistar(\aa))) \\
 & \le \Losslpf(\xlpf(\ad,\aastar,\phistar(\aastar))) \\
 & = \Lulpf. 
\end{aligned}
\end{align*} 
Here, the first (resp. last) inequality follows due to optimality of $\phistar(\aastarnew)$ (resp. $\aastar$). Hence, $\ad \securePreceq \adnew$. 
\end{proof}

\begin{remark}
\noindent Starting with any strategy $\u'\in\Dd$, \cref{prop:childNodeSecure} can be applied recursively to obtain a more secure strategy $\ad\in\Dd : \u' \securePreceq \ad$, which has the property that if a node~$i$ is secure, then all its successor nodes (i.e. all nodes in subtree $\Lambda_i$) are also secured by the defender, i.e., 
 \begin{equation}\label{eq:childNodeSecure}
 \forall\;i\in\N,\;\ad_i = 1 \implies \forall\;j\in\Lambda_i,\;\ad_j = 1.
 \end{equation}

\end{remark}

\begin{proposition}\label{prop:design2Better}
Assume \assmref{assm:throughout}, \assmref{assm:identicalRX}, \assmref{assm:balancedTree}. Let $\u\in\Dd$ (resp. $\adnew\in\Dd$) be as in \eqref{eq:originalSecureDesign} (resp. \eqref{eq:newSecureDesign}). Let $A_\ad = \{(i,j) \in \N\times\N\; |\; \ad_i = 1, \ad_j = 0, \level_i \ge \level_j+1 \}$. If $\ad$ satisfies \eqref{eq:childNodeSecure}, and $(a,b) \in \argmax_{(i,j)\in A_\ad} \abs{\P_i\cap\P_j}$, then $\u \securePreceq \adnew$.  
\end{proposition}

\begin{proof}[\textbf{Proof of \cref{prop:design2Better}}] 
Let $c = \argmax_{(i\in\P_a\cap \P_b)} \level_i$, be the lowest common ancestor of $a$ and $b$. Let $i', i''\in\child_c\; :\;a\in\Lambda_{i'}$ and $b\in\Lambda_{i''}$. 
From \cref{thm:greedyApproach}, we know that the optimal attack $\aastar$ will be a pivot node attack $\aailpf$ for some node, say $i\in\N$. Let $\N' = \Lambda_{i'}\cup \Lambda_{i''}$. 

\textbf{Case $i\in\N'$.} Now $\ad_j = 1\; \forall\; j \in \Lambda_{i'}\backslash\Lambda_{i'}^a \cup\{a\}$ by maximality of $\abs{\P_a\cap\P_b}$. Similarly, $\ad_j = 0 \; \forall\; j \in \Lambda_{i''}^d\cup\{b\}$. Thus, $\forall\;j \in \Lambda_{i'} \text{ s.t. }\ad_j = 0$ there exists a separate node~$k \in \Lambda_{i''}$ such that $j$ and $k$ are homomorphic, and $\ad_k = 0$ (see \cref{fig:design2Tree}). Hence, the subtree $\Lambda_{i''}$ is more vulnerable than the subtree $\Lambda_{i'}$, and it will be more beneficial  for the attacker to target a pivot node in $\Lambda_{i''}$. Now, $i\in\Lambda_{i''}$, and $\forall\;i\in\Lambda_{i''},\; a \prec_i b$. Hence, by using  \cref{prop:downstream}, we get, $\Delta_a(\nulpf_i) <  \Delta_b(\nulpf_i)$. 

\textbf{Case $i\notin \N'$.} Then $a =_i b$, and by \cref{prop:downstream}, we have  $\Delta_a(\nulpf_i) = \Delta_b(\nulpf_i)$. 

We now want to show that $\Lulpfnew \le \Lulpf$. The rest of the proof is similar to the proof of \cref{prop:childNodeSecure}. 
%
\end{proof}

\begin{remark}
\noindent Again, starting with any strategy $\u'\in\Dd$, we can apply \cref{prop:design2Better} recursively to obtain a more secure strategy $\ad\in\Dd : \u' \securePreceq \ad$, in which, if a node is secure, then all nodes in lower levels are also secured by the defender, i.e.,
\begin{equation}\label{eq:design2Better}
\forall\;i,j\in\N,\; (\u_i = 1\text{ and }\level_j > \level_i) \implies \u_j=1. 
\end{equation} 
Thus, \cref{prop:design2Better} is a generalization of \cref{prop:childNodeSecure}. 
\end{remark}

\begin{proposition}
\label{prop:type3secure}
Assume \assmref{assm:throughout}, \assmref{assm:identicalRX}, \assmref{assm:balancedTree}. Let $\u \in \Dd$ be such that $\ad$ satisfies \eqref{eq:design2Better}. Let $\level' = \argmin_{(\exists\ a\in \N_{\level}:\u_a = 1)}\level$. If the  secure nodes on level $\level'$ are uniformly distributed over the level $\level'$, i.e., 
$\abs{\child_j\cap\Ns} \in \{T, T+1\}, \;\forall\;j \in \N_{\level'}, $ 
where $T \in \Z_+$, then $\u$ is an optimal security strategy, i.e., $\forall\;\adnew \in \Dd,\; \adnew \securePreceq \u$. 
\end{proposition} 
\begin{proof}[\textbf{Proof of \cref{prop:type3secure}}] 
\ifoldProofPropTypeThreeSecure 
Consider $i \in \Lambda_{i'}, j\in\Lambda_{i''}$ such that $j\in\N_i$. Now, if $k\in\N$ such that $\aa_k = 1$, then $k\in \Lambda_{i''}$ or $k \not\in \Lambda_{i'}\cup \Lambda_{i''}$. If $k\in \Lambda_{i''}$, then $\Delta_k(\nu_j) > \Delta_k(\nu_i)$ by \cref{prop:downstream}. Again, by \cref{prop:downstream}, if $k \not\in \Lambda_{i'}\cup \Lambda_{i''}$ then $\Delta_k(\nu_j) = \Delta_k(\nu_i)$. Thus, $\Delta_{\aa}(\nu_j) \ge \Delta_{\aa}(\nu_i)$. So, in any attack $\Lambda_{i''}$ will be more impacted by $\Lambda_{i'}$. So, securing $b$ If  This is because compromise of any PV outside $\Lambda_{i'}\cup \Lambda_{i''}$ will have equal impact on $\nu_i$ and $\nu_j$. And, compromise of

Now, suppose, in a worst-case attack, some node~$j \in \Lambda_{i'}\cup \Lambda_{i''}$ is targeted. Let $p$ denote the parent of node $b$. Let $\Lambda_m^j$ denote the subtree rooted at $m$ until the depth of the node~$j$. 

\begin{claim}\label{clm:atLeastChoice}$k \in \Nv^1 \quad\forall\quad k\in \Lambda_{i''}^{p}$
\end{claim}
Suppose, \cref{clm:atLeastChoice} is false. Then $\exists k\in \Lambda_{i''}^{p}$ such that $k\in \Nv^1$. But, clearly, $\abs{\P_k\cap\P_b} > \abs{\P_a\cap\P_b}$, violating the maximality assumption on choice of nodes $a$ and $b$. 

\begin{claim}\label{clm:atMostChoice}$
k \in \Ns^1 \quad\forall\quad k\in \Lambda_{i'}, k\notin\Lambda_{i'}^a, k \ne a$
\end{claim}
Again, by contradiction, we can show that if $\exists k\in \Lambda_{i'}, k\notin\Lambda_{i'}^a$ such that $k\in \Nv^1$, then choosing $b = k$ will violate the maximality of $ \abs{\P_a\cap\P_b}$. 

Let $n' = \abs{\Lambda_{i'}^a} = \Lambda_{i''}^p$. By \cref{clm:atLeastChoice}, at least $n'+1$ nodes are vulnerable in the cluster $\Lambda_{i''}$, whereas by claim~\ref{clm:atMostChoice}, at most $n'-1$ nodes are vulnerable in the cluster $\Lambda_{i'}$. Also, any node in $\Lambda_{i'}$, that is vulnerable, has a homomorphic counterpart in $\Lambda_{i''}$. Hence, if in a worst-case attack, the node with least voltage lies in $\Lambda_{i'}\cup\Lambda_{i''}$, it will be more lucrative for the attacker to target a node in $\Lambda_{i''}$ instead of $ \Lambda_{i'}$. Since, $\Delta_a(\nu_j) < \Delta_b(\nu_j)\quad\forall\quad j\in\N_0,j\in\Lambda_{i}$. Hence, it is better to secure $b$ instead of $a$. 
\else 
Similar to the proof of \cref{prop:design2Better}. 
\fi 
\end{proof}


\begin{remark}

\noindent \cref{prop:type3secure} implies that there exists an optimal security strategy in which there is a top-most level with DER nodes that are uniformly chosen for security investment, while all the lower levels are fully secure. 

\Cref{prop:childNodeSecure,prop:design2Better} capture the attacker preference for the downstream DERs, whereas  \cref{prop:type3secure} capture the attacker preference for cluster attacks. Hence, the optimal security strategy has distributed secured nodes. 

\end{remark}

\begin{proof}[\textbf{Proof of \Cref{thm:optimalSecureDesign}}]
Let $\adstaronelpf \in \Dd$ be any optimal security strategy. From $\adstaronelpf$, by sequentially applying \cref{prop:childNodeSecure}, \cref{prop:design2Better}, and \cref{prop:type3secure}, we can obtain an optimal security strategy $\adstartwolpf$ that satisfies \eqref{eq:childNodeSecure}, \eqref{eq:design2Better}, and has the top-most level with secure nodes having uniformly distributed secured  nodes. 

Now, let $\adstarlpf$ be the output of \cref{algo:securityAlgorithm}. Since in \cref{algo:securityAlgorithm}, nodes are secured from the leaf nodes to the root node level-by-level, $\adstarlpf$ also satisfies \eqref{eq:childNodeSecure} and \eqref{eq:design2Better}. The  \cref{algo:securityAlgorithm} also secures the top-most level with secure nodes with uniformly distributed secured nodes, $\adstarlpf$ is the same as $\adstartwolpf$  upto a homomorphic transformation. 

Finally, we argue that under \assmref{assm:throughout}-\assm{\ref{assm:balancedTree}}, $\adstarlpf$ can be combined with previous results to obtain full solution of $\dadlpf$. Under   \assmref{assm:identicalRX}, the defender set-points are fixed. Since, $\ulpf$ and $\sgsetdstarlpf$ are both fixed, we can compute the set of candidate optimal attack vectors $\Dastarlpf$, by considering only vulnerable DERs. Then for a fixed $\aa\in\Dastarlpf$, the sub-problem $\adsplpf$ reduces to an LP in $\gamma$. Hence,   \Cref{algo:greedyAlgorithm2} solves for $(\psistarlpf,\phistarlpf)$, the optimal solution of $\adlpf$ for $\ad = \ulpf$, by iterating over $\aa \in \Dastarlpf$. The strategy profile $(\adstarlpf, \psistarlpf, \phistarlpf)$, thus obtained, is an optimal solution to for DNs that satisfy  \assmref{assm:throughout}, \assmref{assm:identicalRX}, \assmref{assm:balancedTree}. Similarly, we can solve $\dadupf$. 
\end{proof}

\begin{remark}
We revisit the security strategies $\adone$ and $\adtwo$ in \cref{fig:designs}: which one is better? Firstly, we use symmetricity \assmref{assm:balancedTree} to argue that securing nodes 2, 4, 5 is equivalent to securing nodes 3, 6, 7. Then,  $\Lambda_3$ subtree of $\adtwo$ has more distributed secured nodes than $\Lambda_2$ in $\adone$. Hence, strategy 2 is better. \Cref{thm:optimalSecureDesign} will, of course, give the optimal security strategy $\adstarlpf$ in which nodes $\Ns(\adstarlpf) = \{8,9,10,12,13,14\}$, or other homomorphic strategies of $\adstarlpf$. 

\end{remark}

\fi 

\end{document}